\documentclass[11pt,reqno]{amsart}
\usepackage{amsmath,amssymb,latexsym,esint,cite,mathrsfs}
\usepackage{verbatim,wasysym}
\usepackage[left=1.8cm,right=1.8cm,top=2.4cm,bottom=2.4cm]{geometry}
\usepackage{tikz,enumitem,graphicx, subfig, microtype, color}
\usepackage{epic,eepic}

\usepackage[colorlinks=true,urlcolor=blue, citecolor=red,linkcolor=blue,
linktocpage,pdfpagelabels, bookmarksnumbered,bookmarksopen]{hyperref}
\usepackage[hyperpageref]{backref}
\usepackage[english]{babel}
\usepackage{appendix}

\numberwithin{equation}{section}

\newtheorem{thm}{Theorem}[section]
\newtheorem{lem}[thm]{Lemma}

\newtheorem{Prop}[thm]{Proposition}

\newtheorem{Rem}[thm]{Remark}

\begin{document}
	\title[The nonlinear Hartree equation]
	{Asymptotic behavior of least energy solutions to the nonlinear Hartree equation near critical exponent}
	
	\author[S. Cingolani]{Silvia Cingolani}
	\address{\noindent Silvia Cingolani \newline
Dipartimento di Matematica, Universit\`{a} degli Studi di Bari Aldo Moro,\newline Via Orabona 4, 70125 Bari, Italy.}
	\email{silvia.cingolani@uniba.it}
	
	\author[M. Yang]{Minbo Yang$^\dag$}
	\address{\noindent Minbo Yang  \newline
		School of Mathematical Sciences, Zhejiang Normal University,\newline
		Jinhua 321004, Zhejiang, People's Republic of China.}\email{mbyang@zjnu.edu.cn}
	
	\author[S. Zhao]{Shunneng Zhao$^\ddag$}
	\address{\noindent Shunneng Zhao  \newline
Dipartimento di Matematica, Universit\`{a} degli Studi di Bari Aldo Moro,\newline Via Orabona 4, 70125 Bari, Italy.
		\vspace{2mm}
		\newline
School of Mathematical Sciences, Zhejiang Normal University,\newline
		Jinhua 321004, Zhejiang, People's Republic of China.}
	\email{snzhao@zjnu.edu.cn}

	
	
\thanks{Silvia Cingolani was supported by PNRR MUR project PE0000023 NQSTI - National Quantum Science and Technology Institute (CUP H93C22000670006), PRIN PNRR P2022YFAJH ``Linear and Nonlinear PDEs: New directions and applications," and partially supported by INdAM-GNAMPA.}

	\thanks{$^\dag$Minbo Yang was partially supported by the National Key Research and Development Program of China (No. 2022YFA1005700) and National Natural Science Foundation of China (12471114).}
	\thanks{$^\ddag$Shunneng Zhao was supported by PNRR MUR project PE0000023 NQSTI - National Quantum Science and Technology Institute (CUP H93C22000670006),
National Natural Science Foundation of China (12401146, 12261107) and Natural Science Foundation of Zhejiang Province (LMS25A010007).}
	
	\allowdisplaybreaks
	
	\begin{abstract}
		{\small
			In this paper, we study that the nearly critical nonlocal problem
\begin{equation*}
\left\lbrace
\begin{aligned}
    &-\Delta u=(|x|^{-{(n-2)}}\ast u^{p-\epsilon})u^{p-1-\epsilon} \quad \mbox{in}\quad \Omega,\\
    &u>0\quad \mbox{in}\quad\hspace{1mm} \Omega,\\
    &u=0\quad \mbox{on}\hspace{2.5mm}\partial\Omega,
   \end{aligned}
\right.
\end{equation*}
where $\Omega$ is a smooth bounded domain in $\mathbb{R}^n$ for $n=3,4,5$, $\ast$ denotes the standard convolution, $\epsilon>0$ is a small parameter and $p=\frac{n+2}{n-2}$ is energy-critical exponent.
 We study the asymptotic behavior of least energy solutions as $\epsilon\rightarrow0$. These solutions are shown to blow-up at exactly one point $x_0$ and location of this point is characterized. In addition, the shape and exact rates for blowing-up are studied. Finally, in order to further locate the blowing-up point $x_0$, we prove that $x_0$ is a global maximum point of the Robin's function of $\Omega$.
		}
	\end{abstract}
	
	\vspace{3mm}
	
	\maketitle
	\section{Introduction}
\subsection{Motivation and main results}
In this paper, we are concerned with the following nonlocal equation:
\begin{equation}\label{ele-1.1}
\left\lbrace
\begin{aligned}
    &-\Delta u=(|x|^{-{\mu}}\ast u^{p-\epsilon})u^{p-1-\epsilon} \quad \mbox{in}\quad \Omega,\\
    &u>0\quad \mbox{in}\quad\hspace{1mm} \Omega,\\
    &u=0\quad \mbox{on}\hspace{2.5mm}\partial\Omega,
   \end{aligned}
\right.
\end{equation}
where $\Omega$ is a smooth bounded domain in $\mathbb{R}^n$ for $n\geq3$, $\mu\in(0,n)$, $\epsilon>0$ is a small parameter, the exponent $p=\frac{2n-\mu}{n-2}$ is a threshold on the existence of a solution to \eqref{ele-1.1}. If $\epsilon>0$, one
can find a solution to \eqref{ele-1.1} by applying the standard variational argument with the compact embedding $H^{1}(\Omega)\hookrightarrow L^{\frac{2n(p-\epsilon)}{2n-\mu}}(\Omega)$. If $\epsilon=0$ and $\Omega$ is star-shaped, then an application of the nonlocal-type Pohozaev identity in \cite{GY18} gives nonexistence of a nontrivial solution for \eqref{ele-1.1}.
The problem \eqref{ele-1.1} can be understood as the nonlocal version of the following problem with nonlinearities of slightly subcritical growth
 \begin{equation}\label{SC}
 -\Delta u
=u^{\frac{n+2}{n-2}-\varepsilon}
 ,~~~u>0,~~~\text{in}~~\Omega,~~~
 		u=0,~~~\text{on}~~\partial\Omega.
 \end{equation}
Atkinson and Peletier \cite{ATKINSON-1986} studied the asymptotic behavior of radial solutions of \eqref{SC} by using ODE technique in the unit ball of $\mathbb{R}^3$. Later, Brezis and Peletier \cite{BP} used the method of PDE to obtain the same results as that in \cite{ATKINSON-1986} for the spherical domains. Moreover, they conjectured that the same kind of results hold for non
spherical domains. This question was answered affirmatively in the general case by Han \cite{HANZCHAO} (independently by Rey \cite{Rey-1989}),
in which they independently proved that if $\epsilon$ small enough, $u_{\epsilon}$ blows-up at the unique point $x_0$ that is a critical point of the Robin
function of $\Omega$, and analysis of the shape and exact rates for blowing up. Moreover, for the classical local Brezis-Nirenberg problem,  Rey \cite{Rey-1990} constructed a family of solutions for this problem which asymptotically blow up at a non-degenerate critical point of the Robin function. After these seminal works, numerous results of a similar nature appeared in the literature, and the following represents only a extension of them: for the
number of blow-up point $k>1$ \cite{Musso-Pistoia-2002,B-L-R}, the related eigenvalue problem \cite{CKIM,GP05}, the uniqueness of positive solutions \cite{dgp,GniN,LiWZ}, the fractional Laplacian \cite{CKL}, and references therein. See \cite{Djadli, Malchiodi} for the existence of positive solutions to the conformal scalar curvature equation by applying the perturbation method.

In the spirit of Rey \cite{Rey-1990}, Chen and Wang \cite{cw}
proved that if a solution $u_{\epsilon}$ of \eqref{ele-1.1} satisfies
\begin{equation*}
|\nabla u_\varepsilon|^2\rightarrow C_{HLS}^{\frac{2n-\mu}{n+2-\mu}}\delta(x-x_0)~~~\text{as}~\varepsilon\rightarrow0
\end{equation*}
for $n\geq3$ and $\mu\in(0, \min\{n,4\})$, where $C_{HLS}>0$ is a constant depending on the dimension $n$ and parameters $\mu$ (see \eqref{Prm}) and $\delta(x)$ denotes the Dirac measure at the origin,
then the concentrate point $x_{0}\in \Omega$ is a critical point of Robin function $\phi$ (see below). In addition, they also used the finite dimensional reduction method to give a kind of converse result, i.e., for $\epsilon$ sufficiently small, \eqref{ele-1.1} has a family of solutions $u_{\epsilon}$ concentrating around $x_0$ under the restriction of some dimensions.

In this line of research, motivated by the works of Han \cite{HANZCHAO} and Rey \cite{Rey-1989} on the classical local problem, the aim of this paper is to study the asymptotic behavior of least energy solutions $u_{\epsilon}$ for nonlocal problem \eqref{ele-1.1} when $p=\frac{n+2}{n-2}>2$ is the $\dot{H}^{1}$-energy-critical exponent and $\epsilon>0$ is close to zero.

To state the result,
we recall from \cite{H-L-1928,S1963} that
the classical Hardy-Littlewood-Sobolev inequality
    \begin{equation}\label{hlsi}
    \int_{\mathbb{R}^n}\int_{\mathbb{R}^n}f(x)|x-y|^{-\mu} g(y)dxdy\leq C(n,r,t,\mu)\|f\|_{L^r(\mathbb{R}^n)}\|g\|_{L^t(\mathbb{R}^n)}
    \end{equation}
with $\mu\in(0,n)$, $1<r,t<\infty$ and $\frac{1}{r}+\frac{1}{t}+\frac{\mu}{n}=2$. Moreover, in the general diagonal case $t=r=\frac{2n}{2n-\mu}$, Lieb in \cite{Lieb83} classified the extremal function of HLS inequality with sharp constant by rearrangement and symmetrisation, and obtained the best constant
    \begin{equation}\label{defhlsbc}
    C_{n,r,t,\mu}:=C_{n,\mu}=\frac{\Gamma((n-\mu)/2)\pi^{\mu/2}}{\Gamma(n-\mu/2)}\left(\frac{\Gamma(n)}{\Gamma(n/2)}\right)^{1-\frac{\mu}{n}},
    \end{equation}
   and the equality holds if and only if
    \begin{equation*}
   f(x)=cg(x)=a\Big(\frac{1}{1+\lambda^2|x-x_0|^2}\Big)^{\frac{2n-\mu}{2}}
    \end{equation*}
    for some $a\in \mathbb{C}$, $\lambda\in \mathbb{R}\backslash\{0\}$ and $x_0\in \mathbb{R}^n$.
The classical Hardy-Littlewood-Sobolev inequality and Sobolev inequality tell us that
\begin{equation}\label{Prm}
C_{HLS}\left(\int_{\mathbb{R}^n}(|x|^{-\mu} \ast|u|^{p})|u|^{p} dx\right)^{\frac{1}{p}}\leq\int_{\mathbb{R}^n}|\nabla u|^2 dx, \quad u\in \dot{H}^{1}(\mathbb{R}^n),
\end{equation}
for some positive constant $C_{HLS}$ depending only on $n$ and $\mu$, where $n\geq 3$, $\mu\in(0,n)$ and $\dot{H}^{1}(\mathbb{R}^n):=\mathcal{D}^{1,2}(\mathbb{R}^n)$, the completion of $C_{0}^{\infty}(\mathbb{R}^n)$ under the norm $\|\nabla u\|_{L^2(\mathbb{R}^n)}$.
It is well-known that the Euler-Lagrange equation of \eqref{Prm}, up to scaling, is given by
 \begin{equation}\label{ele}
 -\Delta u=(|x|^{-\mu}\ast |u|^{p})|u|^{p-2}u \quad \mbox{in}\quad \mathbb{R}^n.
 \end{equation}
Furthermore, the authors in \cite{DAIQIN, DY19,GHPS19} independently classified all positive solutions of \eqref{ele} are functions of the form
\begin{equation}\label{defU}
W[\xi,\lambda](x)=\tilde{c}_{n,\mu}\big(\frac{\lambda}{1+\lambda^2|x-\xi|^2}\big)^{\frac{n-2}{2}},\hspace{1mm}\lambda\in\mathbb{R}^{+},\hspace{1mm}\xi\in\mathbb{R}^n,
\end{equation}
and obtained the optimal constant in \eqref{Prm}
    \begin{equation*}
C_{HLS}=S\left[\frac{\Gamma((n-\mu)/2)\pi^{\mu/2}}{\Gamma(n-\mu/2)}\left(\frac{\Gamma(n)}{\Gamma(n/2)}\right)^{1-\frac{\mu}{n}}\right]^{(2-n)/(2n-\mu)}.
    \end{equation*}
Here the constant $\tilde{c}_{n,\mu}$ is given by
\begin{equation}\label{fU}
    \tilde{c}_{n,\mu}:=[n(n-2)]^{\frac{n-2}{4}}S^{\frac{(n-\mu)(2-n)}{4(n-\mu+2)}}C_{n,\mu}^{\frac{2-n}{2(n-\mu+2)}},
      \end{equation}
$S$ is the best Sobolev constant and $C_{n,\mu}$ is defined in \eqref{defhlsbc}.

Our main result is in the following:
    \begin{thm}\label{Figalli}
 Assume that $n=3,4,5$, $p=\frac{n+2}{n-2}$ and $\epsilon$ is sufficiently small. Let $u_\varepsilon$ be a solution of \eqref{ele-1.1} such that
\begin{equation}\label{minimi}
\frac{\int_{\Omega}|\nabla u_{\epsilon}|^2dx}{\left[\int_{\Omega}(|x|^{-(n-2)} \ast|u_{\epsilon}|^{p-\epsilon})|u_{\epsilon}|^{p-\epsilon} dx\right]^{\frac{1}{p-\epsilon}}}=C_{HLS}+o(1)\quad\mbox{as}\quad\epsilon\rightarrow0.
\end{equation}
 Then
\begin{itemize}
\item[$(a)$]
there exists $x_0\in\Omega$ such that, after passing to a subsequence, we have
\begin{equation*}
u_\epsilon\rightarrow0\in C^1(\Omega\setminus\{x_0\})\quad\mbox{as}\quad\epsilon\rightarrow0
\end{equation*}
and
\begin{equation*}
|\nabla u_{\epsilon}|^2\rightarrow\big[C_{HLS}\big]^{\frac{n+2}{4}}\delta(x-x_0)\quad\mbox{as}\quad\epsilon\rightarrow0
\end{equation*}
in the sense of distributions and where $\delta(x)$ denotes the Dirac measure at the origin.
\item[$(b)$] $x_0$ is a critical of $\phi(x):=H(x,x)$ (Robin's function of $\Omega$) for $x\in\Omega$.
The function $H(x,y)$ is given as follows: for any $y\in\Omega$, $H(x,y)$ satisfies
\begin{equation*}
\Delta H(x,y) =0\quad\mbox{in}\quad\Omega,\quad\quad
	H(x,y)=-\frac{1}{(n-2)\omega_n|x-y|^{n-2}}\quad\mbox{on}\quad\partial\Omega.
\end{equation*}
The function $H$ is nothing but the regular part of the Green function. Indeed, if $G(x,y)$ denotes the
Green's function of the Laplacian on $H_{0}^1({\Omega})$, then we have
   \begin{equation}\label{Robin}
	H(x,y)=G(x,y)-\frac{1}{(n-2)\omega_n|x-y|^{n-2}},
\end{equation}
    where $\omega_n$ is the area of the unit sphere in $\mathbb{R}^n$.
\end{itemize}
    \end{thm}

The next two theorems make more precise the behavior of solutions. First we give the rate of blow-up of the maximum of the solutions.
 \begin{thm}\label{Figalli2}
  Let the assumptions of Theorem \ref{Figalli} be satisfied. Then
\begin{equation*}
\lim\limits_{\epsilon\rightarrow0}\epsilon\|u_\epsilon\|_{L^{\infty}(\Omega)}^2=\widetilde{\alpha}_{n}^2\frac{n+2}{n-2}\big[\frac{1}{C_{HLS}}\big]^{\frac{n+2}{4}}\big\|W\big\|^{2(2^{\ast}-1)}_{L^{2^{\ast}-1}(\mathbb{R}^n)}|\phi(x_0)|
\end{equation*}
with
$$\widetilde{\alpha}_{n}=\left[\frac{n(n-2)}{\sqrt{S}}\right]^{\frac{n-2}{4}} \frac{\pi^{\frac{n}{2}}\Gamma(1)}{\Gamma(\frac{n+2}{2})}\left\{\frac{\Gamma(1)\pi^{(n-2)/2}}{\Gamma((n+2)/2)}\left(\frac{\Gamma(n)}{\Gamma(n/2)}\right)^{2/n}\right\}^{\frac{2-n}{8}}, $$
 where $2^{\ast}:=\frac{2n}{n-2}$, $W:=W[0,1](x)$, $\Gamma(s)=\int_0^{+\infty} x^{s-1}e^{-x}\,dx$ and $S$ is the best Sobolev constant in $\mathbb{R}^n$. Moreover, for any $x\in\Omega\setminus\{x_0\}$, it holds that
\begin{equation*}
\frac{u_\epsilon(x)}{\sqrt{\epsilon}}\rightarrow\sqrt{\frac{\tilde{\alpha}_n(n-2)}{n+2}}\big[C_{HLS}\big]^{\frac{n+2}{8}}\frac{G(x,x_0)}{\sqrt{|\phi(x_0)|}}.
\end{equation*}
\end{thm}
Th first result of Theorem \ref{Figalli2} is a consequence of the following the behavior if solutions in terms of the Green's function.
\begin{thm}\label{consequence}
 Let the assumptions of Theorem \ref{Figalli} be satisfied. Then
 $$
\lim\limits_{\epsilon\rightarrow0^{+}}\|u_{\epsilon}\|_{L^{\infty}(\Omega)}u_{\epsilon}(x)=\widetilde{\alpha}_{n}\big\|W\big\|^{2^{\ast}-1}_{L^{2^{\ast}-1}(\mathbb{R}^n)}G(x,x_{0}).
$$
where $\widetilde{\alpha}_{n}$ is defined in Theorem \ref{Figalli2} and where the convergence is in $C^{1,\alpha}(\omega)$ with $\omega$ any subdomain of $\Omega$ not containing $x_0$.
\end{thm}
As in \cite{HANZCHAO}, we
note that author used the moving plane method as a key tool in the proof of uniform boundedness of $u_{\epsilon}$ near $\partial\Omega$, which requires the convexity of the domain. In present paper, we shall remove the convexity assumption and make use of the local Pohozaev-type identity and sharp point-wise estimates of $u_\epsilon$ to obtain that $x_{\epsilon}$ stays away from the boundary. We also observe that one of the key points in the proof of our main results is to find decay estimate of rescaled least energy solutions for Choquard equation \eqref{ele-1.1} by combining the Kelvin transform. We refer the readers to \cite{Cassani, Cingolani-1, Cingolani-2, V-Moroz, Moroz-2} for some relevant studies on nonlinear Choquard equations.

In the following, we hope to further locate the blow-up point $x_0$ and to give a precise asymptotic expansion of the least energy solutions.
More precisely, we shall show that $x_0$ is a global maximum point of $\phi(x)$.
\begin{thm}\label{maximum}
 Let the assumptions of Theorem \ref{Figalli} be satisfied. Then $\phi(x_\epsilon)\rightarrow\max_{x\in\Omega}\phi(x)$ as $\epsilon\rightarrow0$.
\end{thm}

In what follows, we consider the minimization problem
\begin{equation}\label{SKLS}
S_{HL}^{\epsilon}=\inf\Big\{\frac{\int_{\Omega}|\nabla u|^2dx}{\left[\int_{\Omega}(|x|^{-(n-2)} \ast|u|^{p-\epsilon})|u|^{p-\epsilon} dx\right]^{\frac{1}{p-\epsilon}}}:~u\in H^{1,2}_{0}(\Omega),~u\not\equiv0\Big\}.
\end{equation}
In fact, we use the same framework and methods as in \cite{GW,JW0,HL}.
Choosing $x_\epsilon\in\Omega$ and the number $\lambda_{\epsilon}>0$ by
\begin{equation}\label{miu}
\lambda_{\epsilon}^{\frac{2(n-2)}{4-(n-2)\epsilon}}=\|u_{\epsilon}\|_{L^{\infty}(\Omega)}=u_\epsilon(x_\epsilon).
\end{equation}
We define a family of rescaled functions
$$v_{\epsilon}(x)=\lambda_{\epsilon}^{-\frac{2(n-2)}{4-(n-2)\epsilon}}u_{\epsilon}(\lambda_{\epsilon}^{-1}x+x_{\epsilon})\quad\mbox{for}\quad x\in\Omega_{\epsilon}:=\lambda_{\epsilon}(\Omega-x_{\epsilon}).$$
Then we find
(see \eqref{rescaled}),
\begin{equation*}
\begin{cases}
-\Delta v_{\epsilon}(x)=\big(|x|^{-{(n-2)}}\ast v_\epsilon^{p-\epsilon}\big)v_\epsilon^{p-1-\epsilon}\hspace{6mm}\mbox{in}\hspace{2mm} \Omega_{\epsilon},\\
0\leq v_{\varepsilon}(x)\leq1~~~~~~~~~~~~~~~~~~~~~~~~~~~~~~~~~~~~~~~~~~~~~\text{in}~~\Omega_{\varepsilon},
\\
v_{\varepsilon}(x)=0~~~~~~~~~~~~~~~~~~~~~~~~~~~~~~~~~~~~~~~~~~~~~~~~~~~~\text{on}~~\partial\Omega_{\varepsilon},\\
v_{\varepsilon}(0)=\max\limits_{x\in\Omega_{\varepsilon}}v_{\varepsilon}(x)=1.
\end{cases}
\end{equation*}
It is noticing that the authors of \cite{DAIQIN, DY19,GHPS19} independently classified the extremal function of \eqref{Prm} are the bubbles $W[\xi,\lambda]$ in \eqref{defU} and which is all positive solution to \eqref{ele-1.1}, and combined with $\|v_{\epsilon}\|_{L^{\infty}(\Omega)}$ is uniformly bounded by some constant $M$, by elliptic regularity, we have that
$$\|v_{\epsilon}\|_{C^{2+\alpha}(\bar{\Omega})}\leq M\hspace{2mm}\mbox{with}\hspace{2mm}\alpha\in(0,1).$$
Therefore, it is not hard to see that
$$v_\epsilon\rightarrow W[0,1](x)\quad\mbox{in}\quad C_{loc}^2(\mathbb{R}^n)$$
by combining the elliptic interior estimates and where $W[\xi,\lambda](x)$ are the only
positive solutions of the equation \eqref{ele}.
We establish the following asymptotic expansion of $u_{\epsilon}$.
\begin{thm}\label{asymptotic}
Let the assumptions of Theorem \ref{Figalli} be satisfied. Then
$$
v_{\epsilon}(y)
=W[0,1](y)+\lambda_{\varepsilon}^{-(n-2)}
\tilde{H}(x_{\epsilon},\lambda_{\epsilon}^{-1}y+x_{\epsilon})+\lambda_{\varepsilon}^{-(n-2)}\phi_0(y)+o(\lambda_{\varepsilon}^{-(n-2)})\quad{as}\quad\epsilon\rightarrow0,$$
where $\tilde{H}(x,y)=-(n-2)\omega_nH(x,y)$, $o(1)$ is uniform in $B(0,M\lambda_{\epsilon})$ with $M>0$ depending only on $\Omega$, and $\phi_0\in W^{2,q^{\prime}}(\mathbb{R}^n)$ for $q^{\prime}>n$ is the unique bound solution of
\begin{equation*}
\begin{split}
-\Delta\phi_0=&p
\Big(|y|^{-(n-2)}\ast W^{p-1}\phi_0\Big)
W^{p-1}+(p-1)\Big(|y|^{-(n-2)}\ast W^{p}\Big)
W^{p-2}\phi_{0}\\&-\kappa(n,x_0)\big(|y|^{-(n-2)}\ast W^{p}\big)W^{p-1}\log W-\kappa(n,x_0)\big(|y|^{-(n-2)}\ast W^{p}\big)W^{p-1}\log W\\&-(p-1)\tilde{c}_{n,\mu}\big(|y|^{-(n-2)}\ast W^p\big)W^{p-2}\tilde{H}(x_0,x_0)\\&
-p\tilde{c}_{n,\mu}\big(|y|^{-(n-2)}\ast W^{p-1}\tilde{H}(x_0,x_0)\big)W^{p-1}\quad\mbox{in}\quad\mathbb{R}^n,\hspace{2mm}\mbox{with}\hspace{2mm} W(x)=W[0,1](x).
\end{split}
\end{equation*}
\end{thm}
\begin{Rem}
The proof of Theorems \ref{Figalli}-\ref{asymptotic} adapt the methods inspired by \cite{HANZCHAO,GW,JW0,HL}. By using the decay estimate of rescaled least energy solutions, local Pohozaev identity and an asymptotic expansion of ground energy $S_{HL}^{\epsilon}$, one could extend the results to dimensions $n=3, 4, 5$ and we consider the Newtonian potential instead of the Riesz potential because we need to use the Kelvin transform. Therefore, in the following parameter region:
\begin{equation}\label{paramater}
\mu=n-2\hspace{2mm}\mbox{and}\hspace{2mm}n\geq6, \hspace{2mm}\mbox{or}\hspace{2mm}n\geq3\hspace{2mm}\mbox{and}\hspace{2mm}n-2\neq\mu \in(0,n)\hspace{2mm}\mbox{with}\hspace{2mm}(0,4],
\end{equation}
these methods do not seem to lead to any improvement of Theorems \ref{Figalli}-\ref{asymptotic}. Taking this as inspiration allows us to ask an open question: whether the asymptotic behavior of the solution of problem \eqref{ele-1.1} as $\epsilon\rightarrow0$ in the results of this paper continue to hold in the case $n$ and $\mu$ satisfying \eqref{paramater}.
\end{Rem}

On the other hand, we also consider the nonlocal problem
\begin{equation}\label{ele-2}
    -\Delta u_\epsilon=(|x|^{-\mu}\ast u_\epsilon^{p_{\mu}})u_\epsilon^{p_{\mu}}+\epsilon u_\epsilon, \quad u_\epsilon>0\quad\mbox{in}\quad \Omega, \quad u_\epsilon=0\quad\mbox{on}\quad \partial\Omega
    \end{equation}
for $n\geq3$, $p_{\mu}=\frac{2n-\mu}{n-2}$, $\mu\in(0,n)$ with $\mu\in(0,4]$ and $\epsilon>0$ is a small parameter and $\Omega$ is a smooth bounded domain of $\mathbb{R}^n$. The existence if positive solutions for this problem can be found in \cite{GY18} and related topics, we refer the readers to \cite{Ghimenti} and reference therein.
They also proved that if $N\geq3$ and $\varepsilon=0$, \eqref{ele-2} admits no solutions when $\Omega$ is star-shaped.
Similarly to the problem \eqref{ele-1.1}, we can define the least enengy solutions for \eqref{ele-2}. The authors in \cite{YZ} proved that a result for this nonlocal
problem to locate the blow up point $x_0$ (that is $\nabla\phi(x_0)=0$ and $x_0\in \Omega$) by employing the finite dimensional reduction method. If the blow-up points satisfy a certain non-degeneracy condition, \cite{yyz} provided a kind of converse result for Theorem 1.1 in \cite{YZ}. However, the problem of the asymptotic character of $\{u_{\epsilon}\}_{\epsilon>0}$ and exact rates of blowing-up for problem \eqref{ele-2} has not been investigated so far. We shall address this gap by proving the following result.
\begin{thm}\label{Brezi-type}
Assume that $n\geq3$, $p_{\mu}=\frac{2n-\mu}{n-2}$, $\mu\in(0,n)$ with $\mu\in(0,4]$ and $\epsilon$ is sufficiently small. Let $u_{\epsilon}$ be a least energy solution of \eqref{ele-2}. Then the conclusions of Theorem \ref{Figalli} hold, and for $n\geq4$ we have
$$
\lim\limits_{\epsilon\rightarrow0^{+}}\epsilon\big\|u_{\epsilon}\big\|^{\frac{2(n-4)}{n-2}}_{L^{\infty}(\Omega)}=\frac{(n-2)\widetilde{\alpha}_{n,\mu}^2}{2\omega_n\cdot \sigma_n}\big\|W\big\|^{2(2^{\ast}-1)}_{L^{2^{\ast}-1}(\mathbb{R}^n)}
 |\phi(x_0)|
\quad\mbox{if}\quad n>4,
$$
and
$$
\lim\limits_{\epsilon\rightarrow0^{+}}\epsilon\log\big\|u_{\epsilon}\big\|_{L^{\infty}(\Omega)}=\frac{\widetilde{\alpha}_{n,\mu}^2}{\omega_n}\big\|W\big\|^{2(2^{\ast}-1)}_{L^{2^{\ast}-1}(\mathbb{R}^n)} |\phi(x_0)|\quad\hspace{2mm}\mbox{if}\quad n=4,
$$
where $\omega_n=\frac{2\pi^{\frac{n}{2}}}{\Gamma(n/2)}$, $\sigma_n=\int_{0}^{\infty}\frac{r^{n-1}}{(1+r^2)^{n-2}}dr$ and
$$\widetilde{\alpha}_{n,\mu}=\left[n(n-2)\right]^{\frac{n-2}{4}}S^{\frac{(2-n)(n-\mu)}{4(n-\mu+2)}}\frac{\pi^{\frac{n}{2}}\Gamma(\frac{n-\mu}{2})}{\Gamma(\frac{2n-\mu}{2})}\left\{\frac{\Gamma((n-\mu)/2)\pi^{\mu/2}}{\Gamma(n-\mu/2)}\left(\frac{\Gamma(n)}{\Gamma(n/2)}\right)^{1-\frac{\mu}{n}}\right\}^{\frac{2-n}{2(n-\mu+2)}}.$$
Moreover, for any $x\in\Omega\setminus\{x_0\}$, we have
$$
\lim\limits_{\epsilon\rightarrow0^{+}}\|u_{\epsilon}\|_{L^{\infty}(\Omega)}u_{\epsilon}(x)=\widetilde{\alpha}_{n}\big\|W\big\|^{2^{\ast}-1}_{L^{2^{\ast}-1}(\mathbb{R}^n)}G(x,x_{0}),
$$
where the convergence is in $C^{1,\alpha}(\omega)$ with $\omega$ any subdomain of $\Omega$ not containing $x_0$.
\end{thm}

\subsection{Structure of the paper.}
	The paper is organized as follows. In section \ref{sec:sobolev}, we establish a decay estimate of rescaled least energy solutions $v_{\epsilon}$. Section \ref{sec:weaksolution} is devoted to proof uniform bounded of least energy solutions $u_\epsilon$ near $\partial\Omega$ which makes use of a local Pohozaev-type identity and sharp pointwise estimates of $u_\epsilon$, and here not applying the moving plane method. Based on these results, we then study the blow-up points and behavior of blow-up rates by exploiting the Green-type identity and complete the proof of Theorem \ref{Figalli}, Theorem \ref{Figalli2} and Theorem \ref{consequence} in section \ref{consequence00}. In section \ref{theorem1-7}, we give the proof of Theorem \ref{Brezi-type}. In section \ref{maximumpoint}, we obtain an upper bound for $S_{HL}^{\epsilon}$ by taking a good trial function and we further find a lower bound for $S_{HL}^{\epsilon}$ by combining Proposition \ref{pro}. These results are used in next subsection to prove Theorems \ref{maximum}-\ref{asymptotic}. Finally, in section \ref{pro-1}, we conclude that the proof of Proposition \ref{pro}. In the appendices we collect a series of technical lemmata which is required in the previous sections.

\subsection{Notations.}
Throughout this paper, $C$ and $c$ are indiscriminately used to denote various absolutely positive constants. $a\approx b$ means that $a\lesssim b$ and $\gtrsim b$, and we will use big $O$ and small $o$ notations to describe the limit behavior of a certain quantity as $\epsilon\rightarrow0$.

\section{Decay estimate of rescaled least energy solutions}\label{sec:sobolev}
In this section, we are devoted to prove that an adequate decay estimate for least energy solutions to \eqref{ele-1.1}.
Since $u_\epsilon$ become unbounded as $\epsilon\rightarrow0$ then we choose $x_\epsilon\in\Omega$ and the number $\lambda_{\epsilon}>0$ by
\begin{equation}\label{miu}
\lambda_{\epsilon}^{\frac{2(n-2)}{4-(n-2)\epsilon}}=\|u_{\epsilon}\|_{L^{\infty}(\Omega)}=u_\epsilon(x_\epsilon).
\end{equation}
We define a family of rescaled functions
$$v_{\epsilon}(x)=\lambda_{\epsilon}^{-\frac{2(n-2)}{4-(n-2)\epsilon}}u_{\epsilon}(\lambda_{\epsilon}^{-1}x+x_{\epsilon})\quad\mbox{for}\quad x\in\Omega_{\epsilon}:=\lambda_{\epsilon}(\Omega-x_{\epsilon}).$$
Then we have
\begin{equation}\label{rescaled}
-\Delta v_{\epsilon}(x)=\lambda_{\epsilon}^{-(\frac{2(n-2)}{4-(n-2)\epsilon}+2)}(-\Delta u_{\epsilon})\big(\lambda_{\epsilon}^{-1}x+x_{\epsilon}\big)=\big(|x|^{-{(n-2)}}\ast v_\epsilon^{p-\epsilon}\big)v_\epsilon^{p-1-\epsilon}\hspace{3mm}\mbox{in}\hspace{2mm} x\in\Omega_{\epsilon}.
\end{equation}
Now the Kelvin transform $w_{\epsilon}$ of $v_{\epsilon}$ defined by
$$w_{\epsilon}(x)=\frac{1}{|x|^{n-2}}v_{\epsilon}\big(\frac{x}{|x|^2}\big)\quad\mbox{in}\quad\Omega_{\epsilon}^{\ast}:=\big\{x:\frac{x}{|x|^{2}}\in\Omega_{\epsilon}\big\}$$
satisfies
\begin{equation}\label{Kelvin}
\begin{split}
-\Delta w_{\epsilon}(x)&=\frac{1}{|x|^{n+2}}(-\Delta v_{\epsilon})\big(\frac{x}{|x|^2}\big)\\&
=\frac{1}{|x|^{{\epsilon(n-2)}}}\Big(\int_{\Omega_{\epsilon}^{\ast}}\frac{w_\epsilon^{p-\epsilon}(y)}{|x-y|^{{(n-2)}}|y|^{{\epsilon(n-2)}}} dy\Big)w_\epsilon^{p-1-\epsilon}\hspace{4mm}\mbox{in}\hspace{2mm} x\in\Omega_{\epsilon}^{\ast}.
\end{split}
\end{equation}
\begin{Rem}
It is necessary to point out \eqref{rescaled} is not valid when $\mu\neq n-2$ in \eqref{ele-1.1}.
\end{Rem}

\begin{lem}\label{infinite}
Let $\lambda_{\epsilon}>0$ be the number introduced in \eqref{miu} and
the minimizing sequence $u_\epsilon$ of \eqref{minimi} is such that
\begin{equation*}
\lambda_{\epsilon}=\big\|u_{\epsilon}\big\|^{\frac{4-(n-2)\epsilon }{2(n-2)}}_{L^{\infty}(\Omega)}\rightarrow\infty \quad\mbox{as}\quad\epsilon\rightarrow0.
\end{equation*}
\end{lem}
\begin{proof}
In view of \eqref{ele-1.1} and \eqref{minimi}, we have
\begin{equation}\label{SHL}
\lim\limits_{\epsilon\rightarrow0}\int_{\Omega}(|x|^{-(n-2)} \ast|u_{\epsilon}|^{p-\epsilon})|u_{\epsilon}|^{p-\epsilon} dx=\big(C_{HLS}\big)^{\frac{n+2-\epsilon(n-2)}{4-\epsilon(n-2)}}.
\end{equation}
Now, assume that $\{u_\epsilon\}_{\epsilon>0}$ is uniformly bounded in $\Omega$. Then there exists a function $u_0$, up to subsequence, $u_\epsilon\rightarrow u_0$ uniformly in $C^2(\bar{\Omega})$ as $\epsilon\rightarrow0$, and it hold that $u_0\neq0$ due to \eqref{SHL}. Hence by taking the limit in \eqref{minimi}, we obtain
$$0\neq\int_{\Omega}|\nabla u_{0}|^2dx=C_{HLS}\left(\int_{\Omega}(|x|^{-(n-2)} \ast|u_{0}|^{p})|u_{0}|^{p} dx\right)^{\frac{1}{p}},$$
which contradicts that $C_{HLS}$ cannot be achieved by a minimizer in a bounded domain, see \cite{GY18}. Hence it should hold that $\lambda_{\epsilon}\rightarrow\infty$ as $\epsilon\rightarrow0$.
\end{proof}

\begin{lem}\label{finite}
It holds that
\begin{equation*}
\lim\limits_{\epsilon\rightarrow0}\lambda_{\epsilon}dist(x_{\epsilon},\partial\Omega)=\infty \quad\mbox{and}\quad\lim\limits_{\epsilon\rightarrow0}\lambda_{\epsilon}^{\epsilon}=1.
\end{equation*}
In addition, $v_{\epsilon}\rightarrow W[0,1]=(\frac{1}{1+|x|^2})^{(n-2)/2}$ in $\mathcal{D}^{1,2}(\mathbb{R}^n)$ as $\epsilon\rightarrow0$.
\end{lem}
\begin{proof}
It is noticing that $\lambda_{\epsilon}\rightarrow\infty$ as $\epsilon\rightarrow0$ and so we have that $\lambda_{\epsilon}^{\epsilon}\rightarrow1$ as $\epsilon\rightarrow0$.
If $\lambda_{\epsilon}dist(x_{\epsilon},\partial\Omega)\rightarrow c\in[0,\infty)$ as $\epsilon\rightarrow0$, then one may assume that $\Omega_{\epsilon}$ converges to $\mathbb{R}^{n}_{+}$. Otherwise, i.e., if
$$\lim_{\epsilon\rightarrow0}\lambda_{\epsilon}dist(x_{\epsilon},\partial\Omega)=\infty,$$ then $\Omega_{\epsilon}$ converges to $\mathbb{R}^{n}$. In fact, we notice that, by global compactness lemma \cite{yyz}, this is always the case. Moreover, $W$ attains in $\mathbb{R}^{n}$ and combining \eqref{minimi}, we get
\begin{equation}\label{WW}
v_{\epsilon}\rightarrow W[0,1]=\big(\frac{1}{1+|x|^2}\big)^{\frac{n-2}{2}}\quad\mbox{in}\hspace{2mm}\mathcal{D}^{1,2}(\mathbb{R}^n)
\end{equation}
as $\epsilon\rightarrow0$. Concluding the proof.
\end{proof}

We first state the following lemma which is useful in our analysis (see \cite{HANZCHAO,GT} for the proof).
\begin{lem}\label{regular1}
Assume that $u\in H^{1}_{loc}(\Omega)$ be a positive solution of $-\Delta u=a(x)u^q$ with $1<q\leq2^{\ast}-1$. Then there exists $\varepsilon_0>0$ depending on $n$, such that if
$$
\int_{B(Q,r)}|a(x)u^{q-1}|^{\frac{n}{2}}dx\leq\varepsilon_0,
$$
then
\begin{align}\label{cU1}
\|u\|_{L^{(2^{\ast})^2/2}(B(Q,r))}\leq c(n)r^{-\frac{2}{2^{\ast}}}\|u\|_{L^{2^{\ast}}(Q,2r)},
\end{align}
with $c(n)$ depending on $n$ only. Furthermore, if there exists $0<\delta^{\prime}<2$ such that
\begin{align}
a(x)u^{q-1}\in L^{\frac{n}{2-\delta^{\prime}}}\big(B(Q,2r)\big),
\end{align}
then
\begin{align}\label{cU11}
\sup\limits_{B(Q,r)}u\leq cr^{-n}\Big(\int_{B(Q,2r)}u^{2^{\ast}}dx\Big)^{\frac{1}{2^{\ast}}},
\end{align}
where $c$ depends only on $n$, $\delta^{\prime}$ and $r^{\delta^{\prime}}\|au^{q-1}\|_{L^{n/(2-\delta^{\prime})}(B(Q,2r))}$.
\end{lem}
The following lemma will play a key role in our analysis.
\begin{lem}\label{cWU}
There exists a positive constant $c$, independent of $\varepsilon$, such that
\begin{align}\label{cU}
v_\varepsilon(x)\leq cW(x)\quad \mbox{for}\quad x\in\mathbb{R}^n,
\end{align}
where $W(x)=W[0,1](x)$.
\end{lem}
\begin{proof}
Let us recall the problem \eqref{rescaled},
\begin{equation*}
\begin{cases}
-\Delta v_{\epsilon}(x)=\big(|x|^{-{(n-2)}}\ast v_\epsilon^{p-\epsilon}\big)v_\epsilon^{p-1-\epsilon}\hspace{6mm}\mbox{in}\hspace{2mm} \Omega_{\epsilon},\\
0\leq v_{\varepsilon}(x)\leq1~~~~~~~~~~~~~~~~~~~~~~~~~~~~~~~~~~~~~~~~~~~~~\text{in}~~\Omega_{\varepsilon},
\\
v_{\varepsilon}(x)=0~~~~~~~~~~~~~~~~~~~~~~~~~~~~~~~~~~~~~~~~~~~~~~~~~~~~\text{on}~~\partial\Omega_{\varepsilon},\\
v_{\varepsilon}(0)=\max\limits_{x\in\Omega_{\varepsilon}}v_{\varepsilon}(x)=1,
\end{cases}
\end{equation*}
and problem \eqref{Kelvin}, we have
\begin{equation*}
\begin{cases}
-\Delta w_{\epsilon}(x)
=\frac{1}{|x|^{{\epsilon(n-2)}}}\Big(\int_{\Omega_{\epsilon}^{\ast}}\frac{w_\epsilon^{p-\epsilon}(y)}{|x-y|^{{n-2}}|y|^{{\epsilon(n-2)}}} dy\Big)w_\epsilon^{p-1-\epsilon}\hspace{4mm}\mbox{in}\hspace{2mm} \Omega_{\epsilon}^{\ast},\\
w_{\varepsilon}(x)=0~~~~~~~~~~~~~~~~~~~~~~~~~~~~~~~~~~~~~~~~~~~~~~~~~~~~\hspace{8mm}\hspace{8mm}\hspace{6mm}\text{on}~~\partial\Omega_{\epsilon}^{\ast}.\\
\end{cases}
\end{equation*}
By definition of $w_\epsilon$ and $0\leq v_{\varepsilon}(x)\leq1$, we have
\begin{align}\label{decay}
w_\epsilon\leq \frac{1}{|x|^{n-2}}, \quad \mbox{for}\hspace{2mm}x\in\Omega_{\varepsilon}^{\ast}.
\end{align}
Thus to conclude the proof of \eqref{cU}, it is sufficient to obtain
\begin{align}\label{bound}
\|w_\epsilon\|_{L^\infty(B(0,\frac{r}{2}))}\leq c
\end{align}
for some $c>0$ depending on $n$ and $\mu$ and where $B(0,\frac{r}{2})$ is a ball near $0$.
In virtue of Lemma \ref{regular1}, letting $a(x)$ be given by
\begin{equation*}
a(x):=\frac{1}{|x|^{{\epsilon(n-2)}}}\Big(\int_{\Omega_{\epsilon}^{\ast}}\frac{w_\epsilon^{p-\epsilon}(y)}{|x-y|^{{n-2}}|y|^{{\epsilon(n-2)}}} dy\Big).
\end{equation*}
In the following, we claim that $\|a(x)\|_{L^\infty(\Omega_{\varepsilon}^{\ast})}\leq c_1$ where $c_1$ independent of $\epsilon>0$ provided $\epsilon$ is sufficiently small. It is notice that
$$|x|^{{\epsilon(n-2)}}a(x)=\int_{\Omega_{\epsilon}^{\ast}}\frac{w_\epsilon^{p-\epsilon}(y)}{|x-y|^{{n-2}}|y|^{{\epsilon(n-2)}}} dy.$$
Then we have, for any $r>0$,
\begin{equation}\label{ax1}
|x|^{{\epsilon(n-2)}}a(x)\leq\int_{B(0,\frac{r}{2})}\frac{|w_\epsilon(y)|^{p-\epsilon}}{|x-y|^{{n-2}}|y|^{{\epsilon(n-2)}}} dy+\int_{\Omega_{\epsilon}^{\ast}\setminus B(0,\frac{r}{2})}\frac{|w_\epsilon(y)|^{p-\epsilon}}{|x-y|^{{n-2}}|y|^{{\epsilon(n-2)}}} dy.
\end{equation}
For any $x\in B(0,r)$, a direct computation by H\"{o}lder inequality we get
\begin{equation}\label{xn2ax}
\begin{split}
\int_{B(0,\frac{r}{2})}\frac{|w_\epsilon(y)|^{p-\epsilon}}{|x-y|^{{n-2}}|y|^{{\epsilon(n-2)}}} dy&\leq\int_{B(0,\frac{r}{2})}\frac{|w_\epsilon(y)|^{p-\epsilon}}{|y|^{{(\epsilon+1)(n-2)}}} dy+\int_{B(x,\frac{3}{2}r)}\frac{|w_\epsilon(y)|^{p-\epsilon}}{|x-y|^{{(\epsilon+1)(n-2)}}} dy\\&
\leq\big\|w_\epsilon\big\|^{p-\epsilon}_{L^{\frac{(p-\epsilon)q^{\prime}}{q^{\prime}-1}}(B(0,\frac{r}{2}))}\Big\|\frac{1}{|y|^{{(\epsilon+1)(n-2)}}}\Big\|_{L^{q^{\prime}}(B(0,\frac{r}{2}))}\\&+
\big\|w_\epsilon\big\|^{p-\epsilon}_{L^{\frac{(p-\epsilon)q^{\prime}}{q^{\prime}-1}}(B(x,\frac{3}{2}r))}\Big\|\frac{1}{|x-y|^{{(\epsilon+1)(n-2)}}}\Big\|_{L^{q^{\prime}}(B(x,\frac{3}{2}r))}<\infty,
\end{split}
\end{equation}
where $1<q^{\prime}<\frac{n}{(\epsilon+1)(n-2)}$.
For any $x\in \Omega_{\epsilon}^{\ast}\setminus B(0,r)$, we have $|x-y|>|y|$. Therefore, it holds that
\begin{equation}\label{xax}
\begin{split}
\int_{B(0,\frac{r}{2})}\frac{|w_\epsilon(y)|^{p-\epsilon}}{|x-y|^{{n-2}}|y|^{{\epsilon(n-2)}}} dy&\leq\int_{B(0,\frac{r}{2})}\frac{|w_\epsilon(y)|^{p-\epsilon}}{|y|^{{(\epsilon+1)(n-2)}}} dy\\&
\leq\big\|w_\epsilon\big\|^{p-\epsilon}_{L^{\frac{(p-\epsilon)q^{\prime}}{q^{\prime}-1}}(B(0,\frac{r}{2}))}\Big\|\frac{1}{|y|^{{(\epsilon+1)(n-2)}}}\Big\|_{L^{q^{\prime}}(B(0,\frac{r}{2}))}<\infty,
\end{split}
\end{equation}
where $1<q^{\prime}<\frac{n}{(\epsilon+1)(n-2)}$. On the other hand, we deduce that
\begin{equation}\label{xn}
\begin{split}
&\int_{\Omega_{\epsilon}^{\ast}\setminus B(0,\frac{r}{2})}\frac{|w_\epsilon(y)|^{p-\epsilon}}{|x-y|^{{n-2}}|y|^{{\epsilon(n-2)}}} dy\leq\int_{\mathbb{R}^n\setminus B(0,\frac{r}{2})}\frac{|w_\epsilon(y)|^{p-\epsilon}}{|x-y|^{{n-2}}|y|^{{\epsilon(n-2)}}} dy\\&
=\int_{(\mathbb{R}^n\setminus B(0,\frac{r}{2}))\cap B_{\frac{r}{2}}(x)}\frac{|w_\epsilon(y)|^{p-\epsilon}}{|x-y|^{{n-2}}|y|^{{\epsilon(n-2)}}} dy+\int_{(\mathbb{R}^n\setminus B(0,\frac{r}{2}))\cap( \mathbb{R}^n\setminus B_{\frac{r}{2}}(x))}\frac{|w_\epsilon(y)|^{p-\epsilon}}{|x-y|^{{n-2}}|y|^{{\epsilon(n-2)}}} dy.
\end{split}
\end{equation}
Similar to argument of \eqref{xn2ax}, we first obtain
\begin{equation}\label{q-1}
\begin{split}
\int_{(\mathbb{R}^n\setminus B(0,\frac{r}{2}))\cap B_{\frac{r}{2}}(x)}\frac{|w_\epsilon(y)|^{p-\epsilon}}{|x-y|^{{n-2}}|y|^{{\epsilon(n-2)}}} dy&\leq\frac{c}{r^{\epsilon(n-2)}}\int_{(\mathbb{R}^n\setminus B(0,\frac{r}{2}))\cap B_{\frac{r}{2}}(x)}\frac{|w_\epsilon(y)|^{p-\epsilon}}{|x-y|^{{n-2}}}dy\\&
\leq\frac{c}{r^{\epsilon(n-2)}}\int_{ B_{\frac{r}{2}}(x)}\frac{|w_\epsilon(y)|^{p-\epsilon}}{|x-y|^{{n-2}}}dy<\infty.
\end{split}
\end{equation}
Moreover, H\"{o}lder inequality gives that
\begin{equation}\label{q1pie}
\begin{split}
\int_{(\mathbb{R}^n\setminus B(0,\frac{r}{2}))\cap( \mathbb{R}^n\setminus B_{\frac{r}{2}}(x))}&\frac{|w_\epsilon(y)|^{p-\epsilon}}{|x-y|^{n-2}|y|^{{\epsilon(n-2)}}} dy\\&\leq\frac{c}{r^{n-2}}\big\|w_\epsilon\big\|^{p-\epsilon}_{L^{\frac{(p-\epsilon)q^{\prime}}{q^{\prime}-1}}(\mathbb{R}^n\setminus B(0,\frac{r}{2}))}\Big\|\frac{1}{|y|^{{\epsilon(n-2)}}}\Big\|_{L^{q^{\prime}}(\mathbb{R}^n\setminus B(0,\frac{r}{2}))}<\infty,
\end{split}
\end{equation}
where $q^{\prime}>\frac{n}{\epsilon(n-2)}$. By now plug \eqref{xn2ax}-\eqref{xax} and \eqref{xn}-\eqref{q1pie} in \eqref{ax1} we finally get
$$|x|^{{\epsilon(n-2)}}a(x)\in L^\infty(\Omega_{\varepsilon}^{\ast}),$$
where the bound independent of $\epsilon>0$ provided $\epsilon$ is sufficiently small. Therefore, if we have $|x|^{-\epsilon(n-2)}\in L^\infty(\Omega_{\varepsilon}^{\ast})$ and where $c$ independent of $\epsilon>0$, the claim follow. Indeed, by Lemma \ref{infinite}, we have $\||x|^{-\epsilon(n-2)}\|_{L^\infty(\Omega_{\varepsilon}^{\ast})}\leq c$  where $c$ independent of $\epsilon>0$.

Using Lemma \ref{regular1} and since $v_{\epsilon}\rightarrow W[0,1]$ in $\mathcal{D}^{1,2}(\mathbb{R}^n)$ by Lemma \ref{finite}, we can easily show that for any $\epsilon_1>0$ there exists $R_0>0$ such that
\begin{equation}\label{BRO}
\int_{B(0,2R_0)}w_{\epsilon}^{p}(x)dx\leq\epsilon_1,
\end{equation}
where we used that $p>2$.
By \eqref{cU1} of Lemma \ref{regular1} choosing $\epsilon_1=\epsilon_0$ and $\delta^{\prime}=2-\frac{2n(p-2-\epsilon)}{p^2}$, we deduce that
\begin{equation}\label{cO}
w_{\epsilon}\in L^{\frac{n(p-2-\epsilon)}{2-\delta^{\prime}}}\big(B(0,R_0)\big).
\end{equation}
Therefore,
\begin{equation*}
\int_{\Omega_{\varepsilon}^{\ast}}|a(x)w_\epsilon^{p-2-\epsilon}|^{\frac{n}{2-\delta^{\prime}}}dx\leq c\int_{\Omega_{\varepsilon}^{\ast}\setminus B(0,R_0)}|w_\epsilon^{p-2-\epsilon}|^{\frac{n}{2-\delta^{\prime}}}dx+\int_{ B(0,R_0)}|w_\epsilon^{p-2-\epsilon}|^{\frac{n}{2-\delta^{\prime}}}dx\leq c.
\end{equation*}
It follows from \eqref{cU11} of Lemma \ref{regular1} that $w_{\epsilon}\in L^{\infty}(B(0,R_0))$, and hence combining this bound with \eqref{decay}, we conclude that \eqref{decay}, as desired.
\end{proof}

\section{Estimates for $u_\varepsilon$ near the blow-up point}\label{sec:weaksolution}
Let $\{u_\epsilon\}_{\epsilon>0}$ be a family of least enegy solutions to \eqref{ele-1.1} and $x_{\epsilon}$ the blow-up point given in \eqref{miu}. The goal of this section is to deduce uniform bounded of least energy solutions $u_\epsilon$ near $\partial\Omega$ from a local Pohozaev-type identity and sharp point-wise estimates of $u_\epsilon$. More precisely, we prove that
\begin{thm}\label{blow-up}
Let $\Omega$ be a smooth bounded domain in $\mathbb{R}^n$.
Suppose that $\{u_\epsilon\}_{\epsilon>0}$ be a family of least enegy solutions to \eqref{ele-1.1} and satisfying \eqref{minimi}. Then $u_{\epsilon}$ are uniformly bounded near $\partial\Omega$ with respect to $\epsilon>0$ small and blow-up point an interior point of $\Omega$.
\end{thm}
Here the proof of our result will use a local Pohozaev-type identity and pointwise estimates of $u_\epsilon$ and $\nabla u_\epsilon$.
\subsection{The local Pohozaev-type identity}
Next, we will prove a local Pohozaev-type identity for functions in an arbitrary smooth open subset of $\Omega$.
\begin{lem}\label{lem:hardyinequality}
Assume that $u_{\epsilon}\in C^2(\Omega)$ is a solution of \eqref{ele-1.1}. Then  for an arbitrary smooth open subset $B^{\prime}\subset\Omega$, one has the following identity
\begin{equation}\label{PH}
\begin{split}
&-\int_{\partial B^{\prime}}\frac{\partial u_{\epsilon}}{\partial x_j}\frac{\partial u_{\epsilon}}{\partial\nu}dS_{x}+\frac{1}{2}\int_{\partial B^{\prime}}|\nabla u_{\epsilon}|^2\nu_jdS_{x}
=\frac{2
}{p-\epsilon}
\int_{\partial B^{\prime}}\int_{ B^{\prime}}\frac{u^{p-\epsilon}_{\epsilon}(x)u^{p-\epsilon}_{\epsilon}(z)}
{|x-z|^{n-2}}\nu_jdz dS_{x}\\&+\frac{1}{p-\epsilon}
\int_{\partial B^{\prime}}\int_{\Omega\setminus B^{\prime}}\frac{u^{p-\epsilon}_{\epsilon}(x)u^{p-\epsilon}_{\epsilon}(z)}
{|x-z|^{n-2}}\nu_jdz dS_{x}+\frac{n-2}{p-\epsilon}\int_{B^{\prime}}\int_{\Omega\setminus B^{\prime}}(x_j-z_j)\frac{u_{\epsilon}^{p-\epsilon}(z)u^{p-\epsilon}_{\epsilon}(x)}
{|x-z|^{n}}dz dx,
\end{split}
\end{equation}
for $j=1,\cdots,n$,
where $\nu=\nu(x)$ denotes the unit outward normal to the boundary $\partial B^{\prime}$. Moreover, we have
\begin{equation}\label{pohozave}
\begin{split}
\int_{\partial\Omega}\Big(\frac{\partial u_{\epsilon}}{\partial \nu}\Big)^2(\nu,x-z)dS_{x}=\frac{(n-2)^2}{n+2-\epsilon(n-2)}\cdot\epsilon\cdot\int_{\Omega}(|x|^{-(n-2)} \ast u_{\epsilon}^{p-\epsilon})u_{\epsilon}^{p-\epsilon} dx
\end{split}
\end{equation}
\end{lem}
\begin{proof}
To prove \eqref{PH},
We multiply \eqref{ele-1.1} by $\frac{\partial u_{\epsilon}}{\partial x_j}$ and integrating on $B^{\prime}$,  we have
\begin{equation}\label{PH1}
-\int_{B^{\prime}}\Delta u_{\epsilon}\frac{\partial u_{\epsilon}}{\partial x_j}dx
=\int_{B^{\prime}}\frac{\partial u_{\epsilon}}{\partial x_j}\Big(\int_{\Omega}\frac{u^{p-\epsilon}_{\epsilon}(\xi)}
{|x-z|^{\mu}}dz\Big)u^{p-\epsilon}_{\epsilon}(x)dx.
\end{equation}
Then, a direct computation shows that
\begin{equation*}
\begin{split}
\int_{B^{\prime}}\frac{\partial u_{\epsilon}}{\partial x_j}\Big(&\int_{B^{\prime}}\frac{u^{p-\epsilon}_{\epsilon}(z)}
{|x-z|^{n-2}}dz\Big)u^{p-1-\epsilon}_{\epsilon}(x)dx
=-(p-1-\epsilon)\int_{B^{\prime}}\frac{\partial u_{\epsilon}}{\partial x_j}\Big(\int_{B^{\prime}}\frac{u^{p-\epsilon}_{\epsilon}(z)}
{|x-z|^{n-2}}dz\Big)u^{p-1-\epsilon}_{\epsilon}(x)dx\\&
+(n-2)\int_{B^{\prime}}\int_{B^{\prime}}(x_j-z_j)\frac{u^{p-\epsilon}_{\epsilon}(z)u^{p-\epsilon}_{\varepsilon}(x)}
{|x-z|^{n}}dz dx+\int_{\partial B^{\prime}}\int_{ B^{\prime}}\frac{u^{p-\epsilon}_{\varepsilon}(x)u^{p-\epsilon}_{\varepsilon}(z)}
{|x-z|^{n-2}}\nu_jdz dS_{x}.
\end{split}
\end{equation*}
Therefore, we have
\begin{equation}\label{ph00}
\begin{split}
&\int_{B^{\prime}}\frac{\partial u_{\epsilon}}{\partial x_j}\Big(\int_{B^{\prime}}\frac{u^{p-\epsilon}_{\epsilon}(z)}
{|x-z|^{\mu}}d\xi\Big)u^{p-1-\epsilon}_{\epsilon}(x)dx
\\&
=\frac{1}{p-\epsilon}
\int_{\partial B^{\prime}}\int_{ B^{\prime}}\frac{u^{p-\epsilon}_{\epsilon}(x)u^{p-\epsilon}_{\epsilon}(z)}
{|x-z|^{n-2}}\nu_jdz dS_x+\frac{n-2}{p-\epsilon}\int_{B^{\prime}}\int_{B^{\prime}}(x_j-z_j)\frac{u^{p-\epsilon}_{\epsilon}(z)u^{p-\epsilon}_{\epsilon}(x)}
{|x-z|^{n}}dz dx.
\end{split}
\end{equation}
Analogously, we also have
\begin{equation*}
\begin{split}
&\int_{B^{\prime}}\frac{\partial u_{\epsilon}}{\partial z_j}\Big(\int_{B^{\prime}}\frac{u^{p-\epsilon}_{\epsilon}(x)}
{|x-z|^{n-2}}dx\Big)u^{p-1-\epsilon}_{\epsilon}(z)dz
\\&
=\frac{1}{p-\epsilon}
\int_{\partial B^{\prime}}\int_{ B^{\prime}}\frac{u^{p-\epsilon}_{\epsilon}(z)u^{p-\epsilon}_{\epsilon}(x)}
{|x-z|^{n-2}}\nu_jdx dS_{z}+\frac{n-2}{p-\epsilon}\int_{B^{\prime}}\int_{B^{\prime}}(z_j-x_j)\frac{u^{p-\epsilon}_{\epsilon}(x)u^{p-\epsilon}_{\epsilon}(z)}
{|x-z|^{n}}dxdz.
\end{split}
\end{equation*}
Thus we eventually get that
\begin{equation}\label{PH2}
\begin{split}
\int_{B^{\prime}}&\frac{\partial u_{\epsilon}}{\partial x_j}\Big(\int_{B^{\prime}}\frac{u^{p-\epsilon}_{\epsilon}(\xi)}
{|x-z|^{n-2}}dz\Big)u^{p-1-\epsilon}_{\epsilon}(x)dx
=\frac{2}{p-\epsilon}
\int_{\partial B^{\prime}}\int_{ B^{\prime}}\frac{u^{p-\epsilon}_{\epsilon}(x)u^{p-\epsilon}_{\epsilon}(z)}
{|x-z|^{n-2}}\nu_jdz dS_{x}.
\end{split}
\end{equation}
Moreover, similar to the calculations of \eqref{ph00}, we deduce that
\begin{equation}\label{ph01}
\begin{split}
&\int_{B^{\prime}}\frac{\partial u_{\epsilon}}{\partial x_j}\Big(\int_{\Omega\setminus B^{\prime}}\frac{u^{p-\epsilon}_{\epsilon}(z)}
{|x-z|^{n-2}}dz\Big)u^{p-1-\epsilon}_{\epsilon}(x)dx
\\&
=\frac{1}{p-\epsilon}
\int_{\partial B^{\prime}}\int_{\Omega\setminus B^{\prime}}\frac{u^{p-\epsilon}_{\epsilon}(x)u^{p-\epsilon}_{\epsilon}(z)}
{|x-z|^{n-2}}\nu_jdz dS_{x}+\frac{n-2}{p-\epsilon}\int_{B^{\prime}}\int_{\Omega\setminus B^{\prime}}(x_j-z_j)\frac{u^{p-\epsilon}_{\epsilon}(z)u^{p-\epsilon}_{\epsilon}(x)}
{|x-z|^{n}}dz dx.
\end{split}
\end{equation}
On the other hand, we have
\begin{equation}\label{ph}
\begin{split}
-\int_{B^{\prime}}\Delta u_{\epsilon}\frac{\partial u_{\epsilon}}{\partial x_j}dx=-\int_{\partial B^{\prime}}\frac{\partial u_{\epsilon}}{\partial \nu}\frac{\partial u_{\epsilon}}{\partial x_j}dS_{x}+\sum_{l=1}^{n}\int_{ B^{\prime}}\frac{\partial u_{\epsilon}}{\partial x_j}\frac{\partial u^2_{\epsilon}}{\partial x_lx_i}dx.
\end{split}
\end{equation}
Hence, combining \eqref{PH1}, \eqref{PH2}, \eqref{ph01} and \eqref{ph}, we conclude that \eqref{PH}.

Finally, it is clear that
$$0=\big(\Delta u_\epsilon+(|x|^{-{(n-2)}}\ast u_\epsilon^{p-\epsilon})u_\epsilon^{p-1-\epsilon}\big)x\cdot\nabla u.$$
By symmetry and integrating by parts, we obtain
\begin{equation}\label{zave}
\frac{1}{2}\int_{\partial\Omega}\Big(\frac{\partial u_{\epsilon}}{\partial\nu}\Big)^{2}(x-z,n)dS_x=\frac{n+2}{2(p-\epsilon)}\int_{\Omega}(|x|^{-(n-2)} \ast u_{\epsilon}^{p-\epsilon})u_{\epsilon}^{p-\epsilon}dx+\frac{2-n}{2}\int_{\Omega}|\nabla u_{\epsilon}|^2dx.
\end{equation}
Therefore we get the conclusion \eqref{pohozave} by \eqref{ele-1.1} and \eqref{zave}.
\end{proof}

\subsection{Sharp pointwise estimates of $u_\epsilon$ and $\nabla u_\epsilon$}
We first prepare a preliminary lemma which are needed to prove Lemma \ref{lem:Sobolevinequalitynot}.
\begin{lem}[\cite{Ackerman}]\label{Gx}
For all $x\neq y\in\Omega$, there exists a positive constant $c$ such that
$$0<G(x,y)<\frac{c}{|x-y|^{n-2}}\quad\mbox{and}\quad|\nabla_{x}G(x,y)|\leq\frac{c}{|x-y|^{n-1}}.$$
\end{lem}
Assume that $d_{\epsilon}=\frac{1}{4}dist(x_{\epsilon},\partial\Omega)\rightarrow0$ and the constant $A_{W}$ is given by
 $$A_{W}:=\int_{\mathbb{R}^n}\int_{\mathbb{R}^n}\frac{W^{p}(z)W^{p-1}(y)}{|y-z|^{n-2}}dydz.$$
 Then we have the following pointwise estimates of $u_\epsilon$ and first-order derivatives of $u_\epsilon$.
\begin{lem}\label{lem:Sobolevinequalitynot}
Suppose that $u_{\epsilon}$ be a solution of the equation \eqref{ele-1.1}.
Then for each point $x\in\partial B^{\prime}(x_{\epsilon},2d_{\epsilon})$, there holds
\begin{align*}
u_{\epsilon}(x)=\lambda_{\epsilon}^{-\frac{n-2}{2}}A_{W}G(x,x_{\epsilon})+o(\lambda_{\epsilon}^{-\frac{n-2}{2}}d_{\epsilon}^{-(n-2)}),
\end{align*}
and
\begin{align*}
\nabla u_{\epsilon}(x)=\lambda_{\epsilon}^{-\frac{n-2}{2}}A_{W}\nabla_{x}G(x,x_{\epsilon})+o(\lambda_{\epsilon}^{-\frac{n-2}{2}}d_{\epsilon}^{-(n-1)}).
\end{align*}
\end{lem}
\begin{proof}
We set
$$I_1:=G(x,x_{\epsilon})\int_{\Omega}\int_{\Omega}\frac{u_\epsilon^{p-\epsilon}(z)u_\epsilon^{p-1-\epsilon}(y)}{|y-z|^{n-2}}dydz,$$
$$I_2:=\int_{\Omega}\int_{\Omega}\frac{u_\epsilon^{p-\epsilon}(z)u_\epsilon^{p-1-\epsilon}(y)}{|y-z|^{n-2}}[G(x,y)-G(x,x_{\epsilon})]dydz.
$$
Then we get the integral representation of $u_{\epsilon}$ from \eqref{ele-1.1}:
\begin{align}\label{Green}
u_\epsilon=I_1+I_2.
\end{align}
Note that
\begin{equation*}
\begin{split}
\lim\limits_{\epsilon\rightarrow0^{+}}\lambda_{\epsilon}^{\frac{n-2}{2}}\int_{\Omega}\int_{\Omega}\frac{u_\epsilon^{p-\epsilon}(z)u_\epsilon^{p-1-\epsilon}(y)}{|y-z|^{n-2}}dydz
&=\lim\limits_{\epsilon\rightarrow0^{+}}\int_{\Omega_{\epsilon}}\int_{\Omega_{\epsilon}}\frac{v_\epsilon^{p-\epsilon}(z)v_\epsilon^{p-1-\epsilon}(y)}{|y-z|^{n-2}}dydz\\&
=\int_{\mathbb{R}^n}\int_{\mathbb{R}^n}\frac{W^{p}(z)W^{p-1}(y)}{|y-z|^{n-2}}dydz=A_{W}.
\end{split}
\end{equation*}
where we have used Lemma \ref{finite}, Lemma \ref{cWU} and the dominated convergence theorem.
Therefore, with the help of Lemma \ref{Gx} and  a straightforward computation we obtain
\begin{equation*}
\begin{split}
I_1=\lambda_{\epsilon}^{-\frac{n-2}{2}}A_{W}G(x,x_\epsilon)+o(\lambda_{\epsilon}^{-\frac{n-2}{2}}d_{\epsilon}^{-(n-2)}).
\end{split}
\end{equation*}
To estimate $I_2$, we split $\Omega$ as $\Omega=B^{\prime}(x_{\epsilon},d_{\epsilon})\cup[B^{\prime}(x_{\epsilon},4d_{\epsilon})\setminus B^{\prime}(x_{\epsilon},d_{\epsilon})]\cup[\Omega\setminus B^{\prime}(x_{\epsilon},4d_{\epsilon})]=:\Omega_1\cup\Omega_2\cup\Omega_3$.
Then the following decomposition holds
\begin{equation*}
\begin{split}
I_2&=\int_{\Omega}\int_{\Omega}\frac{u_\epsilon^{p-\epsilon}(z)u_\epsilon^{p-1-\epsilon}(y)}{|y-z|^{n-2}}[G(x,y)-G(x,x_{\epsilon})]dydz\\&
=\int_{\Omega_1}\int_{\Omega_1}\cdots+\int_{\Omega_1}\int_{\Omega_2}\cdots+\int_{\Omega_1}\int_{\Omega_3}\cdots+\int_{\Omega_2}\int_{\Omega_1}\cdots+\int_{\Omega_2}\int_{\Omega_2}\cdots+\int_{\Omega_2}\int_{\Omega_3}\cdots\\&+
\int_{\Omega_3}\int_{\Omega_1}\cdots+\int_{\Omega_3}\int_{\Omega_2}\cdots+\int_{\Omega_3}\int_{\Omega_3}\cdots=:I_{21}+I_{22}+\cdots+I_{28}+I_{29}.
\end{split}
\end{equation*}
It is sufficient to analysis integrals $I_{21},\cdots,I_{29}$.

\textbf{Estimate of $I_{21}$.} We first note that $|x-x_{\epsilon}|$ and for any $y\in\Omega_1$, it holds that $|x-y|\geq d_{\epsilon}$. Thus Lemma \ref{Gx} gives
\begin{equation}\label{gxx}
|G(x,y)-G(x,x_\epsilon)|\leq c\frac{|y-x_\epsilon|}{d_{\epsilon}^{n-1}}\quad \mbox{for every}\quad y\in \Omega_1.
\end{equation}
By virtue of Lemma \ref{finite}, Lemma \ref{cWU} and \eqref{gxx}, we compute
\begin{equation*}
\begin{split}
I_{21}(x)&\leq \frac{c}{d_{\epsilon}^{n-1}}\lambda_{\epsilon}^{\frac{2(n-2)}{4-(n-2)\epsilon}[2(p-\epsilon)-1]}\int_{\Omega_1}\int_{\Omega_1}\frac{v_\epsilon^{p-\epsilon}(\lambda_{\epsilon}(z-x_{\epsilon}))v_\epsilon^{p-1-\epsilon}(\lambda_{\epsilon}(y-x_{\epsilon}))}{|y-z|^{n-2}}|y-x_\epsilon|dydz\\&
\leq \frac{c\lambda_{\epsilon}^{-(n-2)/2}}{d_{\epsilon}^{n-2}}\frac{1}{\lambda_{\epsilon}d_{\epsilon}}\int_{B^{\prime}(0,\lambda_{\epsilon}d_{\epsilon})}\int_{B^{\prime}(0,\lambda_{\epsilon}d_{\epsilon})}\frac{v_\epsilon^{p-\epsilon}(z)v_\epsilon^{p-1-\epsilon}(y)}{|y-z|^{n-2}}|y|dydz
=o\big(\frac{\lambda_{\epsilon}^{-(n-2)/2}}{d_{\epsilon}^{n-2}}\big),
\end{split}
\end{equation*}
where we have used $\lambda_{\epsilon}d_{\epsilon}\rightarrow\infty$ as $\epsilon\rightarrow0$ and the following the estimate:
\begin{equation*}
\begin{split}
\int_{B^{\prime}(0,\lambda_{\epsilon}d_{\epsilon})}\int_{B^{\prime}(0,\lambda_{\epsilon}d_{\epsilon})}&\frac{v_\epsilon^{p-\epsilon}(z)v_\epsilon^{p-1-\epsilon}(y)}{|y-z|^{n-2}}|y|dydz \\&\leq c\Big\|\frac{1}{1+|y|}\Big\|_{L^\frac{2n(p-\epsilon)} {p}(B^{\prime}(0,\lambda_{\epsilon}d_{\epsilon}))}^{(n-2)(p-\epsilon)}\Big\|\frac{|y|}{1+|y|}\Big\|_{L^\frac{2n(p-1-\epsilon)} {p}(B^{\prime}(0,\lambda_{\epsilon}d_{\epsilon}))}^{(n-2)(p-1-\epsilon)}\\&
\leq
\begin{cases}
c, \quad\quad\quad\quad\quad\quad\quad\quad\hspace{2mm}\quad\quad\mbox{ when } n<4-2(n-2)\epsilon,\\
c\log(\lambda_{\epsilon}d_{\epsilon}), \hspace{2mm}\quad\quad\quad\quad\quad\quad\mbox{ when } n=4-2(n-2)\epsilon,\\
c(\lambda_{\epsilon}d_{\epsilon})^{\frac{n(n+4)-2n(n-2)(p-1-\epsilon)}{n+2}}, \mbox{ when } n>4-2(n-2)\epsilon,
\end{cases}
\end{split}
\end{equation*}
by Hardy-Littlewood-Sobolev inequality, and Lemma \ref{cWU} and since $\epsilon$ is small enough.

\textbf{Estimate of $I_{22}$.}
In the region $y\in\Omega_1$ and $z\in\Omega_2$, similar to estimate of $I_{21}$, we get first
\begin{equation*}
\begin{split}
\int_{B^{\prime}(0,\lambda_{\epsilon}d_{\epsilon})}\int_{B^{\prime}(0,4\lambda_{\epsilon}d_{\epsilon})\setminus B^{\prime}(0,\lambda_{\epsilon}d_{\epsilon})}&\frac{v_\epsilon^{p-\epsilon}(z)v_\epsilon^{p-1-\epsilon}(y)}{|y-z|^{n-2}}|y|dydz \\&\leq
\begin{cases}
c(\lambda_{\epsilon}d_{\epsilon})^{-[\frac{2n(p-\epsilon)}{p}-n]}, \quad\quad\quad\quad\hspace{1.5mm}\mbox{ when } n<4-2(n-2)\epsilon,\\
c(\lambda_{\epsilon}d_{\epsilon})^{-[\frac{2n(p-\epsilon)}{p}-n]}\log(\lambda_{\epsilon}d_{\epsilon}), \hspace{1.5mm}\mbox{ when } n=4-2(n-2)\epsilon,\\
c(\lambda_{\epsilon}d_{\epsilon})^{\frac{2n[2(p-\epsilon)-1]}{p}+\frac{2n(n+3)}{n+2}}, \hspace{6mm}\mbox{ when } n>4-2(n-2)\epsilon,
\end{cases}
\end{split}
\end{equation*}
Consequently,
\begin{equation*}
\begin{split}
I_{22}(x)&\leq \frac{c}{d_{\epsilon}^{n-1}}\lambda_{\epsilon}^{\frac{2(n-2)}{4-(n-2)\epsilon}[2(p-\epsilon)-1]}\int_{\Omega_1}\int_{\Omega_2}\frac{v_\epsilon^{p-\epsilon}(\lambda_{\epsilon}(z-x_{\epsilon}))v_\epsilon^{p-1-\epsilon}(\lambda_{\epsilon}(y-x_{\epsilon}))}{|y-z|^{n-2}}|y-x_\epsilon|dydz\\&
\leq \frac{c\lambda_{\epsilon}^{-(n-2)/2}}{d_{\epsilon}^{n-2}}\frac{1}{\lambda_{\epsilon}d_{\epsilon}}\int_{B^{\prime}(0,\lambda_{\epsilon}d_{\epsilon})}
\int_{B^{\prime}(0,4\lambda_{\epsilon}d_{\epsilon})\setminus B^{\prime}(0,\lambda_{\epsilon}d_{\epsilon})}\frac{v_\epsilon^{p-\epsilon}(z)v_\epsilon^{p-1-\epsilon}(y)}{|y-z|^{n-2}}|y|dydz
\\&=o\big(\lambda_{\epsilon}^{-(n-2)/2}d_{\epsilon}^{-(n-2)}\big).
\end{split}
\end{equation*}

\textbf{Estimate of $I_{23}$.} As the previous computation, for $y\in\Omega_1$ and $z\in\Omega_3$, we have
$$I_{23}=o\big(\lambda_{\epsilon}^{-(n-2)/2}d_{\epsilon}^{-(n-2)}\big).$$

\textbf{Estimate of $I_{24}$.}
When $y\in\Omega_2$ and $x\in\partial B^{\prime}(x_{\epsilon},2d_{\epsilon})$, we obtain that $|x-y|\leq6d_{\epsilon}$. Then by Lemma \ref{Gx} we deduce that
\begin{equation}\label{I24}
|G(x,y)-G(x,x_\epsilon)|\leq|G(x,y)|+|G(x,x_\epsilon)|\leq\frac{c}{|x-y|^{n-2}}.
\end{equation}
Moreover, for each $y\in\Omega_2$ and by Lemma \ref{cWU} imply that
\begin{equation}\label{I245}
u_{\epsilon}(y)\leq c \lambda_{\epsilon}^{\frac{2(n-2)}{4-(n-2)\epsilon}}v_{\epsilon}\big(\lambda_{\epsilon}(y-x_{\epsilon})\big)\leq c\lambda_{\epsilon}^{\frac{2(n-2)}{4-(n-2)\epsilon}}(\lambda_{\epsilon}d_{\epsilon})^{-(n-2)}.
\end{equation}
Combining this bound and \eqref{I24}, we deduce that
\begin{equation*}
\begin{split}
I_{24}&\leq c\int_{\Omega_2}\int_{\Omega_1}\frac{u_\epsilon^{p-\epsilon}(z)}{|y-z|^{n-2}}\frac{u_\epsilon^{p-1-\epsilon}(y)}{|x-y|^{n-2}}dydz\\&
\leq c\lambda_{\epsilon}^{-(\frac{n-2}{2}+4)}d_{\epsilon}^{-(n+6)}\int_{\Omega_2}\int_{\Omega_1}\frac{1}{|y-z|^{n-2}}\frac{1}{|x-y|^{n-2}}dydz=
o\big(\lambda_{\epsilon}^{-(n-2)/2}d_{\epsilon}^{-(n-2)}\big)
\end{split}
\end{equation*}
as $\epsilon\rightarrow0$.

\textbf{Estimate of $I_{25}$.} Similarly, we conclude that
$$I_{25}=o\big(\lambda_{\epsilon}^{-(n-2)/2}d_{\epsilon}^{-(n-2)}\big).$$

\textbf{Estimate of $I_{26}$.} For $y\in\Omega_1$ and $z\in\Omega_3$, we have $u_{\epsilon}(z)\leq c\lambda_{\epsilon}^{\frac{2(n-2)}{4-(n-2)\epsilon}}v_{\epsilon}(\lambda_{\epsilon}(z-x_{\epsilon}))$. Hence, applying this estimate, Hardy-Littlewood-Sobolev inequality and \eqref{I24}-\eqref{I245}, and Lemma \ref{cWU}, we deduce that
\begin{equation*}
\begin{split}
I_{26}&\leq c\int_{\Omega_2}\int_{\Omega_3}\frac{u_\epsilon^{p-\epsilon}(z)}{|y-z|^{n-2}}\frac{u_\epsilon^{p-1-\epsilon}(y)}{|x-y|^{n-2}}dydz\\&
\leq c\frac{\lambda_{\epsilon}^{\frac{2(n-2)}{4-(n-2)\epsilon}[2(p-\epsilon)-1]}}{\lambda_{\epsilon}^{(n-2)(p-1-\epsilon)+n}d_{\epsilon}^{(n-2)(p-1-\epsilon)}}\Big\|\frac{1}{|x-y|^{n-2}}\Big\|_{L^\frac{2n} {n+2}(\Omega_2)}\Big\|v_\epsilon^{p-\epsilon}(y)\Big\|_{L^\frac{2n} {n+2}(\mathbb{R}^n\setminus B^{\prime}(0,4\lambda_{\epsilon}d_{\epsilon}))}\\&
=o\big(\lambda_{\epsilon}^{-\frac{n-2}{2}}d_{\epsilon}^{-(n-2)}(\lambda_{\epsilon}d_{\epsilon})^{-\frac{6-n}{2}}\big)=o\big(\lambda_{\epsilon}^{-\frac{n-2}{2}}d_{\epsilon}^{-(n-2)}\big),
\end{split}
\end{equation*}
since the following inequality holds
$$\frac{2n}{p}<n<\frac{2n(p-\epsilon)}{p}\quad\mbox{as}\quad\epsilon\rightarrow0.$$

\textbf{Estimate of $I_{27}$.} In the region $\Omega_3$, we have $|x-y|\geq2d_{\epsilon}$ and $|x-x_{\epsilon}|=2d_{\epsilon}$ for $x\in\partial B^{\prime}(x_{\epsilon},2d_{\epsilon})$. Then
\begin{equation}\label{I26}
|G(x,y)-G(x,x_\epsilon)|\leq|G(x,y)|+|G(x,x_\epsilon)|\leq \frac{c}{d_{\epsilon}^{n-2}}.
\end{equation}
By a direct computation, it follows that
\begin{equation*}
\begin{split}
I_{27}&\leq cd_{\epsilon}^{-(n-2)}\int_{\Omega_3}\int_{\Omega_1}\frac{u_\epsilon^{p-\epsilon}(z)u_\epsilon^{p-1-\epsilon}(y)}{|y-z|^{n-2}}dydz\\&
\leq c\frac{\lambda_{\epsilon}^{\frac{2(n-2)}{4-(n-2)\epsilon}[2(p-\epsilon)-1]}}{\lambda_{\epsilon}^{n+2}d_{\epsilon}^{n-2}}
\Big\|v_\epsilon^{p-\epsilon}(y)\Big\|_{L^\frac{2n} {n+2}( B^{\prime}(0,\lambda_{\epsilon}d_{\epsilon}))}\Big\|v_\epsilon^{p-1-\epsilon}(y)\Big\|_{L^\frac{2n} {n+2}(\mathbb{R}^n\setminus B^{\prime}(0,4\lambda_{\epsilon}d_{\epsilon}))}\\&
=O\big(\lambda_{\epsilon}^{-\frac{n-2}{2}}d_{\epsilon}^{-(n-2)}(\lambda_{\epsilon}d_{\epsilon})^{-(\frac{2n(p-1-\epsilon)}{p}-n)}\big)=o\big(\lambda_{\epsilon}^{-\frac{n-2}{2}}d_{\epsilon}^{-(n-2)}\big),
\end{split}
\end{equation*}
since $n<\frac{2n(p-1-\epsilon)}{p}$ as $\epsilon\rightarrow0$.

\textbf{Estimate of $I_{28}$.} Similar to the argument of $I_{27}$, we also obtain
$$I_{28}=o\big(\lambda_{\epsilon}^{-(n-2)/2}d_{\epsilon}^{-(n-2)}\big).$$

\textbf{Estimate of $I_{29}$.} We now proceed similarly to $I_{27}$. Fo $y\in\Omega_3$ and $z\in\Omega_{3}$, we find
\begin{equation*}
\begin{split}
I_{29}&\leq c\lambda_{\epsilon}^{-\frac{n-2}{2}}d_{\epsilon}^{-(n-2)}\Big\|v_\epsilon^{p-\epsilon}(y)\Big\|_{L^\frac{2n} {n+2}( \mathbb{R}^n\setminus B^{\prime}(0,4\lambda_{\epsilon}d_{\epsilon}))}\Big\|v_\epsilon^{p-1-\epsilon}(y)\Big\|_{L^\frac{2n} {n+2}(\mathbb{R}^n\setminus B^{\prime}(0,4\lambda_{\epsilon}d_{\epsilon}))}\\&
=O\big(\lambda_{\epsilon}^{-\frac{n-2}{2}}d_{\epsilon}^{-(n-2)}(\lambda_{\epsilon}d_{\epsilon})^{-[\frac{2n[2(p-\epsilon)-1]}{p}-2n]}\big)=o\big(\lambda_{\epsilon}^{-\frac{n-2}{2}}d_{\epsilon}^{-(n-2)}\big).
\end{split}
\end{equation*}
Summarizing, combining these estimates, we conclude that
$$I_2=o\big(\lambda_{\epsilon}^{-\frac{n-2}{2}}d_{\epsilon}^{-(n-2)}\big).$$
We now plug estimate of $I_1$ and estimate of $I_2$, we get the first identity in Lemma \ref{lem:Sobolevinequalitynot}.

To prove the second identity in Lemma \ref{lem:Sobolevinequalitynot}, we first note that for $x\in\partial B^{\prime}(x_{\epsilon},2d_{\epsilon})$, by \eqref{gxx}-\eqref{I24} and \eqref{I26}, there holds
\begin{equation*}
|\nabla G(x,y)-\nabla G(x,x_\epsilon)|\leq
\begin{cases}
c\frac{|y-x_\epsilon|}{d_{\epsilon}^{n}}, \quad\quad\hspace{1.5mm}\mbox{ when } y\in\Omega_1,\\
c\frac{1}{|x-y|^{n-1}}, \quad\hspace{1.5mm}\mbox{ when } y\in\Omega_2,\\
c\frac{1}{d_{\epsilon}^{n-1}}, \quad\hspace{8mm}\mbox{ when } y\in\Omega_3.
\end{cases}
\end{equation*}
By exploiting this bound and similar to the computation of $I_1$ and $I_2$, it is easy to see that the second conclusion of Lemma \ref{lem:Sobolevinequalitynot} holds. Concluding the proof.
\end{proof}

\subsection{Proof of Theorem \ref{blow-up}}
In this subsection, we can now prove Theorem \ref{blow-up}.
\begin{proof}[Proof of Theorem \ref{blow-up}]
Assume now by contradiction that the conclusion of Theorem \ref{blow-up} does not hold true, in other words that, up to a subsequence,
\begin{equation}\label{contradiction}
d_{\epsilon}=\frac{1}{4}dist(x_{\epsilon},\partial\Omega)\rightarrow0\quad\mbox{as}\quad\epsilon\rightarrow0.
\end{equation}
Recall that a local Pohozaev-type identity in Lemma \ref{lem:hardyinequality}, for any $1\leq j\leq n$, we denote in what follows
$$
F_{j,\epsilon}:=-\int_{\partial B^{\prime}(x_{\epsilon},2d_{\epsilon})}\frac{\partial u_{\epsilon}}{\partial x_j}\frac{\partial u_{\epsilon}}{\partial\nu}dS_{x}+\frac{1}{2}\int_{\partial B^{\prime}(x_{\epsilon},2d_{\epsilon})}|\nabla u_{\epsilon}|^2\nu_jdS_{x},
$$
\begin{equation*}
\begin{split}
K_{j,\epsilon}=&\frac{2
}{p-\epsilon}
\int_{\partial B^{\prime}(x_{\epsilon},2d_{\epsilon})}\int_{ B^{\prime}(x_{\epsilon},2d_{\epsilon})}\frac{u^{p-\epsilon}_{\epsilon}(x)u^{p-\epsilon}_{\epsilon}(z)}
{|x-z|^{n-2}}\nu_jdz dS_{x}\\&+\frac{1}{p-\epsilon}
\int_{\partial B^{\prime}(x_{\epsilon},2d_{\epsilon})}\int_{\Omega\setminus B^{\prime}(x_{\epsilon},2d_{\epsilon})}\frac{u^{p-\epsilon}_{\epsilon}(x)u^{p-\epsilon}_{\epsilon}(z)}
{|x-z|^{n-2}}\nu_jdz dS_{x}\\&+\frac{n-2}{p-\epsilon}\int_{B^{\prime}(x_{\epsilon},2d_{\epsilon})}\int_{\Omega\setminus B^{\prime}(x_{\epsilon},2d_{\epsilon})}(x_j-z_j)\frac{u_{\epsilon}^{p-\epsilon}(z)u^{p-\epsilon}_{\epsilon}(x)}
{|x-z|^{n}}dz dx
\end{split}
\end{equation*}
and compute each term in the right-hand side of $F_{j,\epsilon}$ and $K_{j,\epsilon}$.

\textbf{Estimate of $F_{j,\epsilon}$.}
Due to $|x-x_\epsilon|=2d_{\epsilon}$, it follows from Lemmas \ref{Gx}-\ref{lem:Sobolevinequalitynot} and $\lambda_{\epsilon}d_{\epsilon}\rightarrow\infty$ as $\epsilon\rightarrow0$ that
\begin{equation*}
\begin{split}
F_{j,\epsilon}=&-\lambda_{\epsilon}^{-(n-2)}A_{W}^2\int_{\partial B^{\prime}(x_{\epsilon},2d_{\epsilon})}\Big[\frac{\partial G(x,x_{\epsilon})}{\partial x_j}\frac{\partial G(x,x_{\epsilon})}{\partial\nu}-\frac{1}{2}|\nabla_{x}G(x,x_{\epsilon})|^2\nu_j\Big]dS_{x}\\&
+o\big(|\partial B^{\prime}(x_{\epsilon},2d_{\epsilon})|\lambda_{\epsilon}^{-(n-2)}d_{\epsilon}^{-2(n-1)}\big)
=:F_{j,\epsilon}^{\prime}+o\big(\lambda_{\epsilon}^{-(n-2)}d_{\epsilon}^{-(n-1)}\big).
\end{split}
\end{equation*}
We now proceed similarly to \cite{Cao-Peng-Yan2021} for $F_{j,\epsilon}$. Then we have
$$F_{j,\epsilon}^{\prime}=-\lambda_{\epsilon}^{-(n-2)}A_{W}^2\frac{\partial H}{\partial x_j}(x,x_\epsilon)\Big|_{x=x_{\epsilon}}.$$
For unique unit vector $\nu_{x_{\epsilon}}=(\alpha_1,\alpha_2,\cdots,\alpha_n)\nu_{x_\epsilon}\in \mathbb{S}^{n-1}$ such that $x_{\epsilon}+dist(x_{\epsilon},\partial\Omega)\in\partial\Omega$, then there exists some positive $c$ such that for $1\leq j\leq n$, there holds
\begin{equation}\label{HJE}
\begin{split}
-\sum_{j=1}^{n}\alpha_jF_{j,\epsilon}=\lambda_{\epsilon}^{-(n-2)}A_{W}^2\frac{\partial H}{\partial x_j}(x,x_\epsilon)\Big|_{x=x_{\epsilon}}+o\big(\lambda_{\epsilon}^{-(n-2)}d_{\epsilon}^{-(n-1)}\big)\geq c\frac{\lambda_\epsilon}{(\lambda_{\epsilon}d_{\epsilon})^{n-1}}.
\end{split}
\end{equation}
\textbf{Estimate of $K_{j,\epsilon}$.} It is noticing that for $x\in\partial B^{\prime}(x_{\epsilon},2d_{\epsilon})$, we have
\begin{equation}\label{KJEP}
u_{\epsilon}(x)\leq c \lambda_{\epsilon}^{\frac{2(n-2)}{4-(n-2)\epsilon}}v_{\epsilon}\big(\lambda_{\epsilon}(x-x_{\epsilon})\big)\leq c\lambda_{\epsilon}^{\frac{2(n-2)}{4-(n-2)\epsilon}}(\lambda_{\epsilon}d_{\epsilon})^{-(n-2)}.
\end{equation}
Therefore, combining Lemma \ref{cWU} and Hardy-Littlewood-Sobolev inequality, a direct computation gives that
\begin{equation}\label{KIE}
\begin{split}
\int_{\partial B^{\prime}(x_{\epsilon},2d_{\epsilon})}&\int_{ B^{\prime}(x_{\epsilon},2d_{\epsilon})}\frac{u^{p-\epsilon}_{\epsilon}(x)u^{p-\epsilon}_{\epsilon}(z)}
{|x-z|^{n-2}}\nu_jdz dS_{x}\\&\leq c \lambda_{\epsilon}^\frac{2-n}{2}\Big\|v_\epsilon^{p-\epsilon}(y)\Big\|_{L^\frac{2n} {n+2}( B^{\prime}(0,2\lambda_{\epsilon}d_{\epsilon}))}\big|\partial B^{\prime}(x_{\epsilon},2d_{\epsilon})\big|^{\frac{n+2}{2n}}\lambda_{\epsilon}^{\frac{(n-2)p}{2}}(\lambda_{\epsilon}d_{\epsilon})^{-(n-2)p}\\&
\leq c \lambda_{\epsilon}^2d_{\epsilon}^{\frac{(n+2)(n-1)}{2n}}(\lambda_{\epsilon}d_{\epsilon})^{-(n-2)p} \leq c \lambda_{\epsilon}(\lambda_{\epsilon}d_{\epsilon})^{1-(n-2)p}.
\end{split}
\end{equation}
Analogously, we also have
\begin{equation}\label{KIE12}
\int_{\partial B^{\prime}(x_{\epsilon},2d_{\epsilon})}\int_{\Omega\setminus B^{\prime}(x_{\epsilon},2d_{\epsilon})}\frac{u^{p-\epsilon}_{\epsilon}(x)u^{p-\epsilon}_{\epsilon}(z)}
{|x-z|^{n-2}}\nu_jdz dS_{x}\leq c \lambda_{\epsilon}(\lambda_{\epsilon}d_{\epsilon})^{1-(n-2)p}.
\end{equation}
Applying \eqref{KJEP}, we finally compute
\begin{equation}\label{KIE13}
\begin{split}
&\int_{B^{\prime}(x_{\epsilon},2d_{\epsilon})}\int_{\Omega\setminus B^{\prime}(x_{\epsilon},2d_{\epsilon})}(x_j-z_j)\frac{u_{\epsilon}^{p-\epsilon}(z)u^{p-\epsilon}_{\epsilon}(x)}
{|x-z|^{n}}dz dx
\\&\hspace{6mm}\leq
c\lambda_{\epsilon}^\frac{n+2}{2}(\lambda_{\epsilon}d_{\epsilon})^{-(n-2)p}\int_{B^{\prime}(x_{\epsilon},2d_{\epsilon})}\Big(\int_{\mathbb{R}^n\setminus B^{\prime}(x_{\epsilon},2d_{\epsilon})}\frac{(x_j-z_j)}
{|x-z|^{n}}dz\Big)u^{p-\epsilon}_{\epsilon}(x)dx\\&
\hspace{6mm}\leq
c\lambda_{\epsilon}^\frac{n+2}{2}(\lambda_{\epsilon}d_{\epsilon})^{-(n-2)p}\int_{B^{\prime}(x_{\epsilon},2d_{\epsilon})}\Big(\int_{ 2\lambda_{\epsilon}d_{\epsilon}}^{\infty}\int_{\omega\in \mathbb{S}^{n-1}}\frac{r\omega_j}
{r^{n}}\cdot r^{n-1}d\omega dr\Big)u^{p-\epsilon}_{\epsilon}(x)dx=0.
\end{split}
\end{equation}
From \eqref{KIE}, \eqref{KIE12} and \eqref{KIE13}, we conclude that
\begin{equation*}
|K_{j,\epsilon}|\leq c \lambda_{\epsilon}(\lambda_{\epsilon}d_{\epsilon})^{1-(n-2)p}.
\end{equation*}

As a consequence, by \eqref{HJE} and estimates of $|K_{j,\epsilon}|$, we obtain that
$$c\frac{\lambda_\epsilon}{(\lambda_{\epsilon}d_{\epsilon})^{n-1}}\leq-\sum_{j=1}^{n}\alpha_jF_{j,\epsilon}\leq-\sum_{j=1}^{n}|\alpha_j||K_{j,\epsilon}|\leq
c \lambda_{\epsilon}(\lambda_{\epsilon}d_{\epsilon})^{1-(n-2)p},$$
which means that $n-1\leq1-(n-2)p$ as $\epsilon\rightarrow0$, that is, $p\leq\frac{n}{n-2}$, a contradiction and concluding the proof.
\end{proof}

\section{Proof of Theorems \ref{Figalli}, \ref{Figalli2} and \ref{consequence}}\label{consequence00}

We first give a lemma which is used in the proof of Theorems \ref{Figalli} and \ref{Figalli2}.

\begin{lem}\label{thm:uniquenessofweaksolution}
 There exists $M>0$, such that
\begin{equation}\label{eq:energyestimateu}
\epsilon\leq M \lambda_{\epsilon}^{-[n+2-\frac{2(n-2)}{4-(n-2)\epsilon}(2(p-\epsilon)-1)]}.
\end{equation}
\end{lem}
\begin{proof}
By \eqref{pohozave} and \eqref{SHL}, it suffices to check that
\begin{align}\label{48}
\int_{\partial\Omega}\big(\frac{\partial u_{\epsilon}}{\partial \nu}\big)^2(\nu,x)dS_{x}\leq \lambda_{\epsilon}^{-(n-2)}.
\end{align}
From Lemma \ref{cWU} and Hardy-Littlewood-Sobolev inequality, we calculate
\begin{align}\label{49}
\begin{split}
\int_{\Omega}(|x|^{-(n-2)} \ast u_{\epsilon}^{p-\epsilon})u_{\epsilon}^{p-1-\epsilon}dx&\leq C\lambda_{\epsilon}^{-[n+2-\frac{2(n-2)}{4-(n-2)\epsilon}(2(p-\epsilon)-1)]}\int_{\mathbb{R}^n}\int_{\mathbb{R}^n}\frac{W^{p-\epsilon}(y)W^{p-1-\epsilon}(x)}{|x-y|^{n-2}} dxdy\\&\leq C\lambda_{\epsilon}^{-[n+2-\frac{2(n-2)}{4-(n-2)\epsilon}(2(p-\epsilon)-1)]}.
\end{split}
\end{align}
Furthermore, using Lemma \ref{cWU} again we obtain
\begin{align*}
\begin{split}
(\frac{1}{|x|^{n-2}}\ast u_{\epsilon}^{p-\epsilon})&u_{\epsilon}^{p-1-\epsilon}\leq
\frac{c}{\lambda_{\epsilon}^{[n-2-\frac{2(n-2)}{4-(n-2)\epsilon}](2(p-\epsilon)-1)}}\\&\times\Big(\int_{\mathbb{R}^n}\frac{1}{|x-y|^{n-2}}\frac{1}{1+|y-x_{0}|^{(n-2)(p-\epsilon)}}dy\Big)\frac{1}{|x-x_{0}|^{-(n-2)(p-\epsilon)}}
\end{split}
\end{align*}
for $x\neq x_0$. We now establish the following key result:
\begin{align}\label{infinity}
\begin{split}
\int_{\mathbb{R}^n}\frac{1}{|x-y|^{n-2}}\frac{1}{1+|y-x_{0}|^{(n-2)(p-\epsilon)}}dy\in L^{\infty}(\mathbb{R}^n).
\end{split}
\end{align}
Moreover we note that
$$[n+2-\frac{2(n-2)}{4-(n-2)\epsilon}(2(p-\epsilon)-1)]<[n-2-\frac{2(n-2)}{4-(n-2)\epsilon}](2(p-\epsilon)-1)\hspace{2mm}\mbox{as}\hspace{2mm}\epsilon\rightarrow0.$$
Then using this bound and \eqref{infinity}, there exists $M>0$ such that
\begin{align}\label{iy}
(\frac{1}{|x|^{n-2}}\ast u_{\epsilon}^{p-\epsilon})u_{\epsilon}^{p-1-\epsilon}\leq
M\lambda_{\epsilon}^{-[n+2-\frac{2(n-2)}{4-(n-2)\epsilon}(2(p-\epsilon)-1)]}\frac{1}{|x-x_{0}|^{(n-2)(p-\epsilon)}}
\end{align}
for $x\neq x_0$ and some $M>0$.
Combining \eqref{49}, \eqref{iy} and Lemma \ref{regular}, the estimate \eqref{48} follows. Therefore, the conclusion holds.

We now turn to proof of \eqref{infinity}.
For any $r>0$, we have
\begin{equation*}
\int_{\mathbb{R}^n}\frac{1}{|x-y|^{n-2}}\frac{1}{1+|y-x_{0}|^{(n-2)(p-\epsilon)}}dy=\big(\int_{B(0,r)}+\int_{\mathbb{R}^n\setminus B(0,r)}\big)\frac{(1+|y|^{(n-2)(p-\epsilon)})^{-1}}{|x-x_0-y|^{n-2}}dy.
\end{equation*}
 If $x-x_0\in \mathbb{R}^n\setminus B(0,2r)$, we have
\begin{equation*}
\begin{split}
\int_{B(0,r)}\frac{1}{|x-x_0-y|^{n-2}}\frac{1}{1+|y|^{(n-2)(p-\epsilon)}}dy&
\leq\int_{B(0,r)}\frac{1}{|y|^{n-2}}\frac{1}{1+|y|^{{(n-2)(p-\epsilon)}}} dy\\&
\leq\big\|\frac{1}{|y|^{n-2}}\big\|_{L^{\zeta}(B(0,r))}\Big\|\frac{1}{1+|y|^{{(n-2)(p-\epsilon)}}}\Big\|_{L^{\frac{\zeta}{\zeta-1}}(B(0,r))}\\&<\infty,
\end{split}
\end{equation*}
by H\"{o}lder inequality and where $1<\zeta<\frac{n}{n-2}$.
If $x-x_0\in B(0,2r)$,  we get
\begin{equation*}
\begin{split}
&\int_{B(0,r)}\frac{1}{|x-x_0-y|^{n-2}}\frac{1}{1+|y|^{(n-2)(p-\epsilon)}}dy\\&
\leq\int_{B(0,r)}\frac{1}{|y|^{n-2}}\frac{1}{1+|y|^{(n-2)(p-\epsilon)}}dy+\int_{B(x-x_0,3r)}\frac{1}{|x-x_0-y|^{n-2}}\frac{1}{1+|x-x_0-y|^{(n-2)(p-\epsilon)}}dy\\&
\leq\big\|\frac{1}{|y|^{n-2}}\big\|_{L^{\zeta}(B(0,3r))}\Big\|\frac{1}{1+|y|^{{(n-2)(p-\epsilon)}}}\Big\|_{L^{\frac{\zeta}{\zeta-1}}(B(0,3r))}<\infty,
\end{split}
\end{equation*}
where $1<\zeta<\frac{n}{n-2}$.
Furthermore, we find
\begin{equation*}
\begin{split}
&\int_{\mathbb{R}^n\setminus B(0,r)}\frac{1}{|x-x_0-y|^{n-2}}\frac{1}{1+|y|^{(n-2)(p-\epsilon)}}dy\\&
=\big(\int_{(\mathbb{R}^n\setminus B(0,r))\cap B(x-x_0,r)} +\int_{(\mathbb{R}^n\setminus B(0,r))\cap(\mathbb{R}^n\setminus B(x-x_0,r))}\big)\frac{(1+|y|^{(n-2)(p-\epsilon)})^{-1}}{|x-x_0-y|^{n-2}}dy.
\end{split}
\end{equation*}
Similar to the previous argument, we deduce that
\begin{equation}\label{q-1}
\begin{split}
\int_{(\mathbb{R}^n\setminus B(0,r))\cap B(x-x_0,r)}&\frac{1}{|x-x_0-y|^{n-2}}\frac{1}{1+|y|^{(n-2)(p-\epsilon)}}dy\\&\leq\frac{C}{r^{(n-2)(p-\epsilon)}}\int_{ B(x-x_0,r)}\frac{1}{|x-x_0-y|^{n-2}}dy<\infty,
\end{split}
\end{equation}
and
\begin{equation}\label{q1pie}
\begin{split}
\int_{(\mathbb{R}^n\setminus B(0,r))\cap( \mathbb{R}^n\setminus B(x-x_0,r))}&\frac{1}{|x-x_0-y|^{n-2}}\frac{1}{1+|y|^{(n-2)(p-\epsilon)}}dy\\&\leq\frac{C}{r^{n-2}}
\int_{\mathbb{R}^n\setminus B(0,r))}\frac{1}{|y|^{(n-2)(p-\epsilon)}}dy<\infty.
\end{split}
\end{equation}
Thus, we conclude \eqref{infinity} by  the previous estimates.
\end{proof}
\begin{lem}\label{thm:existenceofweaksolution}
It holds that
\begin{equation}\label{lamta}
|\lambda_{\epsilon}^{-\epsilon}-1|=O(\lambda_{\epsilon}^{-\frac{n-2}{2}}\log\frac{1}{\lambda_{\epsilon}})\quad\mbox{as}\quad\epsilon\rightarrow0.
\end{equation}
\end{lem}
\begin{proof}
It can be derived by Lemmma \ref{thm:uniquenessofweaksolution} and the theorem of mean.
\end{proof}

\subsection{Proof of Theorem \ref{consequence}}

We shall give the proof of Theorem \ref{consequence}.
\begin{proof}[Proof of Theorem \ref{consequence}]
We have
\begin{equation}\label{laplacian}
-\Delta(\|u_{\epsilon}\|_{L^{\infty}(\Omega)}u_{\epsilon})=\|u_{\epsilon}\|_{L^{\infty}(\Omega)}\big(|x|^{-{(n-2)}}\ast u_\epsilon^{p-\epsilon}\big)u_\epsilon^{p-1-\epsilon}\quad\mbox{in}\quad \Omega.
\end{equation}
We integrate the right-hand side of \eqref{laplacian}
$$\int_{\Omega}\int_{\Omega}\|u_{\epsilon}\|_{L^{\infty}(\Omega)}\frac{u_\epsilon^{p-\epsilon}(y)u_\epsilon^{p-1-\epsilon}(x)}{|x-y^{n-2}} dxdy=\frac{\lambda_{\epsilon}^{\frac{4(n+2)}{4-(n-2)\epsilon}}}{\lambda_{\epsilon}^{n+2}}\lambda_{\epsilon}^{-\frac{4(n-2)\epsilon}{4-(n-2)\epsilon}}\int_{\Omega_{\epsilon}}\int_{\Omega_{\epsilon}}\frac{v_\epsilon^{p-\epsilon}(y)v_\epsilon^{p-1-\epsilon}(x)}{|x-y|^{n-2}} dxdy.$$
Therefore, combining \eqref{cU} by dominated convergence, Lemma \ref{thm:existenceofweaksolution} and the following identity (see \cite{gyz})
\begin{equation}\label{p1-00}
|x|^{-(n-2)}\ast W^{p}
=\int_{\mathbb{R}^n}\frac{W^{p}(y)}{|x-y|^{n-2}}dy
=\widetilde{\alpha}_{n}W^{2^{\ast}-p}(x),
\end{equation}
with $\widetilde{\alpha}_{n}=I(\gamma)S^{\frac{(2-n)}{8}}C_{n}^{\frac{2-n}{8}}[n(n-2)]^{\frac{n-2}{4}}$, we get
\begin{equation*}
\begin{split}
\lim\limits_{\epsilon\rightarrow0^{+}}\int_{\Omega}\|u_{\epsilon}\|_{L^{\infty}(\Omega)}\big(|x|^{-{(n-2)}}\ast u_\epsilon^{p-\epsilon}\big)u_\epsilon^{p-1-\epsilon}&=\int_{\mathbb{R}^n}\int_{\mathbb{R}^n}\frac{W^{p}(y)W^{p-1}(x)}{|x-y|^{n-2}}dxdy\\&
=\widetilde{\alpha}_{n}\big\|W\big\|^{2^{\ast}-1}_{L^{2^{\ast}-1}(\mathbb{R}^n)}<\infty.
\end{split}
\end{equation*}
Here and several times in the sequel, we use the elementary identity \eqref{p1-00}. Due to \eqref{infinity} and using Lemma \ref{thm:existenceofweaksolution} again, we obtain
\begin{equation*}
\begin{split}
\|u_{\epsilon}\|_{L^{\infty}(\Omega)}&\big(|x|^{-{(n-2)}}\ast u_\epsilon^{p-\epsilon}\big)u_\epsilon^{p-1-\epsilon}\\&\leq
M\lambda_{\epsilon}^{-(n-2)(2p-1)+\frac{4(n+2)}{4-(n-2)\epsilon}}\cdot\lambda_{\epsilon}^{2(n-2)[1-\frac{2}{4-(n-2)\epsilon}]}\frac{1}{|x-x_{0}|^{(n-2)(p-\epsilon)}}
\end{split}
\end{equation*}
for $x\neq x_0$ and some $M>0$. It is noticing that
$$-(n-2)(2p-1)+\frac{4(n+2)}{4-(n-2)\epsilon}<0.$$
Thus, combining Lemma \ref{infinite} and lemma \ref{finite}, we conclude that
\begin{equation}\label{LWUQ}
\|u_{\epsilon}\|_{L^{\infty}(\Omega)}\big(|x|^{-{(n-2)}}\ast u_\epsilon^{p-\epsilon}\big)u_\epsilon^{p-1-\epsilon}\rightarrow0\quad\mbox{for}\quad x\neq x_0.
\end{equation}
From \eqref{laplacian} and \eqref{LWUQ} we deduce that
$$
-\Delta(\|u_{\epsilon}\|_{L^{\infty}(\Omega)}u_{\epsilon})\rightarrow\widetilde{\alpha}_{n}\big\|W\big\|^{2^{\ast}-1}_{L^{2^{\ast}-1}(\mathbb{R}^n)}\delta_{x=x_0}
$$
in the sense of distributions in $\Omega$. Coming back to Lemma \ref{regular}, we set $\omega$ as a neighborhood of $\partial\Omega$ and $x_0\not\in\omega$.
Then, inequality \eqref{regu} tells us that
\begin{equation*}
\begin{split}
\big\|\|u_{\epsilon}\|_{L^{\infty}(\Omega)}u_{\epsilon}\big\|_{C^{1,\alpha}(\omega^{\prime})}&\leq C\big\|\|u_{\epsilon}\|_{L^{\infty}(\Omega)}\big(|x|^{-{(n-2)}}\ast u_\epsilon^{p-\epsilon}\big)u_\epsilon^{p-1-\epsilon}\big\|_{L^1(\Omega)}\\&+C\big\|\|u_{\epsilon}\|_{L^{\infty}(\Omega)}\big(|x|^{-{(n-2)}}\ast u_\epsilon^{p-\epsilon}\big)u_\epsilon^{p-1-\epsilon}\big\|_{L^\infty(\omega)}\hspace{2mm}\mbox{in}\hspace{2mm}C^{1,\alpha}(\omega)\hspace{2mm}\mbox{as}\epsilon\rightarrow0.
\end{split}
\end{equation*}
In conclusion, we have
$$ \|u_{\epsilon}\|_{L^{\infty}(\Omega)}u_{\epsilon}\rightarrow\widetilde{\alpha}_{n}\big\|W\big\|^{2^{\ast}-1}_{L^{2^{\ast}-1}(\mathbb{R}^n)}G(x,x_0)
\hspace{2mm}\mbox{in}\hspace{2mm}C^{1,\alpha}(\omega)
$$
for any neighborhood $\omega$ of $\partial\Omega$, not containing $x_0$. The result follows.
\end{proof}

\subsection{Proof of Theorem \ref{Figalli2}}
In this subsection, we are position to prove Theorem \ref{Figalli2}.
\begin{proof}[Proof of Theorem \ref{Figalli2}]
Recall that the Pohizave identity in Lemma \ref{lem:hardyinequality} and applying the identity \eqref{pohozave} to functions $\|u_{\epsilon}\|_{L^{\infty}(\Omega)}u_{\epsilon}$,  we obtain
\begin{equation}\label{epsilonfrac}
\begin{split}
\frac{(n-2)^2}{n+2-\epsilon(n-2)}\cdot\epsilon&\|u_{\epsilon}\|^2_{L^{\infty}(\Omega)}\int_{\Omega}(|x|^{-(n-2)} \ast u_{\epsilon}^{p-\epsilon})u_{\epsilon}^{p-\epsilon}dx\\&=\int_{\partial\Omega}\big(\|u_{\epsilon}\|_{L^{\infty}(\Omega)}\nabla u_{\epsilon}, \nu\big)\big(\|u_{\epsilon}\|_{L^{\infty}(\Omega)}\nabla u_{\epsilon}, \nu\big)(\nu,x-z)dS_{x}.
\end{split}
\end{equation}
In view of Theorem \ref{consequence} and \eqref{SHL}, and taking the limit in \eqref{epsilonfrac}, we obtain
\begin{equation*}
\lim\limits_{\epsilon\rightarrow0}\epsilon\|u_{\epsilon}\|^2_{L^{\infty}(\Omega)}=\widetilde{\alpha}_{n}^2\frac{n+2}{(n-2)^2}\big(\frac{1}{C_{HLS}}\big)^{\frac{n+2}{4}}\big\|W\big\|^{2(2^{\ast}-1)}_{L^{2^{\ast}-1}(\mathbb{R}^n)}
 \int_{\partial\Omega}\frac{\partial G(x,x_0)}{\partial \nu}\frac{\partial G(x,x_0)}{\partial \nu}(\nu,x-x_0)dS_{x}.
\end{equation*}
Since we have the following identity (see \cite{BP})
\begin{equation}\label{GX0}
\int_{\partial\Omega}\frac{\partial G(x,x_0)}{\partial \nu}\frac{\partial G(x,x_0)}{\partial \nu}(\nu,x-x_0)dS_{x}=-(n-2)H(x_0,x_0).
\end{equation}
Moreover, it is easy to check from the maximum principle that $H(x,x)<0$ for any $x\in\Omega$.
Consequently, the first result of theorem follows by the above identities. Hence, via this result and Theorem \ref{consequence}, we get limit of $u_{\epsilon}\epsilon^{-\frac{1}{2}}$ and
as desired.
\end{proof}

\subsection{Proof of Theorem \ref{Figalli}}

We are now ready to prove Theorem \ref{Figalli}.
\begin{proof}[Proof of Theorem \ref{Figalli}]
From Lemma \ref{regular} we deduce that
$$\|u_\epsilon\|_{C^{1,\alpha}(\omega)}\leq C\big(\|(|x|^{-{(n-2)}}\ast u_\epsilon^{p-\epsilon})u_\epsilon^{p-1-\epsilon}\|_{L^1(\Omega)}+\|(|x|^{-{(n-2)}}\ast u_\epsilon^{p-\epsilon})u_\epsilon^{p-1-\epsilon}\|_{L^\infty(\omega)}\big).$$
Coupling this bound and \eqref{iy}, we conclude that the first result of Theorem \ref{Figalli}.

Furthermore, it holds that
\begin{equation*}
\begin{split}
\int_{\mathbb{R}^n}|\nabla u_\epsilon|^2&=\int_{\mathbb{R}^n}\int_{\mathbb{R}^n}\frac{u_\epsilon^{p}(x)u_\epsilon^{p}(y)}{|x-y|^{n-2}}dxdy
\\&\rightarrow\int_{\mathbb{R}^n}\int_{\mathbb{R}^n}\frac{W^{p}(x)W^{p}(y)}{|x-y|^{n-2}}dxdy=\big(C_{HLS}\big)^{\frac{n+2}{4}}\quad\mbox{as}\quad\epsilon\rightarrow0.
\end{split}
\end{equation*}
From here and the first result of Theorem \ref{Figalli}, we have
$$|\nabla u_{\epsilon}|^2\rightarrow\big(C_{HLS}\big)^{\frac{n+2-\epsilon(n-2)}{4-\epsilon(n-2)}}\delta_{x_0}\quad\mbox{as}\quad\epsilon\rightarrow0$$
in the sense of distributions.

Finally, we note that from \eqref{pohozave}, the equality $\int_{\partial\Omega}(\frac{\partial u_{\epsilon}}{\partial\nu})^2\nu dS=0$.
In the limit, we obtain
\begin{equation}\label{0X}
\int_{\partial\Omega}(\nabla G(x,x_0),\nu)(\nabla G(x,x_0),\nu)\nu dS_{x}=0\quad\mbox{for every}\quad x_{0}\in\Omega.
\end{equation}
By exploiting the following identity (see \cite{BP})
\begin{equation}\label{GX0X}
\int_{\partial\Omega}(\nabla G(x,x_0),\nu)(\nabla G(x,x_0),\nu)\nu dS_{x}=-\nabla\phi(x_0)\quad\mbox{for every}\quad x_{0}\in\Omega.
\end{equation}
Then, the conclusion follows and hence Theorem \ref{Figalli} is proved.
\end{proof}
\section{Proof of Theorem \ref{Brezi-type}}\label{theorem1-7}
In this section, we are devoted to show that Theorem \ref{Brezi-type}.
As in the previous Pohozaev-type identity \eqref{pohozave}, we have
\begin{equation*}
\frac{2n-\mu}{2p_{\mu}}\int_{\Omega}(\frac{1}{|x|^{\mu}} \ast u_{\epsilon}^{p_{\mu}})u_{\epsilon}^{p_{\mu}}dx+\frac{2-n}{2}\int_{\Omega}|\nabla u_{\epsilon}|^2dx+\frac{n}{2}\epsilon\int_{\Omega}|u_{\epsilon}|^2dx. =\frac{1}{2}\int_{\partial\Omega}\Big(\frac{\partial u_{\epsilon}}{\partial\nu}\Big)^{2}(x-z,n)dS_x,
\end{equation*}
whence
\begin{equation}\label{brety}
\epsilon\int_{\partial\Omega}\Big(\frac{\partial u_{\epsilon}}{\partial \nu}\Big)^2(\nu,x-z)dS_{x}=\epsilon\int_{\Omega}u_{\epsilon}^{2} dx.
\end{equation}
In what follows $x_\epsilon\in\Omega$ and the number $\lambda_{\epsilon}>0$ are given by
\begin{equation}\label{miuinfini}
\lambda_{\epsilon}^{\frac{n-2}{2}}=\|u_{\epsilon}\|_{L^{\infty}(\Omega)}=u_\epsilon(x_\epsilon).
\end{equation}
Observe that, coupling
rescaled function
 $$v_{\epsilon}(x)=\lambda_{\epsilon}^{-\frac{n-2}{2}}u_{\epsilon}(\lambda_{\epsilon}^{-1}x+x_{\epsilon}),$$
 and the Kelvin transform $w_{\epsilon}$ of $v_{\epsilon}$, we have
\begin{equation*}
\begin{cases}
-\Delta v_{\epsilon}(x)=\big(|x|^{-{\mu}}\ast v_\epsilon^{p_{\mu}}\big)v_\epsilon^{p_{\mu}-1}+\epsilon v_\epsilon\hspace{6mm}\mbox{in}\hspace{2mm} \Omega_{\epsilon}:=\{x:\lambda_{\epsilon}^{-1}x+x_{\epsilon}\in\Omega\},\\
0\leq v_{\varepsilon}(x)\leq1~~~~~~~~~~~~~~~~~~~~~~~~~~~~~~~~~~~~~~~~~~~~~\text{in}~~\Omega_{\varepsilon},
\\
v_{\varepsilon}(x)=0~~~~~~~~~~~~~~~~~~~~~~~~~~~~~~~~~~~~~~~~~~~~~~~~~~~~\text{on}~~\partial\Omega_{\varepsilon},\\
v_{\varepsilon}(0)=\max\limits_{x\in\Omega_{\varepsilon}}v_{\varepsilon}(x)=1,
\end{cases}
\end{equation*}
and
\begin{equation*}
\begin{cases}
-\Delta w_{\epsilon}(x)
=\Big(\int_{\Omega_{\epsilon}^{\ast}}\frac{w_\epsilon^{p_\mu}(y)}{|x-y|^{\mu}} dy\Big)w_\epsilon^{{p_\mu}-1}+\epsilon\frac{1}{\lambda_\epsilon^2|x|^{4}}w_\epsilon \hspace{4mm}\mbox{in}\hspace{2mm} \Omega_{\epsilon}^{\ast}:=\{x: |x|^{-1}\in\Omega_{\epsilon}\},\\
w_{\varepsilon}(x)=0~~~~~~~~~~~~~~~~~~~~~~~~~~~~~~~~~~~~\hspace{2mm}\hspace{10mm}\hspace{8mm}\hspace{8mm}\text{on}~~\partial\Omega_{\epsilon}^{\ast}
\end{cases}
\end{equation*}
for $n\geq3$, $p_\mu=\frac{2n-\mu}{n-2}\geq2$ and $\mu\in(0,n)$.
Then we proceed similarly to Lemma \ref{cWU}. There exists a constant $c>0$ independently of $\epsilon>0$ provided $\epsilon$ is sufficiently small, such that
\begin{equation}\label{bredecay}
v_{\epsilon}(x)\leq c\big(\frac{1}{1+|x|^{2}}\big)^{\frac{n-2}{2}}\quad\mbox{in}\quad\Omega_{\epsilon}.
\end{equation}
Therefore, similarly to the computation of Theorem \ref{consequence}, using \eqref{bredecay} by dominated convergence we also deduce that
\begin{equation}\label{ndengyu4} \|u_{\epsilon}\|_{L^{\infty}(\Omega)}u_{\epsilon}\rightarrow\widetilde{\alpha}_{n}\big\|W\big\|^{2^{\ast}-1}_{L^{2^{\ast}-1}(\mathbb{R}^n)}G(x,x_0)
\hspace{2mm}\mbox{in}\hspace{2mm}C^{1,\alpha}(\omega)
\end{equation}
for any neighborhood $\omega$ of $\partial\Omega$, not containing $x_0$.
Recalling \eqref{brety}, and by \eqref{ndengyu4} yields
\begin{equation*}
\lim\limits_{\epsilon\rightarrow0}\epsilon\|u_{\epsilon}\|^2_{L^{\infty}(\Omega)}\lambda_{\epsilon}^{-2}\int_{\Omega_{\epsilon}}v_{\epsilon}^2(y)dy=\frac{\widetilde{\alpha}_{n}^2}{2}\big\|W\big\|^{2(2^{\ast}-1)}_{L^{2^{\ast}-1}(\mathbb{R}^n)}
 \int_{\partial\Omega}\frac{\partial G(x,x_0)}{\partial \nu}\frac{\partial G(x,x_0)}{\partial \nu}(\nu,x-x_0)dS_{x}.
\end{equation*}
Consequently, a direct calculation shows that
\begin{equation*}
\lim\limits_{\epsilon\rightarrow0}\epsilon\big\|u_{\epsilon}\big\|^{\frac{2(n-4)}{n-2}}_{L^{\infty}(\Omega)}\int_{\Omega_{\epsilon}}v_{\epsilon}^2(y)dy
=\frac{(n-2)\widetilde{\alpha}_{n}^2}{2}\big\|W\big\|^{2(2^{\ast}-1)}_{L^{2^{\ast}-1}(\mathbb{R}^n)}
 |\phi(x_0)|.
\end{equation*}
where we have used \eqref{miuinfini} and identity \eqref{GX0}. Consequently, combining this equality and \eqref{WW} yields the conclusion. On the other hand, the conclusions of Theorem \ref{Figalli} for problem \eqref{ele-2} follows by
Lemma \ref{regular} and \eqref{0X}-\eqref{GX0X}.

\section{Proofs of Theorems \ref{maximum}-\ref{asymptotic}}\label{maximumpoint}
We denote by $PW_{\xi,\lambda}$ the projection of a function $W_{\xi,\lambda}$ onto  $H_{0}^{1}(\Omega)$, namely,
\begin{equation}\label{eq2.13}
		\Delta PW[\xi,\lambda] =\Delta W[\xi,\lambda]\quad\mbox{in}\quad\Omega,\quad\quad
		PW[\xi,\lambda]=0\quad\mbox{on}\quad\partial\Omega.
\end{equation}
To prove Theorem \ref{asymptotic}, we will use $W[\xi,\lambda]$ (see \eqref{defU}) as an approximate solution and so we write
$$v_{\epsilon}=\lambda_{\epsilon}^{-(n-2)/2}PW[0,1]+\lambda_{\epsilon}^{-(n-2)}\phi_{\epsilon}.$$
\subsection{An upper bound for $S_{HL}^{\epsilon}$}
For the simplicity of notations, we write $W(x)$ and $W_1(x)$ instead of $W[0,1](x)$ and $W[x_1,\lambda_{\epsilon}]$ in the sequel, respectively.
We deduce the following important estimate.
\begin{lem}\label{shlep}
 Suppose that the assumptions of Theorem \ref{Figalli} be satisfied and $x_{\epsilon}\rightarrow x_0$ as $\epsilon\rightarrow0$. Then for any point $x_1\in\Omega$, it holds that
\begin{equation*}
\frac{S_{HL}^{\epsilon}}{\lambda_{\epsilon}^{\frac{(n-2)\epsilon}{p-\epsilon}}}\leq C_{HLS}+\lambda_{\varepsilon}^{-(n-2)}\bigg[\tilde{H}(x_1,x_1)C_{HLS}^{-\frac{n-2}{4}}\mathcal{F}^{\prime}+\frac{2}{p}C_{HLS}^{-\frac{n-2}{4}}\mathcal{D}^{\prime}-\frac{n+2}{4p^2}C_{HLS}\mathcal{E}^{\prime}\bigg]+o(\lambda_{\varepsilon}^{-(n-2)}).
\end{equation*}
Here $\tilde{H}(x,y)=-(n-2)\omega_nH(x,y)$ and
$$\mathcal{D}^{\prime}:=\kappa(n,x_0)\int_{\mathbb{R}^n}\int_{\mathbb{R}^n}\frac{W^{p}(y)W^{p}(z)\log W}{|y-z|^{n-2}}dydz,\hspace{2mm}\mathcal{E}^{\prime}:=\kappa(n,x_0)\log C_{HLS},$$
$$\mathcal{F}^{\prime}:=\tilde{c}_{n,\mu}\int_{\mathbb{R}^n}\int_{\mathbb{R}^n}\frac{W^{p}(z)W^{p-1}(y)}{|y-z|^{n-2}}dydz,\hspace{2mm}\kappa(n,x_0):=\widetilde{\alpha}_{n}^2\frac{n+2}{n-2}\big[\frac{1}{C_{HLS}}\big]^{\frac{n+2}{4}}\big\|W\big\|^{2(2^{\ast}-1)}_{L^{2^{\ast}-1}(\mathbb{R}^n)}|\phi(x_0)|.$$
\end{lem}
\begin{proof}
For any $x_1\in\Omega$ and we set $u(x)=PW_{1}(x)$.
Then, since \eqref{eq2.13}, Lemma \ref{Lem2.2} and $\Omega_{\epsilon}:=\lambda_{\epsilon}(\Omega-x_{1})$,
a direct calculation shows that
\begin{equation}\label{tiduuo}
\begin{split}
\int_{\Omega}&|\nabla u|^2dx=\int_{\Omega}|\nabla PW_{1}(x)|^2dx\\&=\int_{\Omega}\big(|x|^{-(n-2)}\ast W_{1}^{p}\big)W_{1}^{p-1}(x)\big[W_{1}(x)-\tilde{c}_{n,\mu}\lambda_{\epsilon}^{-\frac{n-2}{2}}\tilde{H}(x_1,x)-f_{\lambda_{\epsilon}},
\big]dx\\&
=\int_{\mathbb{R}^n}\int_{\mathbb{R}^n}\frac{W^{p}(y)W^{p}(z)}{|y-z|^{n-2}} dydz-\frac{\tilde{c}_{n,\mu}}{\lambda_{\epsilon}^{n-2}}\int_{\Omega_{\epsilon}}\int_{\Omega_{\epsilon}}\frac{W^{p}(z)W^{p-1}(y)\tilde{H}(x_1,x_1+\lambda_{\epsilon}^{-1}y)}{|y-z|^{n-2}} dydz+O(\frac{1}{\lambda_{\epsilon}^{n}})\\&
=C_{HLS}^{\frac{n+2}{4}}-\frac{\tilde{c}_{n,\mu}}{\lambda_{\epsilon}^{n-2}}\int_{\Omega_{\epsilon}}\int_{\Omega_{\epsilon}}\frac{W^{p}(z)W^{p-1}(y)\tilde{H}(x_1,x_1+\lambda_{\epsilon}^{-1}y)}{|y-z|^{n-2}} dydz+O(\frac{1}{\lambda_{\epsilon}^{n}}).
\end{split}
\end{equation}
Note that,
\begin{equation*}
\begin{split}
\int_{\Omega_{\epsilon}}\int_{\Omega_{\epsilon}}\frac{W^{p}(z)W^{p-1}(y)}{|y-z|^{n-2}}&=
\int_{\mathbb{R}^n}\int_{\mathbb{R}^n}\frac{W^{p}(z)W^{p-1}(y)}{|y-z|^{n-2}} dydz-2\int_{\mathbb{R}^n\setminus\Omega_{\epsilon}}\int_{\Omega_{\epsilon}}\frac{W^{p}(z)W^{p-1}(y)}{|y-z|^{n-2}}
\\&\hspace{4mm}-\int_{\mathbb{R}^n\setminus\Omega_{\epsilon}}\int_{\mathbb{R}^n\setminus\Omega_{\epsilon}}\frac{W^{p}(z)W^{p-1}(y)}{|y-z|^{n-2}}.
\end{split}
\end{equation*}
Then using \eqref{p1-00}, Lemma \ref{Gx} and the fact that $|\tilde{H}(x_1,x_1+\lambda_{\epsilon}^{-1}y)-\tilde{H}(x_1,x_1)|\leq C\lambda_{\epsilon}^{-1}y$, we deduce that
\begin{equation}\label{HX1X1}
\begin{split}
\int_{\Omega_{\epsilon}}\int_{\Omega_{\epsilon}}&\frac{W^{p}(z)W^{p-1}(y)\tilde{H}(x_1,x_1+\lambda_{\epsilon}^{-1}y)}{|y-z|^{n-2}} dydz\\&=\int_{\Omega_{\epsilon}}\int_{\Omega_{\epsilon}}\frac{W^{p}(z)W^{p-1}(y)}{|y-z|^{n-2}} [\tilde{H}(x_1,x_1+\lambda_{\epsilon}^{-1}y)-\tilde{H}(x_1,x_1)]dydz\\&\hspace{4mm}+\tilde{H}(x_1,x_1)\int_{\Omega_{\epsilon}}\int_{\Omega_{\epsilon}}\frac{W^{p}(z)W^{p-1}(y)}{|y-z|^{n-2}} dydz\\&
=\tilde{H}(x_1,x_1)\int_{\mathbb{R}^n}\int_{\mathbb{R}^n}\frac{W^{p}(z)W^{p-1}(y)}{|y-z|^{n-2}} dydz
+O\Big(\lambda_{\epsilon}^{-1}
\int_{\Omega_{\epsilon}}W^{2^{\ast}-1}|y|(y)dy\Big)
\\&=\tilde{H}(x_1,x_1)\int_{\mathbb{R}^n}\int_{\mathbb{R}^n}\frac{W^{p}(z)W^{p-1}(y)}{|y-z|^{n-2}} dydz+O(\frac{1}{\lambda_{\epsilon}}).
\end{split}
\end{equation}
Consequently, combining this estimate and \eqref{tiduuo} entails that
\begin{equation}\label{womiga}
\begin{split}
\int_{\Omega}&|\nabla u|^2dx=C_{HLS}^{\frac{n+2}{4}}-\frac{\tilde{c}_{n,\mu}}{\lambda_{\epsilon}^{n-2}}\tilde{H}(x_1,x_1)\int_{\mathbb{R}^n}\int_{\mathbb{R}^n}\frac{W^{p}(z)W^{p-1}(y)}{|y-z|^{n-2}} dydz+o(\frac{1}{\lambda_{\epsilon}^{n-2}})=:I_{\epsilon}.
\end{split}
\end{equation}

On the other hand, in view of $PW[x_1,\lambda_{\epsilon}]=W[x_1,\lambda_{\epsilon}]-\psi[x_1,\lambda_{\epsilon}]>0$, one has $0\leq\frac{\psi[x_1,\lambda_{\epsilon}]}{W[x_1,\lambda_{\epsilon}]}\leq1$.
As the result, combining the estimate
$$PW_1^{p-\epsilon}=W_1^{p-\epsilon}+O(W_1^{p-1-\epsilon}\psi[x_1,\lambda_{\epsilon}]),\quad\|\psi[x_1,\lambda_{\epsilon}]\|_{L^{\infty}(\Omega)}=\lambda_{\epsilon}^{-\frac{n-2}{2}} (dist(x_1,\partial\Omega))^{n-2},$$
and Lemma \ref{Lem2.2}, we have that
\begin{equation*}
\begin{split}
\lambda_{\epsilon}^{\frac{(n-2)\epsilon}{p-\epsilon}}&\Big[\int_{\Omega}(|x|^{-(n-2)} \ast PW_1^{p-\epsilon})PW_1^{p-\epsilon} dx\Big]^{\frac{1}{p-\epsilon}}\\
=&\lambda_{\epsilon}^{\frac{(n-2)\epsilon}{p-\epsilon}}\Big[\int_{\Omega}\int_{\Omega}\big(|x-y|^{-(n-2)} \big[W_1(y)-\frac{\tilde{c}_{n,\mu}}{\lambda_\epsilon^{\frac{n-2}{2}}}H(x_1,y)-f_{\lambda_\epsilon}(y)\big]^{p-\epsilon}\big)\\&\times\big[W_1(x)-\frac{\tilde{c}_{n,\mu}}{\lambda_\epsilon^{\frac{n-2}{2}}}H(x_1,x)-f_{\lambda_\epsilon}(x)\big]^{p-\epsilon} dxdy\Big]^{\frac{1}{p-\epsilon}}\\
=&\bigg[\lambda_{\epsilon}^{(n-2)\epsilon-(n+2)}\int_{\Omega_{\epsilon}}\int_{\Omega_{\epsilon}}|y-z|^{-(n-2)}\Big(\lambda_\epsilon^{\frac{n-2}{2}}W(z)-\frac{\tilde{c}_{n,\mu}}{\lambda_\epsilon^{\frac{n-2}{2}}}\tilde{H}(x_1,x_1+\lambda_{\epsilon}^{-1}z)-f_{\lambda_\epsilon}(z)\Big)^{p-\epsilon}
\\&\times\Big(\lambda_\epsilon^{\frac{n-2}{2}}W(y)-\frac{\tilde{c}_{n,\mu}}{\lambda_\epsilon^{\frac{n-2}{2}}}\tilde{H}(x_1,x_1+\lambda_{\epsilon}^{-1}y)-f_{\lambda_\epsilon}(y)\Big)^{p-\epsilon} dydz\bigg]^{\frac{1}{p-\epsilon}}\\
=&\bigg[\int_{\Omega_{\epsilon}}\int_{\Omega_{\epsilon}}\frac{W^{p-\epsilon}(y)W^{p-\epsilon}(z)}{|y-z|^{n-2}} dydz-p\frac{\tilde{c}_{n,\mu}}{\lambda_{\epsilon}^{n-2}}\int_{\Omega_{\epsilon}}\int_{\Omega_{\epsilon}}\frac{W^{p}(z)W^{p-1}(y)H(x_1,x_1+\lambda_{\epsilon}^{-1}y)}{|y-z|^{n-2}} dydz\\&\hspace{4mm}-p\frac{\tilde{c}_{n,\mu}}{\lambda_{\epsilon}^{n-2}}\int_{\Omega_{\epsilon}}\int_{\Omega_{\epsilon}}\frac{W^{p-1}(z)H(x_1,x_1+\lambda_{\epsilon}^{-1}z)W^{p}(y)}{|y-z|^{n-2}} dydz\\&
+p^2\frac{\tilde{c}_{n,\mu}^2}{\lambda_{\epsilon}^{2(n-2)}}\int_{\Omega_{\epsilon}}\int_{\Omega_{\epsilon}}\frac{W^{p-1}(z)H(x_1,x_1+\lambda_{\epsilon}^{-1}z)W^{p-1}(y)H(x_1,x_1+\lambda_{\epsilon}^{-1}y)}{|y-z|^{n-2}} dydz+o(\frac{1}{\lambda_{\epsilon}^{n-2}})\bigg]^{\frac{1}{p-\epsilon}}.
\end{split}
\end{equation*}
Recalling the previous computation in \eqref{HX1X1}, we have
\begin{equation*}
\begin{split}
&\lambda_{\epsilon}^{\frac{(n-2)\epsilon}{p-\epsilon}}\Big[\int_{\Omega}(|x|^{-(n-2)} \ast PW_1^{p-\epsilon})PW_1^{p-\epsilon} dx\Big]^{\frac{1}{p-\epsilon}}\\&
=\bigg[\int_{\mathbb{R}^n}\int_{\mathbb{R}^n}\frac{W^{p-\epsilon}(y)W^{p-\epsilon}(z)}{|y-z|^{n-2}} dydz-p\frac{\tilde{c}_{n,\mu}}{\lambda_{\epsilon}^{n-2}}\tilde{H}(x_1,x_1)\int_{\mathbb{R}^n}\int_{\mathbb{R}^n}\frac{W^{p}(z)W^{p-1}(y)}{|y-z|^{n-2}} dydz\\&
\hspace{4mm}-p\frac{\tilde{c}_{n,\mu}}{\lambda_{\epsilon}^{n-2}}\tilde{H}(x_1,x_1)\int_{\mathbb{R}^n}\int_{\mathbb{R}^n}\frac{W^{p-1}(z)W^{p}(y)}{|y-z|^{n-2}} dydz+o(\frac{1}{\lambda_{\epsilon}^{n-2}})\bigg]^{\frac{1}{p-\epsilon}}
\end{split}
\end{equation*}
as $\epsilon\rightarrow0$, where the equality holds since
\begin{equation*}
\begin{split}
\int_{\Omega_{\epsilon}}\int_{\Omega_{\epsilon}}&\frac{W^{p-1}(z)\tilde{H}(x_1,x_1+\lambda_{\epsilon}^{-1}z)W^{p}(y)}{|y-z|^{n-2}} dydz
\\&=\tilde{H}(x_1,x_1)\int_{\mathbb{R}^n}\int_{\mathbb{R}^n}\frac{W^{p}(z)W^{p-1}(y)}{|y-z|^{n-2}} dydz+O(\frac{1}{\lambda_{\epsilon}}),
\end{split}
\end{equation*}
and
\begin{equation*}
\begin{split}
\frac{1}{\lambda_{\epsilon}^{2(n-2)}}\int_{\Omega_{\epsilon}}\int_{\Omega_{\epsilon}}&\frac{W^{p-1}(z)|H|W^{p-1}(y)|H|}{|y-z|^{n-2}} dydz=
\begin{cases}
O\big(\frac{1}{\lambda_{\epsilon}^{2(n-2)}}\big), \quad\quad\hspace{1.5mm}\mbox{ when } n=3,\\
O\big(\frac{\log\lambda_{\epsilon}}{\lambda_{\epsilon}^{2(n-2)}}\big), \quad\quad\hspace{2mm}\mbox{ when } n=4,\\
O\big(\frac{1}{\lambda_{\epsilon}^{4}}\big), \hspace{1mm}\quad\quad\hspace{8mm}\mbox{ when } n=5.
\end{cases}
\end{split}
\end{equation*}
Moreover,
\begin{equation}\label{taileextend}
\begin{split}
\int_{\mathbb{R}^n}\int_{\mathbb{R}^n}\frac{W^{p-\epsilon}(y)W^{p-\epsilon}(z)}{|y-z|^{n-2}} dydz&=\int_{\mathbb{R}^n}\int_{\mathbb{R}^n}\frac{W^{p}(y)W^{p}(z)}{|y-z|^{n-2}} dydz+o(\lambda_{\epsilon}^{-(n-2)})\\&-\kappa(n,x_0)\lambda_{\epsilon}^{-(n-2)}\int_{\mathbb{R}^n}\int_{\mathbb{R}^n}\frac{W^{p}(z)W^{p}(y)\log W(y)}{|y-z|^{n-2}} dydz\\&
-\kappa(n,x_0)\lambda_{\epsilon}^{-(n-2)}\int_{\mathbb{R}^n}\int_{\mathbb{R}^n}\frac{W^{p}(z)\log W(z)W^{p}(y)}{|y-z|^{n-2}} dydz\\&
+\kappa^2(n,x_0)\lambda_{\epsilon}^{-2(n-2)}\int_{\mathbb{R}^n}\int_{\mathbb{R}^n}\frac{W^{p}\log W(y)W^{p}\log W(z)}{|y-z|^{n-2}} dydz,
\end{split}
\end{equation}
where we have exploited the fact that
\begin{equation*}
\begin{split}
W^{p-\epsilon}=W^{p}(1-\epsilon\log W+O(\epsilon^2))&=W^{p}-\epsilon W^{p}\log W+O(\epsilon^2W^{p}))\\&
=W^{p}-\kappa(n,x_0)\lambda_{\epsilon}^{-(n-2)}W^{p}\log W+O(\lambda_{\epsilon}^{-2(n-2)}W^{p})
\end{split}
\end{equation*}
by Taylor's theorem and Theorem \ref{Figalli2}. One has
\begin{equation}\label{w(x)}
\begin{split}
&\lambda_{\epsilon}^{\frac{(n-2)\epsilon}{p-\epsilon}}\Big[\int_{\Omega}(|x|^{-(n-2)} \ast PW_1^{p-\epsilon})PW_1^{p-\epsilon} dx\Big]^{\frac{1}{p-\epsilon}}\\&
=\bigg\{\int_{\mathbb{R}^n}\int_{\mathbb{R}^n}\frac{W^{p}(y)W^{p}(z)}{|y-z|^{n-2}} dydz-2\frac{\kappa(n,x_0)}{\lambda_{\epsilon}^{n-2}}\int_{\mathbb{R}^n}\int_{\mathbb{R}^n}\frac{W^{p}(z)W^{p}(y)\log W(y)}{|y-z|^{n-2}} dydz\\&
\hspace{2mm}-2p\tilde{c}_{n,\mu}\lambda_{\epsilon}^{-(n-2)}\tilde{H}(x_1,x_1)\int_{\mathbb{R}^n}\int_{\mathbb{R}^n}\frac{W^{p}(z)W^{p-1}(y)}{|y-z|^{n-2}} dydz
+o(\frac{1}{\lambda_{\epsilon}^{n-2}})\bigg\}^{\frac{1}{p-\epsilon}},
\end{split}
\end{equation}
by virtue of
$$
\kappa^2(n,x_0)\int_{\mathbb{R}^n}\int_{\mathbb{R}^n}\frac{W^{p}\log W(y)W^{p}\log W(z)}{|y-z|^{n-2}}dydz<\infty.
$$
Thus we only need to deal with main terms $\mathcal{A}$, $\mathcal{B}$ and  $\mathcal{C}$ which are given by
$$\mathcal{A}:=\int_{\mathbb{R}^n}\int_{\mathbb{R}^n}\frac{W^{p}(y)W^{p}(z)}{|y-z|^{n-2}} dydz,\quad\mathcal{B}:=2\kappa(n,x_0)\lambda_{\epsilon}^{-(n-2)}\int_{\mathbb{R}^n}\int_{\mathbb{R}^n}\frac{W^{p}(y)W^{p}(z)\log W}{|y-z|^{n-2}}dydz,$$
$$\mathcal{C}:=p\tilde{c}_{n,\mu}\lambda_{\epsilon}^{-(n-2)}\tilde{H}(x_1,x_1)\int_{\mathbb{R}^n}\int_{\mathbb{R}^n}\frac{W^{p}(z)W^{p-1}(y)}{|y-z|^{n-2}} dydz.$$
Also, the elementary inequality
$$(1+t)^{-1}=1-t+t^2-\cdots,\quad \mbox{as}\quad|t|<1,$$
one has
$$\frac{1}{p-\epsilon}=\frac{1}{p}(1-\frac{\epsilon}{p})^{-1}\approx{\frac{1}{p}}+\frac{\epsilon}{p^2}\quad\mbox{as}\quad\epsilon\rightarrow0.$$
Therefore, from Taylor's theorem $e^{z}=1+z+\frac{z^2}{2!}+\cdots(z\in\mathbb{C})$ imply that
\begin{equation*}
\begin{split}
\Big(\mathcal{A}+\mathcal{B}+\mathcal{C}+o(\lambda_{\epsilon}^{-(n-2)})\Big)^{\frac{1}{p-\epsilon}}
&\approx e^{\frac{1}{p}\log\big(\mathcal{A}+\mathcal{B}+\mathcal{C}+o(\lambda_{\epsilon}^{-(n-2)})\big)}\cdot e^{\frac{\epsilon}{p^2}\log\big(\mathcal{A}+\mathcal{B}+\mathcal{C}+o(\lambda_{\epsilon}^{-(n-2)})\big)}\\&
\approx\big(\mathcal{A}+\mathcal{B}+\mathcal{C}+o(\frac{1}{\lambda_{\epsilon}^{n-2}})\big)^{\frac{1}{p}}\big[1+\frac{\epsilon}{p^2}\log\big(\mathcal{A}+\mathcal{B}+\mathcal{C}+o(\frac{1}{\lambda_{\epsilon}^{n-2}})\big)\big].
\end{split}
\end{equation*}
Furthermore, we note that
$$\mathcal{A}+\mathcal{B}+\mathcal{C}+o(\frac{1}{\lambda_{\epsilon}^{n-2}})\approx\mathcal{A}$$
as $\epsilon>0$ small enough. Thus, thanks to Theorem \ref{Figalli2} and $\lambda_{\epsilon}^{-(n-2)}\rightarrow0$ as $\epsilon\rightarrow0$,
$$\frac{\epsilon}{p^2}\big(\mathcal{A}+\mathcal{B}+\mathcal{C}+o(\frac{1}{\lambda_{\epsilon}^{n-2}})\big)^{\frac{1}{p}}\log\big(\mathcal{A}+\mathcal{B}+\mathcal{C}+o(\frac{1}{\lambda_{\epsilon}^{n-2}})\big)\approx\kappa(n,x_0)\lambda_{\epsilon}^{-(n-2)}\frac{1}{p^2}\mathcal{A}^{\frac{1}{p}}\log\mathcal{A}
$$
as $\epsilon>0$ small enough. As the result,
$$\big(\mathcal{A}+\mathcal{B}+\mathcal{C}+o(\lambda_{\epsilon}^{-(n-2)})\big)^{\frac{1}{p-\epsilon}} \approx\big(\mathcal{A}+\mathcal{B}+\mathcal{C}+o(\frac{1}{\lambda_{\epsilon}^{n-2}})\big)^{\frac{1}{p}}+\kappa(n,x_0)\lambda_{\epsilon}^{-(n-2)}\frac{1}{p^2}\mathcal{A}^{\frac{1}{p}}\log\mathcal{A}+
o(\frac{1}{\lambda_{\epsilon}^{n-2}}).
$$
Finally, combining all this together, from \eqref{w(x)} we conclude that
\begin{equation*}
\begin{split}
&\lambda_{\epsilon}^{\frac{(n-2)\epsilon}{p-\epsilon}}\Big[\int_{\Omega}(|x|^{-(n-2)} \ast PW_1^{p-\epsilon})PW_1^{p-\epsilon} dx\Big]^{\frac{1}{p-\epsilon}}\\&
=\Big[\int_{\mathbb{R}^n}\int_{\mathbb{R}^n}\frac{W^{p}(y)W^{p}(z)}{|y-z|^{n-2}} dydz-2\kappa(n,x_0)\lambda_{\epsilon}^{-(n-2)}\int_{\mathbb{R}^n}\int_{\mathbb{R}^n}\frac{W^{p}(z)W^{p}(y)\log W}{|y-z|^{n-2}} dydz\\&
\hspace{4mm}-2p\frac{\tilde{c}_{n,\mu}}{\lambda_{\epsilon}^{n-2}}\tilde{H}(x_1,x_1)\int_{\mathbb{R}^n}\int_{\mathbb{R}^n}\frac{W^{p}(z)W^{p-1}(y)}{|y-z|^{n-2}} dydz+o(\lambda_{\epsilon}^{-(n-2)})\Big]^{\frac{1}{p}}\\&
\hspace{4mm}+\kappa(n,x_0)\lambda_{\epsilon}^{-(n-2)}\frac{1}{p^2}\big(\int_{\mathbb{R}^n}\int_{\mathbb{R}^n}\frac{W^{p}(y)W^{p}(z)}{|y-z|^{n-2}}\big)^{\frac{1}{p}}\log\big(\int_{\mathbb{R}^n}\int_{\mathbb{R}^n}\frac{W^{p}(y)W^{p}(z)}{|y-z|^{n-2}}\big)+o(\frac{1}{\lambda_{\epsilon}^{n-2}}).
\end{split}
\end{equation*}
Using Taylor's theorem again, we obtain
\begin{equation*}
\begin{split}
&\lambda_{\epsilon}^{\frac{(n-2)\epsilon}{p-\epsilon}}\Big[\int_{\Omega}(|x|^{-(n-2)} \ast PW_1^{p-\epsilon})PW_1^{p-\epsilon} dx\Big]^{\frac{1}{p-\epsilon}}\\&
=\bigg\{\int_{\mathbb{R}^n}\int_{\mathbb{R}^n}\frac{W^{p}(y)W^{p}(z)}{|y-z|^{n-2}} dydz\bigg\}^{\frac{1}{p}}-\frac{1}{p}\bigg\{\int_{\mathbb{R}^n}\int_{\mathbb{R}^n}\frac{W^{p}(y)W^{p}(z)}{|y-z|^{n-2}} dydz\bigg\}^{\frac{1}{p}-1}\\&\times\bigg\{\frac{2\kappa(n,x_0)}{\lambda_{\epsilon}^{n-2}}\int_{\mathbb{R}^n}\int_{\mathbb{R}^n}\frac{W^{p}(z)W^{p}\log W}{|y-z|^{n-2}} dydz+2p\frac{\tilde{c}_{n,\mu}}{\lambda_{\epsilon}^{n-2}}\tilde{H}(x_1,x_1)\int_{\mathbb{R}^n}\int_{\mathbb{R}^n}\frac{W^{p}(z)W^{p-1}}{|y-z|^{n-2}} dydz\bigg\}\\&
+\frac{\kappa(n,x_0)}{\lambda_{\epsilon}^{n-2}}\frac{1}{p^2}\big(\int_{\mathbb{R}^n}\int_{\mathbb{R}^n}\frac{W^{p}(y)W^{p}(z)}{|y-z|^{n-2}}dxdy\big)^{\frac{1}{p}}\log\big(\int_{\mathbb{R}^n}\int_{\mathbb{R}^n}\frac{W^{p}(y)W^{p}(z)}{|y-z|^{n-2}}dydz\big)+o(\lambda_{\epsilon}^{-(n-2)})
\\&
=C_{HLS}^{\frac{n+2}{4p}}-\frac{1}{p}C_{HLS}^{\frac{n+2}{4}(\frac{1}{p}-1)}2\kappa(n,x_0)\lambda_{\epsilon}^{-(n-2)}\int_{\mathbb{R}^n}\int_{\mathbb{R}^n}\frac{W^{p}(y)W^{p}(z)\log W}{|y-z|^{n-2}} dydz\\&
\hspace{4mm}-2C_{HLS}^{\frac{n+2}{4}(\frac{1}{p}-1)}\tilde{H}(x_1,x_1)\tilde{c}_{n,\mu}\lambda_{\epsilon}^{-(n-2)}\int_{\mathbb{R}^n}\int_{\mathbb{R}^n}\frac{W^{p}(z)W^{p-1}(y)}{|y-z|^{n-2}}dydz\\&
\hspace{4mm}+\kappa(n,x_0)\frac{1}{p^2}C_{HLS}^{\frac{n+2}{4p}}\lambda_{\epsilon}^{-(n-2)}\log C_{HLS}^{\frac{n+2}{4}}+o(\lambda_{\epsilon}^{-(n-2)})=:J_{\epsilon}.
\end{split}
\end{equation*}
Then as a consequence,
\begin{equation}\label{IJE}
\begin{split}
&\frac{S_{HL}^{\epsilon}}{\lambda_{\epsilon}^{\frac{(n-2)\epsilon}{p-\epsilon}}}\leq \frac{\int_{\Omega}|\nabla u|^2dx}{\lambda_{\epsilon}^{\frac{(n-2)\epsilon}{p-\epsilon}} \left[\int_{\Omega}(|x|^{-(n-2)} \ast|u_{\epsilon}|^{p-\epsilon})|u_{\epsilon}|^{p-\epsilon} dx\right]^{\frac{1}{p-\epsilon}}}=\frac{I_{\epsilon}
}{J_{\epsilon}}.
\end{split}
\end{equation}

In what follows, we consider the each term of $I_{\epsilon}$ and $J_{\epsilon}$, respectively. By a direct computation, it follows that
\begin{equation}\label{Iepusi}
\begin{split}
\frac{I_{\epsilon}}{C_{HLS}^{(n+2)/4p}}=\frac{C_{HLS}^{\frac{n+2}{4}}}{C_{HLS}^{(n+2)/4p}}-\frac{\tilde{c}_{n,\mu}}{\lambda_{\epsilon}^{n-2}}\frac{H(x_1,x_1)}{C_{HLS}^{(n+2)/4p}}\int_{\mathbb{R}^n}\int_{\mathbb{R}^n}\frac{W^{p}(z)W^{p-1}(y)}{|y-z|^{n-2}} dydz+\frac{o(\lambda_{\epsilon}^{-(n-2)})}{C_{HLS}^{(n+2)/4p}}.
\end{split}
\end{equation}
Thanks to the identity
$$\frac{1}{1-t}=1+t+t^2+\cdots\quad \mbox{for}\quad|t|<1,$$
Then, by a direct computation, we deduce that
\begin{equation}\label{JLONG}
\begin{split}
\big(\frac{J_{\epsilon}}{C_{HLS}^{(n+2)/4p}}\big)^{-1}&=\frac{1}{1-C_{HLS}^{-\frac{n+2}{4}}\lambda_{\epsilon}^{-(n-2)}\mathcal{D}
+\lambda_{\epsilon}^{-(n-2)}\mathcal{E}
-2C_{HLS}^{-\frac{n+2}{4}}\lambda_{\epsilon}^{-(n-2)}\mathcal{F}+o(\lambda_{\epsilon}^{-(n-2)})C_{HLS}^{-\frac{n+2}{4p}}}\\&
=1+\frac{\lambda_{\epsilon}^{-(n-2)}}{C_{HLS}^{\frac{n+2}{4}}}\mathcal{D}
-\lambda_{\epsilon}^{-(n-2)}\mathcal{E}
+2\frac{\lambda_{\epsilon}^{-(n-2)}}{C_{HLS}^{\frac{n+2}{4}}}\mathcal{F}-\frac{o(\lambda_{\epsilon}^{-(n-2)})}{C_{HLS}^{\frac{n+2}{4p}}}+O(\lambda_{\epsilon}^{-2(n-2)}),
\end{split}
\end{equation}
where
$$\mathcal{D}:=\frac{1}{p}2\kappa(n,x_0)\int_{\mathbb{R}^n}\int_{\mathbb{R}^n}\frac{W^{p}(y)W^{p}(z)\log W}{|y-z|^{n-2}}dydz,\hspace{2mm}\mathcal{E}:=\frac{1}{p^2}\kappa(n,x_0)\log C_{HLS}^{\frac{n+2}{4}},$$
$$\mathcal{F}:=\tilde{c}_{n,\mu}\tilde{H}(x_1,x_1)\int_{\mathbb{R}^n}\int_{\mathbb{R}^n}\frac{W^{p}(z)W^{p-1}(y)}{|y-z|^{n-2}}dydz.$$
Hence, combining \eqref{IJE}, \eqref{Iepusi} and \eqref{JLONG} yields
\begin{equation*}
\begin{split}
\frac{S_{HL}^{\epsilon}}{\lambda_{\epsilon}^{\frac{(n-2)\epsilon}{p-\epsilon}}}&\leq\frac{C_{HLS}^{\frac{n+2}{4}}-\lambda_{\varepsilon}^{-(n-2)} \mathcal{F}+o(\lambda_{\varepsilon}^{-(n-2)})}{C_{HLS}^{\frac{n+2}{4p}}-C_{HLS}^{\frac{n+2}{4}(\frac{1}{p}-1)}\lambda_{\epsilon}^{-(n-2)}\mathcal{D}
+C_{HLS}^{\frac{n+2}{4p}}\lambda_{\epsilon}^{-(n-2)}\mathcal{E}
-C_{HLS}^{\frac{n+2}{4}(\frac{q}{p}-1)}\lambda_{\epsilon}^{-(n-2)}\mathcal{F}+o(\lambda_{\epsilon}^{-(n-2)})}\\
&=\frac{C_{HLS}^{\frac{n+2}{4}}}{C_{HLS}^{\frac{n+2}{4p}}}+\lambda_{\varepsilon}^{-(n-2)}\bigg[\frac{\mathcal{F}}{C_{HLS}^{\frac{n+2}{4p}}}+\frac{\mathcal{D}}{C_{HLS}^{\frac{n+2}{4p}}}-\frac{\mathcal{E}}{C_{HLS}^{\frac{n+2}{4}(\frac{1}{p}-1)}}\bigg]+o(\lambda_{\varepsilon}^{-(n-2)})\\&
=C_{HLS}+\lambda_{\varepsilon}^{-(n-2)}\bigg[\frac{\mathcal{F}}{C_{HLS}^{\frac{n-2}{4}}}+\frac{\mathcal{D}}{C_{HLS}^{\frac{n-2}{4}}}-\frac{\mathcal{E}}{C_{HLS}^{-1}}\bigg]+o(\lambda_{\varepsilon}^{-(n-2)}).
\end{split}
\end{equation*}
This concludes the proof.
\end{proof}
\subsection{Proofs of Theorems \ref{maximum}-\ref{asymptotic}}
In order to prove our results, we first recall that
the problem
\begin{equation}\label{question}
\begin{cases}
-\Delta v_{\epsilon}(x)=\big(|x|^{-{(n-2)}}\ast v_\epsilon^{p-\epsilon}\big)v_\epsilon^{p-1-\epsilon}\hspace{6mm}\mbox{in}\hspace{2mm} \Omega_{\epsilon},\\
 v_{\varepsilon}(x)>0~~~~~~~~~~~~~~~~~~~~~~~~~~~~~~~~~~~~~~~~~~~~~~~~\quad\text{in}~~\Omega_{\varepsilon}:=\lambda_{\epsilon}(\Omega-x_{\epsilon}),
\\
v_{\varepsilon}(x)=0~~~~~~~~~~~~~~~~~~~~~~~~~~~~~~~~~~~~~~~~~~~~~~~~~~~~\text{on}~~\partial\Omega_{\varepsilon},\\
v_{\varepsilon}(0)=\max\limits_{x\in\Omega_{\varepsilon}}v_{\varepsilon}(x)=1.
\end{cases}
\end{equation}
Suppose
$$v_{\epsilon}(y)=\lambda_{\epsilon}^{-(n-2)/2}PW[x_{\epsilon},\lambda_{\epsilon}](\lambda_{\epsilon}^{-1}y+x_{\epsilon})+\lambda_{\epsilon}^{-(n-2)}\phi_{\epsilon}(y).$$
Then
\begin{equation}\label{fepusilong}
-\Delta\phi_{\epsilon}(y)=p\Big(|y|^{-(n-2)}\ast W^{p-1}\phi_{\epsilon}\Big)
W^{p-1}+(p-1)\Big(|y|^{-(n-2)}\ast W^{p}\Big)W^{p-2}\phi_{\epsilon}+f_{\epsilon}(\phi_{\epsilon}),
\end{equation}
where
\begin{equation*}
\begin{split}
f_{\epsilon}(\phi_{\epsilon})
&:=\lambda_{\epsilon}^{n-2}\bigg\{\Big(|y|^{-(n-2)}\ast v_{\epsilon}^{p-\epsilon}\Big)v_{\epsilon}^{p-1-\epsilon}
-\Big(|y|^{-(n-2)}\ast W^{p}\Big)W^{p-1}\\&
~~~-\lambda_{\epsilon}^{-(n-2)}\Big[p
\Big(|y|^{-(n-2)}\ast W^{p-1}\phi_{\epsilon}\Big)
W^{p-1}+(p-1)
\Big(|y|^{-(n-2)}\ast W^{p}\Big)
W^{p-2}\phi_{\epsilon}\Big]\bigg\}.
\end{split}
\end{equation*}
For $f_{\epsilon}(\phi_{\epsilon})$, we need the following useful result.
\begin{lem}\label{usefulestimate}
It holds that
\begin{equation*}
\begin{split}
|f_{\epsilon}(\phi_{\epsilon})|&\leq C\big[|y|^{-(n-2)}\ast W^{p-1-\epsilon}|\phi_{\epsilon}|\big]
W^{p-2-\epsilon}|v_{\epsilon}-W|+\big(|y|^{-(n-2)}\ast W^{p-1-\epsilon}\big)W^{p-2-\epsilon}(1+\lambda_{\epsilon}^{-(n-2)})
\\&~~~+C\big[
\big(|y|^{-(n-2)}\ast W^{p}\big)
W^{p-1-\epsilon}|\log W|+\big(|y|^{-(n-2)}\ast W^{p-1-\epsilon}|\log W|\big)W^{p-1}\big]\\&~~~+C\big[\lambda_{\epsilon}^{-(n-2)}\big(|y|^{-(n-2)}\ast W^{p-1-\epsilon}\big)W^{p-1-\epsilon}(\log W+1)\big]\\&
~~~+C\big[\big(|y|^{-(n-2)}\ast W^{p}\big)W^{p-2-\epsilon}|\log W|+\big(|y|^{-(n-2)}\ast W^{p-\epsilon}|\log W|\big)
W^{p-1-\epsilon}\big]\\&~~~+C\big[\big(|y|^{-(n-2)}\ast W^{p-\epsilon}(|\log W|+1)\big)W^{p-2-\epsilon}\big].
\end{split}
\end{equation*}
\end{lem}
\begin{proof}
We note that
\begin{equation}\label{lambdaepsilon}
\lambda_{\epsilon}^{-(n-2)/2}PW[a,\lambda](x)=W(y)-\tilde{c}_{n,\mu}\lambda^{-(n-2)}\tilde{H}(x_0,x_0+\lambda_{\epsilon}^{-1}y)+o(\lambda_{\epsilon}^{-(n-2)})
\end{equation}
by Lemma \ref{Lem2.2}.
Direct calculation shows that
\begin{equation*}
\begin{split}
&\big|\lambda_{\epsilon}^{-(n-2)}f_{\epsilon}(\phi_{\epsilon})\big|=\Big|\big(|y|^{-(n-2)}\ast v_{\epsilon}^{p-\epsilon}\big)v_{\epsilon}^{p-1-\epsilon}
-\big(|y|^{-(n-2)}\ast W^{p}\big)W^{p-1}\\&
~~~-\lambda_{\epsilon}^{-(n-2)}\big[p
\big(|y|^{-(n-2)}\ast W^{p-1}\phi_{\epsilon}\big)
W^{p-1}+(p-1)
\big(|y|^{-(n-2)}\ast W^{p}\big)
W^{p-2}\phi_{\epsilon}\big]\Big|\\&
\leq\Big|\big(|y|^{-(n-2)}\ast v_{\epsilon}^{p-\epsilon}\big)v_{\epsilon}^{p-1-\epsilon}
\\&~~~-\lambda_{\epsilon}^{-(n-2)}\Big[(p-\epsilon)
\big(|y|^{-(n-2)}\ast W^{p-1-\epsilon}\big(\phi_{\epsilon}-\tilde{c}_{n,\mu}H(x_0,x_0+\lambda_{\epsilon}^{-1}z)+o(\lambda_{\epsilon}^{-(n-2)})\big)\big)
W^{p-1-\epsilon}\\&~~~+(p-1-\epsilon)
\big(|y|^{-(n-2)}\ast W^{p-\epsilon}\big)
W^{p-2-\epsilon}\big(\phi_{\epsilon}-\tilde{c}_{n,\mu}H(x_0,x_0+\lambda_{\epsilon}^{-1}z)+o(\lambda_{\epsilon}^{-(n-2)})\big)\Big]\\&
~~~-\big(|x|^{-(n-2)}\ast W^{p-\epsilon}\big)W^{p-1-\epsilon}\Big|+\Big|\big(|x|^{-(n-2)}\ast W^{p-\epsilon}\big)W^{p-1-\epsilon}-\big(|x|^{-(n-2)}\ast W^{p}\big)W^{p-1}\Big|
\\&~~~+\lambda_{\epsilon}^{-(n-2)}\Big|p
\big(|y|^{-(n-2)}\ast W^{p-1}\phi_{\epsilon}\big)
W^{p-1}\\&~~~-(p-\epsilon)
\big[|y|^{-(n-2)}\ast W^{p-1-\epsilon}\big(\phi_{\epsilon}-\tilde{c}_{n,\mu}H(x_0,x_0+\lambda_{\epsilon}^{-1}z)+o(\lambda_{\epsilon}^{-(n-2)})\big)\big]
W^{p-1-\epsilon}\Big|\\&~~~+\lambda_{\epsilon}^{-(n-2)}\Big|(p-1)
\big(|y|^{-(n-2)}\ast W^{p}\big)
W^{p-2}\phi_{\epsilon}\\&~~~-(p-1-\epsilon)
\big(|y|^{-(n-2)}\ast W^{p-\epsilon}\big)
W^{p-2-\epsilon}\big(\phi_{\epsilon}-\tilde{c}_{n,\mu}H(x_0,x_0+\lambda_{\epsilon}^{-1}z)+o(\lambda_{\epsilon}^{-(n-2)})\big)\Big|\\&
:=J_1+J_2+J_3+J_4.
\end{split}
\end{equation*}
In the following, we shall estimate values of both $J_1-J_2$ and $J_3-J_4$.
For $J_1$, we first note that for $\epsilon$ small enough, there holds
\begin{equation}\label{ab}
a^t(1+\frac{b}{a})^t\approx a^t(1+t\frac{b}{a})\quad\mbox{as}\quad|\frac{b}{a}|\ll1.
\end{equation}
Du to \eqref{ab}, we get
\begin{equation*}
\begin{split}
J_1&\leq C\lambda_{\epsilon}^{-2(n-2)}
\big[|y|^{-(n-2)}\ast W^{p-1-\epsilon}\big(\phi_{\epsilon}-\tilde{c}_{n,\mu}H(x_0,x_0+\lambda_{\epsilon}^{-1}z)+o(\lambda_{\epsilon}^{-(n-2)})\big)\big]\\&
\hspace{4mm}\times
W^{p-2-\epsilon}\big[\phi_{\epsilon}-\tilde{c}_{n,\mu}H(x_0,x_0+\lambda_{\epsilon}^{-1}z)+o(\lambda_{\epsilon}^{-(n-2)})\big]\\&
\leq C\lambda_{\epsilon}^{-(n-2)}
\big[|y|^{-(n-2)}\ast W^{p-1-\epsilon}|\phi_{\epsilon}|\big]
W^{p-2-\epsilon}\big[1+|v_{\epsilon}-W|+\lambda_{\epsilon}^{-(n-2)}\big]\\&
\hspace{4mm}+C\lambda_{\epsilon}^{-2(n-2)}
\big[|y|^{-(n-2)}\ast W^{p-1-\epsilon}\big]
W^{p-2-\epsilon}\big[1+|\phi_{\epsilon}|\big],
\end{split}
\end{equation*}
as $\epsilon\rightarrow0$. We compute
$$|J_2|\leq \kappa(n,x_0)\lambda_{\epsilon}^{-(n-2)}\big[\big(|y|^{-(n-2)}\ast W^{p}\big)
W^{p-1-\epsilon}|\log W|+\big(|y|^{-(n-2)}\ast W^{p-\varepsilon}|\log W|\big)
W^{p-1-\epsilon}\big]$$
by the mean theorem and Theorem \ref{Figalli2}. For $J_3$, we have
\begin{equation*}
\begin{split}
J_3&\leq C\lambda_{\epsilon}^{-(n-2)}\Big[
\big(|y|^{-(n-2)}\ast W^{p-1}|\phi_{\epsilon}|\big)W^{p-1}-\big(|y|^{-(n-2)}\ast W^{p-1-\epsilon}|\phi_{\epsilon}|\big)
W^{p-1-\epsilon}\\&~~~+C\lambda_{\epsilon}^{-(n-2)}\Big[\epsilon\big(|y|^{-(n-2)}\ast W^{p-1-\epsilon}|\phi_{\epsilon}|\big)W^{p-1-\epsilon}+\big(|y|^{-(n-2)}\ast W^{p-1-\epsilon}\big)W^{p-1-\epsilon}\Big]\\&
\leq C\lambda_{\epsilon}^{-(n-2)}\Big[\lambda_{\epsilon}^{-(n-2)}
\big(|y|^{-(n-2)}\ast W^{p-1-\epsilon}|\log W||\phi_{\epsilon}|\big)
W^{p-1}+\big(|y|^{-(n-2)}\ast W^{p-1-\epsilon}\big)W^{p-1-\epsilon}\Big]\\&~~~+C\lambda_{\epsilon}^{-(n-2)}\Big[\lambda_{\epsilon}^{-(n-2)}\big(|y|^{-(n-2)}\ast W^{p-1-\epsilon}|\phi_{\epsilon}|\big)W^{p-1-\epsilon}(|\log W|+1)\Big],
\end{split}
\end{equation*}
by the mean theorem and Theorem \ref{Figalli2}.
Analogously to estimate of $J_3$, we also have
\begin{equation*}
\begin{split}
J_4&\leq C\lambda_{\epsilon}^{-(n-2)}\Big[\lambda_{\epsilon}^{-(n-2)}
\big(|y|^{-(n-2)}\ast W^{p}\big)
W^{p-2-\epsilon}|\log W||\phi_{\epsilon}|+\big(|y|^{-(n-2)}\ast W^{p-\epsilon}\big)W^{p-2-\epsilon}\Big]\\&~~~+C\lambda_{\epsilon}^{-(n-2)}\Big[\lambda_{\epsilon}^{-(n-2)}\big(|y|^{-(n-2)}\ast W^{p-\epsilon}(\log W+1)\big)W^{p-2-\epsilon}|\phi_{\epsilon}|\Big].
\end{split}
\end{equation*}
Combined with \eqref{cU}, hence the conclusion follows.
\end{proof}
We will need the following crucial result.
\begin{Prop}\label{pro}
Let $n=3,4,5$ and there exists a constant $M>0$ depending only on $\Omega$ such that
$B(x_0,3M)\subset\Omega$, there holds
$$\phi_{\epsilon}\rightarrow\phi_0\quad\mbox{in}\quad L^{\infty}(B(0,M\lambda_{\epsilon}))$$
as $\epsilon\rightarrow0$, where $\phi_0\in W^{2,q^{\prime}}(\mathbb{R}^n)(~\mbox{for}~q^{\prime}>n)$ is a bounded solution of
\begin{equation}\label{fai0}
\begin{split}
-\Delta\phi_0&=p
\Big(|y|^{-(n-2)}\ast W^{p-1}\phi_0\Big)
W^{p-1}+(p-1)\Big(|y|^{-(n-2)}\ast W^{p}\Big)
W^{p-2}\phi_{0}\\&-\kappa(n,x_0)\big(|y|^{-(n-2)}\ast W^{p}\big)W^{p-1}\log W-\kappa(n,x_0)\big(|y|^{-(n-2)}\ast W^{p}\log W\big)W^{p-1}\\&-(p-1)\tilde{c}_{n,\mu}\big(|y|^{-(n-2)}\ast W^p\big)W^{p-2}\tilde{H}(x_0,x_0)\\&
-p\tilde{c}_{n,\mu}\big(|y|^{-(n-2)}\ast W^{p-1}\tilde{H}(x_0,x_0)\big)W^{p-1}\quad\mbox{in}\quad\mathbb{R}^n.
\end{split}
\end{equation}
\end{Prop}
Its proof is postponed to Section \ref{pro-1}, but use it to show Theorems \ref{maximum}-\ref{asymptotic} in section \ref{THEOREM}.
\subsection{A lower bound for $S_{HL}^{\epsilon}$}
This subsection is devoted to the derivation of a lower bound for $S_{HL}^{\epsilon}$.
\begin{Prop}\label{pro00}
It holds that
\begin{equation*}
\frac{S_{HL}^{\epsilon}}{\lambda_{\epsilon}^{\frac{(n-2)\epsilon}{p-\epsilon}}}= C_{HLS}+\lambda_{\varepsilon}^{-(n-2)}\bigg[\tilde{H}(x_0,x_0)C_{HLS}^{-\frac{n-2}{4}}\mathcal{F}^{\prime}+\frac{2}{p}C_{HLS}^{-\frac{n-2}{4}}\mathcal{D}^{\prime}-\frac{1}{p^2}C_{HLS}\mathcal{E}^{\prime}\bigg]+o(\lambda_{\varepsilon}^{-(n-2)}),
\end{equation*}
where the above notations $\mathcal{F}^{\prime}$, $\mathcal{D}^{\prime}$ and $\mathcal{E}^{\prime}$ are defined in Lemma \ref{shlep}.
\end{Prop}
\begin{proof}
We compute
\begin{equation}\label{SHLeong}
\begin{split}
\lambda_{\epsilon}^{-(n-2)\epsilon/(p-\epsilon)}S_{HL}^{\epsilon}&=\lambda_{\epsilon}^{-(n-2)\epsilon/(p-\epsilon)}S_{HL}^{\epsilon}(u_{\epsilon})\\&
=\lambda_{\epsilon}^{-(n-2)\epsilon/(p-\epsilon)}\left[\int_{\Omega}(|x|^{-(n-2)} \ast|u_{\epsilon}|^{p-\epsilon})|u_{\epsilon}|^{p-\epsilon} dx\right]^{1-\frac{1}{p-\epsilon}}\\&
=\left[\int_{\Omega_\epsilon}(|y|^{-(n-2)} \ast|v_{\epsilon}|^{p-\epsilon})|v_{\epsilon}|^{p-\epsilon} dy\right]^{1-\frac{1}{p-\epsilon}}.
\end{split}
\end{equation}
Owing to Lemma \ref{cWU} and Hardy-Littlewood-Sobolev inequality, we have
\begin{equation}\label{Leong}
\begin{split}
\left[\int_{\Omega_\epsilon}\int_{\Omega_\epsilon}\frac{|v_{\epsilon}|^{p-\epsilon}|v_{\epsilon}|^{p-\epsilon}}{|y-z|^{n-2}} dydz\right]^{1-\frac{1}{p-\epsilon}}=\left[\int_{B_{M\lambda_{\epsilon}}}\int_{B_{M\lambda_{\epsilon}}}\frac{|v_{\epsilon}(y)|^{p-\epsilon}|v_{\epsilon}(z)|^{p-\epsilon}}{|y-z|^{n-2}} dydz+O(\frac{1}{\lambda_{\epsilon}^4})\right]^{1-\frac{1}{p-\epsilon}},
\end{split}
\end{equation}
where we have used the fact that
\begin{equation*}
\begin{split}
2\int_{\Omega_{\epsilon}\setminus B_{M\lambda_{\epsilon}}}\int_{\Omega_{\epsilon}}\frac{W^{p}(z)W^{p}(y)}{|y-z|^{n-2}}dydz
+\int_{\Omega_{\epsilon}\setminus\ B_{M\lambda_{\epsilon}}}\int_{\Omega_{\epsilon}\setminus\ B_{M\lambda_{\epsilon}}}\frac{W^{p}(z)W^{p}(y)}{|y-z|^{n-2}}dydz=O(\frac{1}{\lambda_{\epsilon}^4}).
\end{split}
\end{equation*}
Furthermore, we note that
$$v_{\epsilon}(y)=W(y)+\lambda_{\epsilon}^{-(n-2)}[\phi_{\epsilon}-\tilde{c}_{n,\mu}H(x_0,x_0+\lambda_{\epsilon}^{-1}y))]+o(\frac{1}{\lambda_{\epsilon}^{n-2}}):=W(y)+\mathcal{H}_{\phi_{\epsilon}}(y)+o(\frac{1}{\lambda_{\epsilon}^{n-2}}).$$ Then, combing \eqref{lambdaepsilon} and \eqref{SHLeong}-\eqref{Leong} entail that
\begin{equation}\label{S2H}
\begin{split}
&\lambda_{\epsilon}^{-(n-2)\epsilon/(p-\epsilon)}S_{HL}^{\epsilon}\\=&\left[\int_{B_{M\lambda_{\epsilon}}}\int_{B_{M\lambda_{\epsilon}}}\frac{|v_{\epsilon}(y)|^{p-\epsilon}|v_{\epsilon}(z)|^{p-\epsilon}}{|y-z|^{n-2}} dydz+O(\frac{1}{\lambda_{\epsilon}^4})\right]^{1-\frac{1}{p-\epsilon}}\\
=&\left[\int_{B_{M\lambda_{\epsilon}}}\int_{B_{M\lambda_{\epsilon}}}\frac{(W+\mathcal{H}_{\phi_{\epsilon}}+o(\frac{1}{\lambda_{\epsilon}^{n-2}}))^{p-\epsilon}(W+\mathcal{H}_{\phi_{\epsilon}}+o(\frac{1}{\lambda_{\epsilon}^{n-2}}))^{p-\epsilon}}{|y-z|^{n-2}} dydz\right]^{1-\frac{1}{p-\epsilon}}+o(\frac{1}{\lambda_{\epsilon}^{n-2}}).
\end{split}
\end{equation}
We note that
\begin{equation*}
\begin{split}
\int_{B_{M\lambda_{\epsilon}}}\int_{B_{M\lambda_{\epsilon}}}\frac{W^{p-1}(z)\mathcal{H}_{\phi_{\epsilon}}(z)W^{p-1}(y)\mathcal{H}_{\phi_{\epsilon}}(y)}{|y-z|^{n-2}} dydz=
\begin{cases}
O\big(\frac{1}{\lambda_{\epsilon}^{2(n-2)}}\big), \quad\quad\hspace{1.5mm}\mbox{ when } n=3,\\
O\big(\frac{\log\lambda_{\epsilon}}{\lambda_{\epsilon}^{2(n-2)}}\big), \quad\quad\hspace{2mm}\mbox{ when } n=4,\\
O\big(\frac{1}{\lambda_{\epsilon}^{4}}\big), \hspace{1mm}\quad\quad\hspace{8mm}\mbox{ when } n=5.
\end{cases}
\end{split}
\end{equation*}
Therefore, using Hardy-Littlewood-Sobolev inequality, \eqref{taileextend}, Proposition \ref{pro} and \eqref{p1-00}, a direct calculation shows that
\begin{equation*}
\begin{split}
&\left[\int_{B_{M\lambda_{\epsilon}}}\int_{B_{M\lambda_{\epsilon}}}\frac{(W(z)+\mathcal{H}_{\phi_{\epsilon}}(z))^{p-\epsilon}(W(y)+\mathcal{H}_{\phi_{\epsilon}}(y))^{p-\epsilon}}{|y-z|^{n-2}} dydz\right]^{1-\frac{1}{p-\epsilon}}\\&
=\bigg[\int_{\mathbb{R}^n}\int_{\mathbb{R}^n}\frac{W^{p-\epsilon}(z)W^{p-\epsilon}(y)}{|y-z|^{n-2}} dydz+p\int_{B_{M\lambda_{\epsilon}}}\int_{B_{M\lambda_{\epsilon}}}\frac{W^p(z)W^{p-1}(y)\mathcal{H}_{\phi_{\epsilon}}(y)}{|y-z|^{n-2}} dydz\\&
\hspace{4mm}+p\int_{B_{M\lambda_{\epsilon}}}\int_{B_{M\lambda_{\epsilon}}}\frac{W^{p-1}(z)\mathcal{H}_{\phi_{\epsilon}}(z)W^{p}(y)}{|y-z|^{n-2}} dydz\\&\hspace{4mm}+p^2\int_{B_{M\lambda_{\epsilon}}}\int_{B_{M\lambda_{\epsilon}}}\frac{W^{p-1}(z)\mathcal{H}_{\phi_{\epsilon}}(z)W^{p-1}(y)\mathcal{H}_{\phi_{\epsilon}}(y)}{|y-z|^{n-2}} dydz+o(\lambda_{\epsilon}^{-(n-2)})\bigg]^{1-\frac{1}{p-\epsilon}}
\\&
=\bigg[\int_{\mathbb{R}^n}\int_{\mathbb{R}^n}\frac{W^{p}(y)W^{p}(z)}{|y-z|^{n-2}} dydz-\kappa(n,x_0)\lambda_{\epsilon}^{-(n-2)}\int_{\mathbb{R}^n}\int_{\mathbb{R}^n}\frac{W^{p}(z)W^{p}(y)\log W(y)}{|y-z|^{n-2}} dydz\\&
\hspace{4mm}-\kappa(n,x_0)\lambda_{\epsilon}^{-(n-2)}\int_{\mathbb{R}^n}\int_{\mathbb{R}^n}\frac{W^{p}(z)\log W(z)W^{p}(y)}{|y-z|^{n-2}} dydz\\&
\hspace{4mm}+\frac{p}{\lambda_{\epsilon}^{n-2}}\int_{B_{M\lambda_{\epsilon}}}\int_{B_{M\lambda_{\epsilon}}}\frac{W^p(z)W^{p-1}(y)[\phi_{\epsilon}-\tilde{c}_{n,\mu}H(x_0,x_0+\lambda_{\epsilon}^{-1}y)]}{|y-z|^{n-2}} dydz\\&\hspace{4mm}+\frac{p}{\lambda_{\epsilon}^{n-2}}\int_{B_{M\lambda_{\epsilon}}}\int_{B_{M\lambda_{\epsilon}}}\frac{W^{p-1}[\phi_{\epsilon}-\tilde{c}_{n,\mu}H(x_0,x_0+\lambda_{\epsilon}^{-1}z)]W^{p}}{|y-z|^{n-2}} +o(\frac{1}{\lambda_{\epsilon}^{n-2}})\bigg]^{1-\frac{1}{p-\epsilon}}+o(\frac{1}{\lambda_{\epsilon}^{n-2}})\\&
\end{split}
\end{equation*}
\begin{equation*}
\begin{split}
&=\bigg[\int_{\mathbb{R}^n}\int_{\mathbb{R}^n}\frac{W^{p}(y)W^{p}(z)}{|y-z|^{n-2}} dydz-\kappa(n,x_0)\lambda_{\epsilon}^{-(n-2)}\int_{\mathbb{R}^n}\int_{\mathbb{R}^n}\frac{W^{p}(z)W^{p}(y)\log W(y)}{|y-z|^{n-2}} dydz\\&
\hspace{4mm}-\frac{\kappa(n,x_0)}{\lambda_{\epsilon}^{n-2}}\int_{\mathbb{R}^n}\int_{\mathbb{R}^n}\frac{W^{p}(z)\log W(z)W^{p}(y)}{|y-z|^{n-2}} dydz
+\frac{p}{\lambda_{\epsilon}^{n-2}}\int_{\mathbb{R}^n}\int_{\mathbb{R}^n}\frac{W^p(z)W^{p-1}(y)\phi_{0}(y)}{|y-z|^{n-2}} dydz\\&+\frac{p}{\lambda_{\epsilon}^{n-2}}\int_{\mathbb{R}^n}\int_{\mathbb{R}^n}\frac{W^{p-1}(z)\phi_{0}(z)W^{p}(y)}{|y-z|^{n-2}} dydz-\frac{p\tilde{c}_{n,\mu}}{\lambda_{\epsilon}^{n-2}}\tilde{H}(x_0,x_0)\int_{\mathbb{R}^n}\int_{\mathbb{R}^n}\frac{W^{p-1}(z)W^{p}(y)}{|y-z|^{n-2}} dydz\\&\hspace{4mm}-p\tilde{c}_{n,\mu}\lambda_{\epsilon}^{-(n-2)}\tilde{H}(x_0,x_0)\int_{\mathbb{R}^n}\int_{\mathbb{R}^n}\frac{W^p(z)W^{p-1}(y)}{|y-z|^{n-2}} dydz+o(\lambda_{\epsilon}^{-(n-2)})\bigg]^{1-\frac{1}{p-\epsilon}}+o(\lambda_{\epsilon}^{-(n-2)}),
\end{split}
\end{equation*}
where we have used the fact that $|\tilde{H}(x_0,x_0+\lambda_{\epsilon}^{-1}y)-\tilde{H}(x_0,x_0)|\leq C\lambda_{\epsilon}^{-1}y$.
Moreover, one has
\begin{equation*}
\begin{split}
&\int_{\mathbb{R}^n}\int_{\mathbb{R}^n}\frac{W^p(z)W^{p-1}(y)\phi_{0}}{|y-z|^{n-2}} dydz+\int_{\mathbb{R}^n}\int_{\mathbb{R}^n}\frac{W^{p-1}(z)\phi_{0}W^{p}(y)\phi_{0}}{|y-z|^{n-2}} dydz\\&
=\frac{1}{p-1}\bigg\{\kappa(n,x_0)\int_{\mathbb{R}^n}\int_{\mathbb{R}^n} \frac{W^{p}(y)W^{p}\log W}{|y-z|^{n-2}}dydz+\kappa(n,x_0)\int_{\mathbb{R}^n}\int_{\mathbb{R}^n} \frac{W^{p}(z)W^{p}\log W }{|y-z|^{n-2}}dydz\\&\hspace{2mm}+\tilde{c}_{n,\mu}\tilde{H}(x_0,x_0)\Big[(p-1)\int_{\mathbb{R}^n}\int_{\mathbb{R}^n}\frac{W^p(z)W^{p-1}(y)}{|y-z|^{n-2}} dydz+p\int_{\mathbb{R}^n}\int_{\mathbb{R}^n}\frac{W^{p-1}(z)W^{p}(y)}{|y-z|^{n-2}} dydz\Big]\bigg\}
\end{split}
\end{equation*}
by \eqref{fai0}. Consequently, recalling \eqref{S2H} and combined with the above identities, straightforward computations show that
\begin{equation*}
\begin{split}
&\lambda_{\epsilon}^{-(n-2)\epsilon/(p-\epsilon)}S_{HL}^{\epsilon}\\&
=\bigg\{\int_{\mathbb{R}^n}\int_{\mathbb{R}^n}\frac{W^{p}(y)W^{p}(z)}{|y-z|^{n-2}} dydz+\frac{2}{p-1}\kappa(n,x_0)\lambda_{\epsilon}^{-(n-2)}\int_{\mathbb{R}^n}\int_{\mathbb{R}^n}\frac{W^{p}(z)W^{p}(y)\log W(y)}{|y-z|^{n-2}} dydz\\&
+\frac{p}{p-1}\tilde{c}_{n,\mu}\lambda_{\epsilon}^{-(n-2)}\tilde{H}(x_0,x_0)\int_{\mathbb{R}^n}\int_{\mathbb{R}^n}\frac{W^p(z)W^{p-1}(y)}{|y-z|^{n-2}} dydz+o(\lambda_{\epsilon}^{-(n-2)})\bigg\}^{1-\frac{1}{p-\epsilon}}+o(\lambda_{\epsilon}^{-(n-2)}).
\end{split}
\end{equation*}
Since \eqref{ab} and by
$$
\frac{p-1-\epsilon}{p-\epsilon}=\frac{p-1}{p}+\frac{-\epsilon}{p^2}+O(\epsilon^2)\quad\mbox{as}\quad\epsilon\rightarrow0,
$$
the conclusion of Theorem \ref{Figalli2} and Taylor's expansion, we have
$$\big(a^{\prime}+b^{\prime}+c^{\prime}+o(\lambda_{\epsilon}^{-(n-2)})\big)^{\frac{p-1-\epsilon}{p-\epsilon}} \approx\big(a^{\prime}+b^{\prime}+c^{\prime}+o(\frac{1}{\lambda_{\epsilon}^{n-2}})\big)^{\frac{p-1}{p}}-\frac{\kappa(n,x_0)}{\lambda_{\epsilon}^{n-2}}\frac{1}{p^2}{a^{^{\prime}}}^{\frac{1}{p}}\log a^{\prime}+
o(\frac{1}{\lambda_{\epsilon}^{n-2}}).
$$
As a consequence, coupling this bound and $(1+t)^m\approx1+mz$ as $z$ small enough, we use the same idea as in the proof for Lemma \ref{shlep} to get
\begin{equation*}
\begin{split}
&\lambda_{\epsilon}^{-(n-2)\epsilon/(p-\epsilon)}S_{HL}^{\epsilon}\\&
=\bigg\{\int_{\mathbb{R}^n}\int_{\mathbb{R}^n}\frac{W^{p}(y)W^{p}(z)}{|y-z|^{n-2}} dydz+\frac{p}{p-1}\tilde{c}_{n,\mu}\lambda_{\epsilon}^{-(n-2)}\tilde{H}(x_0,x_0)\int_{\mathbb{R}^n}\int_{\mathbb{R}^n}\frac{W^p(z)W^{p-1}(y)}{|y-z|^{n-2}} dydz\\&
\hspace{4mm}+\frac{2}{p-1}\kappa(n,x_0)\lambda_{\epsilon}^{-(n-2)}\int_{\mathbb{R}^n}\int_{\mathbb{R}^n}\frac{W^{p}(z)W^{p}(y)\log W(y)}{|y-z|^{n-2}}dydz+o(\lambda_{\epsilon}^{-(n-2)})\bigg\}^{1-\frac{1}{p}}
\\&\hspace{4mm}-\kappa(n,x_0)\lambda_{\epsilon}^{-(n-2)}\frac{1}{p^2}\big(\int_{\mathbb{R}^n}\int_{\mathbb{R}^n}\frac{W^{p}(y)W^{p}(z)}{|y-z|^{n-2}}\big)^{\frac{p-1}{p}}\log\big(\int_{\mathbb{R}^n}\int_{\mathbb{R}^n}\frac{W^{p}(y)W^{p}(z)}{|y-z|^{n-2}}\big)+o(\lambda_{\epsilon}^{-(n-2)})\\&
=C_{HLS}+\lambda_{\varepsilon}^{-(n-2)}\Big[\tilde{H}(x_0,x_0)C_{HLS}^{-\frac{n-2}{4}}\mathcal{F}^{\prime}+\frac{2}{p}C_{HLS}^{-\frac{n-2}{4}}\mathcal{D}^{\prime}-\frac{n+2}{4p^2}C_{HLS}\mathcal{E}^{\prime}\Big]+o(\lambda_{\varepsilon}^{-(n-2)})
\end{split}
\end{equation*}
as $\epsilon\rightarrow0$, wehere
$$\mathcal{D}^{\prime}:=\kappa(n,x_0)\int_{\mathbb{R}^n}\int_{\mathbb{R}^n}\frac{W^{p}(y)W^{p}(z)\log W}{|y-z|^{n-2}}dydz,\hspace{2mm}\mathcal{E}^{\prime}:=\kappa(n,x_0)\log C_{HLS},$$
$$\mathcal{F}^{\prime}:=\tilde{c}_{n,\mu}\int_{\mathbb{R}^n}\int_{\mathbb{R}^n}\frac{W^{p}(z)W^{p-1}(y)}{|y-z|^{n-2}}dydz,$$
as desired.
\end{proof}
\subsection{Proofs of Theorems \ref{maximum}-\ref{asymptotic}}\label{THEOREM}
We can now prove Theorems \ref{maximum}-\ref{asymptotic} by using Lemma \ref{shlep}, Proposition \ref{pro} and Proposition \ref{pro00}.
\begin{proof}[Proof of Theorem \ref{maximum}]
By virtue of Lemma \ref{shlep} and Proposition \ref{pro00}, we obtain
\begin{equation*}
\begin{split}
&C_{HLS}+\lambda_{\varepsilon}^{-(n-2)}\bigg[\tilde{H}(x_0,x_0)C_{HLS}^{-\frac{n-2}{4}}\mathcal{F}^{\prime}+\frac{2}{p}C_{HLS}^{-\frac{n-2}{4}}\mathcal{D}^{\prime}-\frac{1}{p^2}C_{HLS}\mathcal{E}^{\prime}\bigg]+o(\lambda_{\varepsilon}^{-(n-2)}).
\\&\leq C_{HLS}+\lambda_{\varepsilon}^{-(n-2)}\bigg[\tilde{H}(x_1,x_1)C_{HLS}^{-\frac{n-2}{4}}\mathcal{F}^{\prime}+\frac{2}{p}C_{HLS}^{-\frac{n-2}{4}}\mathcal{D}^{\prime}-\frac{n+2}{4p^2}C_{HLS}\mathcal{E}^{\prime}\bigg]+o(\lambda_{\varepsilon}^{-(n-2)}).
\end{split}
\end{equation*}
As a result,
$$
\tilde{H}(x_0,x_0)\leq\tilde{H}(x_1,x_1)\quad\mbox{for any}\quad x_1\in\Omega,
$$
which implies that $H(x_0,x_0)\geq H(x_1,x_1)$ and the conclusion follows.
\end{proof}
\begin{proof}[Proof of Theorem \ref{asymptotic}]
Owing to
\begin{equation*}
\begin{split}
v_{\epsilon}(y)&=\lambda_{\epsilon}^{-\frac{n-2}{2}}PW[x_{\epsilon},\lambda_{\epsilon}](\lambda_{\epsilon}^{-1}y+x_{\epsilon})+\lambda_{\epsilon}^{-(n-2)}\phi_{\epsilon}(y)\\&
=W(y)-\tilde{c}_{n,\mu}\lambda^{-(n-2)}\tilde{H}(x_\epsilon,x_\epsilon+\lambda_{\epsilon}^{-1}y)+\lambda_{\epsilon}^{-(n-2)}\phi_{0}(y)+o(\lambda_{\epsilon}^{-(n-2)}).
\end{split}
\end{equation*}
Combining this equality with Proposition \ref{pro} yields the conclusion.
\end{proof}

\section{Proof of Proposition \ref{pro}}\label{pro-1}
This section is devoted to the proof of Proposition \ref{pro}. As a first step, we
introduce the nondegeneracy property of the positive solutions of equation \eqref{ele} at $W_\lambda:=W[\xi,\lambda]$.
\begin{Prop}[Non-degeneracy \cite{XLi}]\label{prondgr}
Assume $n\geq 3$, $0<\mu\leq n$ with $0<\mu\leq4$ and $p^{\ast}=\frac{2n-\mu}{n-2}$. Then the solution
 $ W[\xi,\lambda](x)$ of the equation
\begin{equation*}
    -\Delta u=(|x|^{-\mu}\ast u^{p^{\ast}})u^{p^{\ast}-1} \quad \mbox{in}\quad \mathbb{R}^n
    \end{equation*}
is non-degenerate in the sense that all bounded solutions of linearized equation
\begin{equation*}
-\Delta \phi-p^{\ast}\left(|x|^{-\mu} \ast (W_{\lambda}^{p^{\ast}-1}\phi)\right)W_{\lambda}^{p^{\ast}-1}
-(p^{\ast}-1)\left(|x|^{-\mu} \ast W_{\lambda}^{p^{\ast}}\right)W_{\lambda}^{p^{\ast}-2}\phi=0
\end{equation*}
are the linear combination of the functions
\begin{equation*}
\partial_\lambda W_{\lambda}[\xi,\lambda],\hspace{2mm}\partial_{\xi_1}W_{\lambda}[\xi,\lambda],\cdots,\partial_{\xi_n}W_{\lambda}[\xi,\lambda].
\end{equation*}
\end{Prop}

We need the following result.
\begin{lem}\label{bopund} Up to subsequence, there exists a sequence of $\epsilon_j$ such that for $j$ large, there holds
$\big\|\eta_{j}(y)\big\|_{W^{2,q^{\prime}}}\leq C$, where we denote $\eta_{j}(y):=\big\|\phi_{\epsilon_j}(y)\big\|^{-1}_{L^{q^{\prime}}(\Omega_{\epsilon_j})}\phi_{\epsilon_j}(y)\hspace{2mm}\mbox{in}\hspace{2mm}\Omega_{\epsilon_j}.$
\end{lem}
\begin{proof}
We argue by contradiction, similarly to \cite{JW0}. Assume that there exists a sequence of $\epsilon_j$ such that $\big\|\phi_{\epsilon_j}\big\|_{L^{q^{\prime}}(\Omega_{\epsilon_j})}\rightarrow\infty$.
Let us denote $\Omega_j:=\Omega_{\epsilon_j}$, $\lambda_{j}:=\lambda_{\epsilon_j}$,
$\phi_{\epsilon_j}:=\phi_j$ and $f_j:=f_{\epsilon_j}$ for simplicity. Then $\eta_{j}$ satisfies
\begin{equation}\label{92}
\begin{split}
\begin{cases}
-\Delta\eta_{j}(y)=p\Big(|y|^{-(n-2)}\ast W^{p-1}\eta_{j}\Big)
W^{p-1}\\ \quad \quad\quad\hspace{2mm}+(p-1)\Big(|y|^{-(n-2)}\ast W^{p}\Big)W^{p-2}\eta_{j}+f_{\epsilon_j}(\phi_{\epsilon})\big\|\phi_{j}(y)\big\|^{-1}_{L^{q^{\prime}}(\Omega_{j})}=0,\quad\mbox{in}\quad\Omega_j,\\
\eta_j(y)=0,\quad \quad\quad\quad \quad\quad\quad \quad\quad\quad \quad\quad\quad \quad\quad\quad \quad\quad\quad \quad\quad\quad \quad \quad\quad\quad\hspace{3mm}\mbox{on}\quad\partial\Omega_j.
\end{cases}
\end{split}
\end{equation}
By regularity theorem \cite{GT}, we infer that
\begin{equation}\label{ittarq}
\begin{split}
\big\|\eta_{j}(y)\big\|_{W^{2,q^{\prime}}}&\leq C\big(\big\|\eta_{j}(y)\big\|_{L^{q^{\prime}}}+\big\|f_{\epsilon}(\phi_{\epsilon})\big\|_{L^{q^{\prime}}}\big\|\phi_{j}(y)\big\|^{-1}_{L^{q^{\prime}}}\big)\\&
+C\Big\|\Big(|y|^{-(n-2)}\ast W^{p-1}\eta_{j}\Big)
W^{p-1}+\Big(|y|^{-(n-2)}\ast W^{p}\Big)W^{p-2}\eta_{j}\Big\|_{L^{q^{\prime}}}.
\end{split}
\end{equation}
Here and after we drop the domain $\Omega_{j}$ if there is no confusion for simplicity.
In order to get the boundedness $\eta_{j}(y)$ in $W^{2,q^{\prime}}(\Omega_{j})$, it is sufficient to estimate the two rightmost terms in the above inequality.
Then, since
$$\frac{1}{r}-\frac{1}{2q^{\prime}}=\frac{2}{n}\quad\mbox{and}\quad\frac{1}{q^{\ast}}+\frac{2n}{4q^{\prime}+n}=1,$$
Hardy-Littlewood-Sobolev inequality and H\"{o}lder inequality give that
\begin{equation}\label{0091qq}
\begin{split}
\Big\|\Big(|y|^{-(n-2)}\ast W^{p-1}\eta_{j}\Big)
W^{p-1}\Big\|_{L^{q^{\prime}}}&\leq C\big\|W^{p-1}\eta_{j}\big\|_{L^{r}}\big\|W^{\frac{8}{n-2}}\big\|^{\frac{1}{2}}_{L^{q^{\prime}}}
\\&\leq C\big\|W^{p-1}\big\|_{L^{q^{\ast}r}}\big\|\eta_{j}\big\|^{\frac{2n}{(4q^{\prime}+n)r}}_{L^{q^{\prime}}}\big\|W^{\frac{8}{n-2}}\big\|^{\frac{1}{2}}_{L^{q^{\prime}}}<\infty.
\end{split}
\end{equation}
Similarly, we have
\begin{equation}\label{0q668}
\Big\|\Big(|y|^{-(n-2)}\ast W^{p}\Big)W^{p-2}\eta_{j}\Big\|_{L^{q^{\prime}}}<\infty.
\end{equation}
Furthermore, by virtue of Lemma \ref{usefulestimate}, we get
\begin{equation}\label{92qq}
\begin{split}
\big\|f_{j}\big\|_{L^{q^{\prime}}(\Omega_{j})}&\leq
C\Big\|\big[|y|^{-(n-2)}\ast W^{p-1-\epsilon_j}|\phi_{j}|\big]
W^{p-2-\epsilon_j}|v_{\epsilon_j}-W|\Big\|_{L^{q^{\prime}}}\\+&C\Big\|\big(|y|^{-(n-2)}\ast W^{p-1-\epsilon_j}\big)W^{p-2-\epsilon_j}(1+\lambda_{j}^{-(n-2)})\Big\|_{L^{q^{\prime}}}
\\+&C\Big\|\big[
\big(|y|^{-(n-2)}\ast W^{p}\big)
W^{p-1-\epsilon_j}|\log W|+\big(|y|^{-(n-2)}\ast W^{p-1-\epsilon_j}|\log W|\big)W^{p-1}\big]\Big\|_{L^{q^{\prime}}}\\+&C\Big\|\big[\lambda_{j}^{-(n-2)}\big(|y|^{-(n-2)}\ast W^{p-1-\epsilon_j}\big)W^{p-1-\epsilon}(\log W+1)\big]\Big\|_{L^{q^{\prime}}}\\+&C\Big\|\big[\big(|y|^{-(n-2)}\ast W^{p}\big)W^{p-2-\epsilon_j}|\log W|+\big(|y|^{-(n-2)}\ast W^{p-\epsilon_j}|\log W|\big)
W^{p-1-\epsilon_j}\big]\Big\|_{L^{q^{\prime}}}\\+&C\Big\|\big[\big(|y|^{-(n-2)}\ast W^{p-\epsilon_j}(|\log W|+1)\big)W^{p-2-\epsilon_j}\big]\Big\|_{L^{q^{\prime}}}.
\end{split}
\end{equation}
Applying Hardy-Littlewood-Sobolev inequality, H\"{o}lder inequality and \eqref{cU}, we obtain that, for $\epsilon$ small
\begin{equation}\label{933qq}
\begin{split}
\Big\|\big[|y|^{-(n-2)}\ast W^{p-1-\epsilon_j}&|\phi_{j}|\big]
W^{p-2-\epsilon_j}|v_{\epsilon_j}-W|\Big\|_{L^{q^{\prime}}}\\&\leq C\big\| W^{p-1-\epsilon_j}\big\|_{L^{q^{\ast}r}}\big\|\phi_{j}\big\|^{\frac{2n}{(2q^{\prime}+n)r}}_{L^{q^{\prime}}}
\big\|W^{p-2-\epsilon_j}|v_{\epsilon_j}-W|\big\|_{L^{q^{\prime}}}
\\&\leq C\big\|\phi_{j}\big\|^{\frac{2n}{(2q^{\prime}+n)r}}_{L^{q^{\prime}}}.
\end{split}
\end{equation}
For the remaining terms, since $q^{\prime}>n$ and using \eqref{p1-00} together with the fact that $\log W\leq W$,  a straightforward computations shows that
\begin{equation}\label{9934qq}
\begin{split}
&\Big\|\big(|y|^{-(n-2)}\ast W^{p-1-\epsilon}\big)W^{p-2-\epsilon_j}(1+\lambda_{j}^{-(n-2)})\Big\|_{L^{q^{\prime}}}
\\&+\Big\|\big[
\big(|y|^{-(n-2)}\ast W^{p}\big)
W^{p-1-\epsilon}|\log W|+\big(|y|^{-(n-2)}\ast W^{p-1-\epsilon_j}|\log W|\big)W^{p-1}\big]\Big\|_{L^{q^{\prime}}}\\&+\Big\|\big[\lambda_{j}^{-(n-2)}\big(|y|^{-(n-2)}\ast W^{p-1-\epsilon_j}\big)W^{p-1-\epsilon}(\log W+1)\big]\Big\|_{L^{q^{\prime}}}\\&+\Big\|\big[\big(|y|^{-(n-2)}\ast W^{p}\big)W^{p-2-\epsilon_j}|\log W|+\big(|y|^{-(n-2)}\ast W^{p-\epsilon_j}|\log W|\big)
W^{p-1-\epsilon_j}\big]\Big\|_{L^{q^{\prime}}}\\&+\Big\|\big[\big(|y|^{-(n-2)}\ast W^{p-\epsilon_j}(|\log W|+1)\big)W^{p-2-\epsilon_j}\big]\Big\|_{L^{q^{\prime}}}<\infty
\end{split}
\end{equation}
by Hardy-Littlewood-Sobolev inequality and H\"{o}lder inequality as $\epsilon\rightarrow0$.
Combining the previous estimates together, we get
\begin{equation}\label{itabounded}
\big\|\eta_{j}(y)\big\|_{W^{2,q^{\prime}}}\leq C\Big(\big\|\eta_{j}(y)\big\|_{L^{q^{\prime}}}+C\big(1+\big\|\phi_{j}\big\|^{\frac{2n}{(2q^{\prime}+n)r}}_{L^{q^{\prime}}}\big)\Big)\big\|\phi_{j}(y)\big\|^{-1}_{L^{q^{\prime}}}\leq C
\end{equation}
as $j\rightarrow\infty$. By Lemma \ref{regular000}, we can extend $\eta_{j}$ to $\mathbb{R}^n$ in such a way that
\begin{equation}\label{Rrnn}
\big\|\eta_{j}(y)\big\|_{W^{2,q^{\prime}}(\mathbb{R}^n)}\leq C\big\|\eta_{j}(y)\big\|_{W^{2,q^{\prime}}}.
\end{equation}
Therefore, combined with the Sobolev imbedding therorem, we infer that
$$\big\|\eta_{j}(y)\big\|_{L^{\infty}(\mathbb{R}^n)}\leq C.$$
The result follows.
\end{proof}
\begin{lem}\label{upto}
Up to a subsequence, there exists a function $\phi_0$ in $W^{2,q^{\prime}}$ such that $\phi_{\epsilon_j}\rightarrow \phi_0$ weakly in $W^{2,q^{\prime}}(\mathbb{R}^n)$ for every $q^{\prime}>n$ and $\phi_{\epsilon_j}\rightarrow \omega_0$ in $C_{loc}^{1}(\mathbb{R}^n)$.
\end{lem}
\begin{proof}
To conclude the proof of Lemma \ref{upto}, we establish now the
following key estimate:
\begin{equation}\label{keyestimate}
\big\|\phi_{\epsilon}\big\|_{L^{q^{\prime}}(\Omega_{\epsilon})}\leq C(q^{\prime})\quad\mbox{for every}\quad n<q^{\prime}<\infty.
\end{equation}
Once this is done, as in \eqref{ittarq}, a well-known regularity argument show that
$\big\|\phi_{\epsilon}\big\|_{W^{2,q^{\prime}}(\Omega_{\epsilon})}\leq C(q^{\prime})$ for every $q^{\prime}>n$.
Furthermore, we can extend $\phi_{\epsilon}$ to $\mathbb{R}^n$ in the same way similarl to \eqref{Rrnn}. Then up to subsequence, we may extract a sequence-which we still call $\phi_{\epsilon_j}$-which converges weakly in $W^{2,q^{\prime}}(\mathbb{R}^n)$, and strongly convergence to a limit $\phi_0$ in $C_{loc}^{1}(\mathbb{R}^n)$. This concludes the proof of Lemma.

Assume now by contradiction that \eqref{keyestimate} does not hold true, in other words that, up to a subsequence, there exists a sequence of $\epsilon_j$ such that $\big\|\phi_{\epsilon_j}\big\|_{L^{q^{\prime}}(\Omega_{\epsilon_j})}\rightarrow\infty$.
In order to have the desired contradiction with $\|\eta_j\|_{L^{q^{\prime}}(\Omega_{j})}=1$, it is sufficient to obtain
\begin{equation}\label{contradictiion}
\big\|\eta_j\big\|_{W^{2,q^{\prime}}(\mathbb{R}^n)}=o(1) \quad\mbox{as}\quad j\rightarrow\infty,
\end{equation}
which we prove next.
Since Lemma \ref{bopund},
then up to a subsequence, we have $\eta_j\rightarrow \omega_0$ weakly in $W^{2,q^{\prime}}(\mathbb{R}^n)$ for every $q^{\prime}>n$ and $\eta_j\rightarrow \omega_0$ in $C_{loc}^{1}(\mathbb{R}^n)$.

We claim now that $\omega_0=0$. Indeed, recalling \eqref{itabounded}, and Lemma \ref{usefulestimate} entail that
\begin{equation*}
\begin{split}
&\big\|f_{j}\big\|_{L^{\infty}}\big\|\phi_{j}(y)\big\|^{-1}_{L^{\infty}}\leq
C\big\|\phi_{j}(y)\big\|^{-1}_{L^{\infty}}\Big\|\big[|y|^{-(n-2)}\ast W^{p-1-\epsilon_j}|\phi_{j}|\big]
W^{p-2-\epsilon_j}|v_{\epsilon_j}-W|\Big\|_{L^{\infty}}\\+&C\big\|\phi_{j}(y)\big\|^{-1}_{L^{\infty}}\Big\|\big(|y|^{-(n-2)}\ast W^{p-1-\epsilon_j}\big)W^{p-2-\epsilon_j}(1+\lambda_{j}^{-(n-2)})\Big\|_{L^{\infty}}
\\+&C\big\|\phi_{j}(y)\big\|^{-1}_{L^{\infty}}\Big\|\big[
\big(|y|^{-(n-2)}\ast W^{p}\big)
W^{p-1-\epsilon_j}|\log W|+\big(|y|^{-(n-2)}\ast W^{p-1-\epsilon}|\log W|\big)W^{p-1}\big]\Big\|_{L^{\infty}}\\+&C\big\|\phi_{j}(y)\big\|^{-1}_{L^{\infty}}\Big\|\big[\lambda_{j}^{-(n-2)}\big(|y|^{-(n-2)}\ast W^{p-1-\epsilon_j}\big)W^{p-1-\epsilon_j}(\log W+1)\big]\Big\|_{L^{\infty}}\\+&C\big\|\phi_{j}(y)\big\|^{-1}_{L^{\infty}}\Big\|\big[\big(|y|^{-(n-2)}\ast W^{p}\big)W^{p-2-\epsilon_j}|\log W|+\big(|y|^{-(n-2)}\ast W^{p-\epsilon_j}|\log W|\big)
W^{p-1-\epsilon_j}\big]\Big\|_{L^{\infty}}\\+&C\big\|\phi_{j}(y)\big\|^{-1}_{L^{\infty}}\Big\|\big[\big(|y|^{-(n-2)}\ast W^{p-\epsilon_j}(|\log W|+1)\big)W^{p-2-\epsilon_j}\big]\Big\|_{L^{\infty}}\rightarrow0
\end{split}
\end{equation*}
for $\epsilon\rightarrow0$ sufficiently small and $j\rightarrow\infty$ by Hardy-Littlewood-Sobolev inequality and H\"{o}lder inequality and the fact that $\|v_{\epsilon_j}-W\|_{L^{\infty}(\Omega_{j})}\rightarrow0$ as $\epsilon\rightarrow0$ and $j\rightarrow\infty$. This means $\omega_0$ is a classical solution of
\begin{equation}\label{9002}
\begin{split}
\begin{cases}
-\Delta\omega_0(y)=p\Big(|y|^{-(n-2)}\ast W^{p-1}\omega_0\Big)
W^{p-1}\\ \quad\quad\quad\quad+(p-1)\Big(|y|^{-(n-2)}\ast W^{p}\Big)W^{p-2}\omega_0\quad\mbox{for all}\quad\omega_0\in W^{2,q^{\prime}}(\mathbb{R}^n).
\end{cases}
\end{split}
\end{equation}
From the non-degeneracy of $W$ in Proposition \ref{prondgr}, we have
$$\omega_0=\omega=c_0\frac{\partial W[0,\lambda]}{\partial\lambda}\Big|_{\lambda=1}+\sum_{i=1}^{n}c_i\frac{\partial W[\xi,1]}{\partial\xi_i}\Big|_{\xi=0},\hspace{4mm}i=1,\cdots,n.$$
On the other hand, owing to \eqref{lambdaepsilon} and \eqref{question} entails that
\begin{equation*}
\eta_{j}(0)=\tilde{c}_{n,\mu}\big\|\phi_{j}\big\|^{-1}_{L^{q^{\prime}}(\Omega_{j})}\tilde{H}(x_0,x_0)+o(1),\quad\nabla\eta_{j}(0)=o(1).
\end{equation*}
Therefore, we infer that
$$0=\omega_0(0)=c_0\frac{\partial W[0,\lambda](0)} {\partial\lambda}\Big|_{\lambda=1}+\sum_{i=1}^{n}c_i\frac{\partial W[\xi,1](0)}{\partial\xi_i}\Big|_{\xi=0}=\frac{n-2}{2}c_0,$$
and
$$0=\nabla \omega_{0}(0)=c_0\nabla\frac{\partial W[0,\lambda](0)} {\partial\lambda}\Big|_{\lambda=1}+\sum_{i=1}^{n}c_i\nabla\frac{\partial W[\xi,1](0)}{\partial\xi_i}\Big|_{\xi=0},$$
which gives us $c_0=0$ and $c_1,\cdots,c_n=0$ because $\nabla\partial W_{\xi_1}[\xi,1](0)\big|_{\xi=0},\cdots,\nabla\partial W_{\xi_n}[\xi,1](0)\big|_{\xi=0}$ is linearly independent and $\nabla\partial W_{\lambda}[0,\lambda](0)\big|_{\lambda=1}=0$. Thus, the claim follows, and so that $\eta_j\rightarrow 0$ weakly in $W^{2,q^{\prime}}(\mathbb{R}^n)$ for every $q^{\prime}>n$.

Now we are position to prove \eqref{contradictiion}. Owing to Lemma \ref{regular000},
\begin{equation}\label{rq94}
\begin{split}
\big\|\eta_{j}(y)\big\|_{W^{2,q^{\prime}}}&\leq C\big\|f_{\epsilon}(\phi_{\epsilon})\big\|_{L^{q^{\prime}}}\big\|\phi_{j}(y)\big\|^{-1}_{L^{q^{\prime}}}
+C\big\|f_{j}(\phi_{\epsilon})\big\|_{L^{r^{\prime}}}\big\|\phi_{j}(y)\big\|^{-1}_{L^{r^{\prime}}}\\&
+C\Big\|\Big(|y|^{-(n-2)}\ast W^{p-1}\eta_{j}\Big)
W^{p-1}+\Big(|y|^{-(n-2)}\ast W^{p}\Big)W^{p-2}\eta_{j}\Big\|_{L^{q^{\prime}}}\\&
+C\Big\|\Big(|y|^{-(n-2)}\ast W^{p-1}\eta_{j}\Big)
W^{p-1}+\Big(|y|^{-(n-2)}\ast W^{p}\Big)W^{p-2}\eta_{j}\Big\|_{L^{r^{\prime}}}\\&:=P_1+P_2+P_3+P_4.
\end{split}
\end{equation}
Similarly to the previous discussion in \eqref{ittarq}, we apply Hardy-Littlewood-Sobolev inequality, H\"{o}lder inequality and dominated convergence theorem to estimate each term on the right-hand side of the above inequality. Recalling \eqref{92qq}-\eqref{9934qq}, since $\frac{2n}{(2q^{\prime}+n)r}<1$, we obtain
\begin{equation}\label{lqqq}
\begin{split}
P_1&\leq C\Big(1+\big\|\phi_{j}\big\|^{\frac{2n}{(2q^{\prime}+n)r}}_{L^{q^{\prime}}(\Omega_{j})}
\big\|W^{p-2-\epsilon}|v_{\epsilon_j}-W|\big\|_{L^{q^{\prime}}(\Omega_{j})}\Big)\big\|\phi_{j}\big\|^{-1}_{L^{q^{\prime}}(\Omega_{j})}=o(1)
\end{split}
\end{equation}
as $j\rightarrow\infty$. Similarly to \eqref{933qq}-\eqref{9934qq}, we have
\begin{equation}\label{lq00}
\begin{split}
P_2&\leq C\Big(1+\big\|\phi_{j}\big\|^{\frac{2n}{(2q^{\prime}+n)r}}_{L^{r^{\prime}}(\Omega_{j})}
\big\|W^{p-2-\epsilon}|v_{\epsilon_j}-W|\big\|_{L^{r^{\prime}}(\Omega_{j})}\Big)\big\|\phi_{j}\big\|^{-1}_{L^{r^{\prime}}(\Omega_{j})}=o(1)
\end{split}
\end{equation}
as $j\rightarrow\infty$. Next we ready to estimate $P_3$ and $P_4$. We obtain that
\begin{equation*}
\begin{split}
\Big\|&\Big(|y|^{-(n-2)}\ast W^{p-1}\eta_{j}\Big)
W^{p-1}\Big\|_{L^{r^{\prime}}}\leq C\Big\||y|^{-(n-2)}\ast W^{p-1}\eta_{j}\Big\|_{L^{2r^{\prime}}}\Big\|W^{p-1}\Big\|_{L^{2r^{\prime}}}\\&
\leq C\Big(\big(\int_{\Omega_j\cap B_{R}(0)}|W^{p-1}(y)\eta_{j}(y)|^{r}dy\big)^{1/r}+\big(\int_{\Omega_j\cap B_{R}^{c}(0)}|W^{p-1}(y)\eta_{j}(y)|^{r}dy\big)^{1/r}\Big),
\end{split}
\end{equation*}
where $R>$ will be fixed subsequently and $r=\frac{2nr^{\prime}}{n+4r^{\prime}}$. Writing the H\"{o}lder conjugate of $\frac{q^{\prime}}{r}$ as $\zeta$, we find
$$
\big(\int_{\Omega_j\cap B_{R}(0)}|W^{p-1}(y)\eta_{j}(y)|^{r}dy\big)^{1/r}\leq C \big(\int_{\mathbb{R}^n}|W^{p-1}|^{\zeta r}dy\big)^{1/\zeta r}\big(\int_{B_{R}(0)}|\eta_{j}(y)|^{q^{\prime}}dy\big)^{1/q^{\prime}}=o(1),
$$
by $\eta_j\rightarrow 0$ strongly in $W^{2,q^{\prime}}(B_{R}(0))$ for any fixed $R$.
Furthermore, for $\epsilon>0$, we can choose $R>0$ such that
$$\big(\int_{\mathbb{R}^n\setminus B_{R}(0)}|W^{p-1}(y)|^{\zeta r}dy\big)^{1/(\zeta r)}\leq\epsilon.$$
Then we have
\begin{equation*}
\begin{split}
\big(\int_{\Omega_j\cap B_{R}^{c}(0)}&|W^{p-1}(y)\eta_{j}(y)|^{r}dy\big)^{1/r}\\&\leq C\big(\int_{\mathbb{R}^n\setminus B_{R}(0)}|W^{p-1}(y)|^{\zeta r}dy\big)^{1/(\zeta r)}\big(\int_{\Omega_j}|\eta_{j}(y)|^{q^{\prime}}dy\big)^{1/q^{\prime}}=o(1).
\end{split}
\end{equation*}
Similarly, we also have
\begin{equation*}
\begin{split}
\Big\|\Big(|y|^{-(n-2)}\ast W^{p}\Big)W^{p-2}\eta_{j}\Big\|_{L^{r^{\prime}}}
=o(1)\quad\mbox{as}\quad j\rightarrow\infty.
\end{split}
\end{equation*}
Combining the previous estimates together, we get
\begin{equation}\label{qprime}
\begin{split}
\Big\|\Big(|y|^{-(n-2)}\ast W^{p-1}\eta_{j}\Big)
W^{p-1}+\Big(|y|^{-(n-2)}\ast W^{p}\Big)W^{p-2}\eta_{j}\Big\|_{L^{r^{\prime}}}=o(1)
\end{split}
\end{equation}
as $j\rightarrow\infty$. Analogously, we also get
\begin{equation}\label{qp}
\begin{split}
C\Big\|\Big(|y|^{-(n-2)}\ast W^{p-1}\eta_{j}\Big)
W^{p-1}+\Big(|y|^{-(n-2)}\ast W^{p}\Big)W^{p-2}\eta_{j}\Big\|_{L^{q^{\prime}}}=o(1)
\end{split}
\end{equation}
as $j\rightarrow\infty$. From \eqref{Rrnn}, \eqref{rq94}, \eqref{lqqq}, \eqref{lq00}, \eqref{qprime} and \eqref{qp} we conclude that
$$\big\|\eta_{j}(y)\big\|_{W^{2,q^{\prime}}(\mathbb{R}^n)}=o(1)\quad\mbox{as}\quad j\rightarrow\infty,$$
a contradiction.
\end{proof}

We can now prove Proposition \ref{pro} by using Lemma \ref{upto} and Lemma \ref{regular000}.
\begin{proof}[Proof of Proposition \ref{pro}]
We shall prove that $\phi_0$ is a bounded solution of equation \eqref{fai0}. To show this fact, we first set
\begin{equation*}
\begin{split}
\Psi(y)=&-\kappa(n,x_0)\big(|y|^{-(n-2)}\ast W^{p}\big)W^{p-1}\log W-\kappa(n,x_0)\big(|y|^{-(n-2)}\ast W^{p}\big)W^{p-1}\log W\\&-(p-1)\tilde{c}_{n,\mu}\big(|y|^{-(n-2)}\ast W^p\big)W^{p-2}\tilde{H}(x_0,x_0)\\&
-p\tilde{c}_{n,\mu}\big(|y|^{-(n-2)}\ast W^{p-1}\tilde{H}(x_0,x_0)\big)W^{p-1}.
\end{split}
\end{equation*}
Then we only need to show that $f_{j}\rightarrow\Psi(y)$ in $L^{\infty}(\mathbb{R}^n)$ as $j\rightarrow\infty$. It is immediate to have that
\begin{equation*}
\begin{split}
&\big|\lambda_{j}^{-(n-2)}(f_{j}-\Psi(y))\big|
\leq\Big|\big(|y|^{-(n-2)}\ast v_{\epsilon_j}^{p-\epsilon_j}\big)v_{\epsilon}^{p-1-\epsilon_j}-\big(|x|^{-(n-2)}\ast W^{p-\epsilon_j}\big)W^{p-1-\epsilon_j}
\\&~~~-\lambda_{j}^{-(n-2)}\Big[(p-\epsilon_j)
\big(|y|^{-(n-2)}\ast W^{p-1-\epsilon_j}\big(\phi_{\epsilon_j}-\tilde{c}_{n,\mu}\tilde{H}(x_0,x_0+\lambda_{j}^{-1}z)\big)\big)
W^{p-1-\epsilon_j}\\&~~~+(p-1-\epsilon_j)
\big(|y|^{-(n-2)}\ast W^{p-\epsilon_j}\big)
W^{p-2-\epsilon_j}\big(\phi_{\epsilon_j}-\tilde{c}_{n,\mu}\tilde{H}(x_0,x_0+\lambda_{j}^{-1}z)\big)\Big]\Big|\\&~~~+\Big|\big(|x|^{-(n-2)}\ast W^{p-\epsilon_j}\big)W^{p-1-\epsilon_j}-\big(|x|^{-(n-2)}\ast W^{p}\big)W^{p-1}\\&~~~-\kappa(n,x_0)\lambda_{j}^{-(n-2)}\big[\big(|y|^{-(n-2)}\ast W^{p}\big)W^{p-1}\log W-\big(|y|^{-(n-2)}\ast W^{p}\log W\big)W^{p-1}\big]\Big|
\\&~~~+\lambda_{j}^{-(n-2)}\Big|p
\big(|y|^{-(n-2)}\ast W^{p-1}\phi_{j}\big)
W^{p-1}-p\tilde{c}_{n,\mu}\big(|y|^{-(n-2)}\ast W^{p-1}\tilde{H}(x_0,x_0)\big)W^{p-1}\\&~~~-(p-\epsilon_j)
\big[|y|^{-(n-2)}\ast W^{p-1-\epsilon_j}\big(\phi_{j}-\tilde{c}_{n,\mu}\tilde{H}(x_0,x_0+\lambda_{j}^{-1}z)\big)\big]
W^{p-1-\epsilon_j}\Big|\\&~~~+\lambda_{j}^{-(n-2)}\Big|(p-1)
\big(|y|^{-(n-2)}\ast W^{p}\big)
W^{p-2}\phi_{j}-(p-1)\tilde{c}_{n,\mu}\big(|y|^{-(n-2)}\ast W^p\big)W^{p-2}\tilde{H}(x_0,x_0)\\&~~~-(p-1-\epsilon_j)
\big(|y|^{-(n-2)}\ast W^{p-\epsilon_j}\big)
W^{p-2-\epsilon_j}\big(\phi_{j}-\tilde{c}_{n,\mu}\tilde{H}(x_0,x_0+\lambda_{j}^{-1}z)\big)\Big|\\&
:=\mathcal{E}_1+\mathcal{E}_2+\mathcal{E}_3+\mathcal{E}_4.
\end{split}
\end{equation*}
Arguing as in Lemma \ref{usefulestimate} for $J_1$, we get that
 \begin{equation*}
 \begin{split}
\|\mathcal{E}_1\|_{L^{\infty}}\leq C\|\lambda_{j}^{-2(n-2)}
\big[|y|^{-(n-2)}\ast W^{p-1-\epsilon_j}|\phi_{j}-\tilde{c}_{n,\mu}\tilde{H}|\big]
W^{p-2-\epsilon_j}|\phi_{j}-\tilde{c}_{n,\mu}\tilde{H})|\|_{L^{\infty}}=o(1).
\end{split}
\end{equation*}
Analogously to $J_2$, we deduce that
\begin{equation*}
\begin{split}
\|\mathcal{E}_2\|_{L^{\infty}}\leq& o(1)\|\lambda_{j}^{-(n-2)}\big(|y|^{-(n-2)}\ast W^{p}\big)
W^{p-1-\epsilon_j}|\log W|^2+\big(|y|^{-(n-2)}\ast W^{p-\varepsilon_j}|\log W|^2\big)W^{p-1-\epsilon_j}\|_{L^{\infty}}\\&+o(1)\|\big(|y|^{-(n-2)}\ast W^{p}|\log W|\big)W^{p-1-\epsilon}|\log W|\|_{L^{\infty}}=o(1).
\end{split}
\end{equation*}
Recalling $|\tilde{H}(x_0,x_0+\lambda_{j}^{-1}y)-\tilde{H}(x_0,x_0)|\leq C\lambda_{j}^{-1}y$ and $|\phi_j|_{L^{\infty}(\mathbb{R}^n)}\leq C$, we deduce that
\begin{equation*}
\|\mathcal{E}_3\|_{L^{\infty}}\leq C\lambda_{j}^{-(n-1)}\|
\big(|y|^{-(n-2)}\ast W^{p-1}|z|\big)W^{p-1}\|_{L^{\infty}}=o(1),
\end{equation*}
and similarly, we get that
\begin{equation*}
\|\mathcal{E}_4\|_{L^{\infty}}\leq C\lambda_{j}^{-(n-1)}\|
\big(|y|^{-(n-2)}\ast W^{p}\big)W^{p-2}|y|\|_{L^{\infty}}=o(1).
\end{equation*}
Therefore, we obtain $\phi_0$ is a weak solution $W^{2,q^{\prime}}$ and hence a classical solution of equation \eqref{fai0}.

Furthermore, we can consider as a test function in \eqref{fepusilong} and \eqref{fai0} a function of the form $z_j=\tilde{\xi}(\lambda_{j}^{-1}y)\phi_0(y)$ with the cut-off function $\tilde{\xi}$ equal to $1$ in $B(x_0,M)$ and vanishing outside $B(x_0,3M)$. Then, a simple computation gives that
\begin{equation*}
\begin{split}
-&\Delta(\phi_j-z_j)=p\Big[\Big(|y|^{-(n-2)}\ast W^{p-1}\phi_{j}\Big)
W^{p-1}(y)-\Big(|y|^{-(n-2)}\ast W^{p-1}\phi_{0}\Big)
W^{p-1}(y)\tilde{\xi}(\lambda_{j}^{-1}y)\Big]\\&+(p-1)\Big(|y|^{-(n-2)}\ast W^{p}\Big)W^{p-2}(\phi_{j}-z_j)+\big[f_{j}(\phi_{j})-\Psi(y)\big]\\&
-(1-\tilde{\xi}(\lambda_{j}^{-1}y))\kappa(n,x_0)\Big[\big(|y|^{-(n-2)}\ast W^{p}\big)W^{p-1}\log W+\big(|y|^{-(n-2)}\ast W^{p}\log W\big)W^{p-1}\Big]\\&-(1-\tilde{\xi}(\lambda_{j}^{-1}y))\tilde{c}_{n,\mu}\tilde{H}(x_0,x_0)\Big[(p-1)\big(|y|^{-(n-2)}\ast W^p\big)W^{p-2}+p\big(|y|^{-(n-2)}\ast W^{p-1}\big)W^{p-1}\Big]\\&
+\big[\Delta_{y}\tilde{\xi}\phi_0+2\nabla_{y}\tilde{\xi}\nabla\phi_0\big]:=\mathcal{I}_1+\mathcal{I}_2+\mathcal{I}_3+\mathcal{I}_4+\mathcal{I}_5+\mathcal{I}_6.
\end{split}
\end{equation*}
We aim to show that
\begin{equation}\label{faizjBM}
\phi_j\rightarrow z_j\quad\mbox{in}\quad L^{\infty}(B(0,M\lambda_{\epsilon})).
\end{equation}
In view of Sobolev imbedding theorem, to conclude \eqref{faizjBM}, it is sufficient to verify that as $j\rightarrow\infty$, there holds
\begin{equation}\label{fai-zj}
\big\|\phi_j-z_j\big\|_{W^{2,q^{\prime}}(\Omega_{j})}=o(1).
\end{equation}
Exploiting Lemma \ref{regular000},  we next have to estimate each term of $\mathcal{I}_1-\mathcal{I}_6$ in $L^{q^{\prime}}$-norm and $L^{r^{\prime}}$-norm on $\Omega_j$, respectively. For $\mathcal{I}_3$, we have using the above conclusion of the proof, $\|f_j-\Psi(y)\|_{L^{q^{\prime}}}+\|f_j-\Psi(y)\|_{L^{r^{\prime}}}=o(1)$.
For $\mathcal{I}_2$, similar to the argument of \eqref{qprime}-\eqref{qp}, we deduce that
\begin{equation*}
\|\mathcal{I}_2\|_{L^{q^{\prime}}}
=o(1),\quad
\|\mathcal{I}_2\|_{L^{r^{\prime}}}
=o(1),
\end{equation*}
as $j\rightarrow\infty$ and $R\rightarrow\infty$. We similarly compute and gives that
\begin{equation*}
\begin{split}
\|\mathcal{I}_1\|_{L^{q^{\prime}}}&=p\Big(|y|^{-(n-2)}\ast W^{p-1}(\phi_{j}-\phi_0)\Big)
W^{p-1}(y)\\&~~~+p(1-\tilde{\xi}(\lambda_{j}^{-1}y))\Big(|y|^{-(n-2)}\ast W^{p-1}\phi_{0}\Big)
W^{p-1}(y)\\&
\leq 
o(1)+C\lambda_{\epsilon}^{\frac{n}{2q^{\prime}}}\big[\int_{\Omega_j}W^{2(p-1)q^{\prime}}(\lambda_{\epsilon}y)dy\big]^{\frac{1}{2q^{\prime}}}\big[\int_{\Omega_j}W^{(p-1)r}(\lambda_{\epsilon}y)dy\big]^{\frac{1}{r}}\\&=o(1)+O(\lambda_{\epsilon}^{-(8-\frac{n}{2q^{\prime}})})=o(1),
\end{split}
\end{equation*}
where $r=\frac{2nq^{\prime}}{n+4q^{\prime}}$ and so we also have
\begin{equation*}
\|\mathcal{I}_1\|_{L^{r^{\prime}}}=O(\lambda_{\epsilon}^{-(8-\frac{n}{2r^{\prime}})})=o(1).
\end{equation*}
In same way, we deduce that
\begin{equation*}
\|\mathcal{I}_4\|_{L^{q^{\prime}}}=O(\lambda_{\epsilon}^{-(2n-\frac{n}{q^{\prime}})})=o(1),\quad\|\mathcal{I}_4\|_{L^{r^{\prime}}}=O(\lambda_{\epsilon}^{-(2n-\frac{n}{r^{\prime}})})=o(1),
\end{equation*}
and that
\begin{equation*}
\|\mathcal{I}_5\|_{L^{q^{\prime}}}=O(\lambda_{\epsilon}^{-(4-\frac{n}{q^{\prime}})})=o(1),\quad\|\mathcal{I}_5\|_{L^{r^{\prime}}}=O(\lambda_{\epsilon}^{-(4-\frac{n}{r^{\prime}})})=o(1).
\end{equation*}
To estimate $\mathcal{I}_6$, we write right hand side of \eqref{fai0} as $F(y)$. Then we consider the integral equation
$$\phi_0=\int_{\mathbb{R}^n}\frac{F(y)}{|x-y|^{n-2}}dy.$$
By virtue of the boundedness of $\phi_0$, recalling the right hand side of \eqref{fai0}, we have the following estimates hold:
\begin{equation*}
\int_{\mathbb{R}^n}\frac{1}{|x-y|^{n-2}}\frac{1}{(1+|y|)^{4}}dy\leq
\begin{cases}
\frac{C}{(1+|y|)^{2}}, \quad\quad\hspace{9.8mm}\mbox{ when } n=3,\\
\frac{C(1+\log|y|}{(1+|y|)^{2}}, \quad\quad\hspace{6mm}\mbox{ when } n=4,\\
\frac{C}{(1+|y|)^{3}}\big), \hspace{1mm}\quad\quad\hspace{7.5mm}\mbox{ when } n=5,
\end{cases}
\end{equation*}
and
\begin{equation*}
\quad\int_{\mathbb{R}^n}\frac{1}{|x-y|^{n-2}}\frac{1}{(1+|y|)^{2n}}dy\leq \frac{C}{(1+|y|)^{n-2}}.
\end{equation*}
Therefore, it is not hard to get that
$$|F(y)|\leq\frac{C}{|y|^{4}}\quad\mbox{for}\quad|y|\geq1.$$
By estimate of the Newtonian potential given in $F(y)$ (see \cite[lemma 2.3]{LINI}), then we get
$$|\phi_0|\leq\frac{C}{|y|^{4}}\quad\mbox{for}\quad|y|\geq1,$$
and so we similarly compute and gives that $|\nabla \phi_0|\leq\frac{C}{|y|^{4}}$ for $|y|\geq1$. Combining this bound, then we get
\begin{equation*}
\begin{split}
\|\mathcal{I}_6\|_{L^{q^{\prime}}}+\|\mathcal{I}_6\|_{L^{r^{\prime}}}&=O(\lambda_{\epsilon}^{-1}\|\phi_0\|_{W^{2,q^{\prime}}(\mathbb{R}^n)})+O(\lambda_{\epsilon}^{-(1-\frac{n}{r^{\prime}})}(\|\nabla \phi_{0}(\lambda_{\epsilon}^{-1}y)\|_{L^{r^{\prime}}}+\lambda_{\epsilon}^{-1}\|\phi_{0}(\lambda_{\epsilon}^{-1}y)\|_{L^{r^{\prime}}})\\&=o(1).
\end{split}
\end{equation*}
This, combined with the previous estimates together,  gives the conclusion.

Finally, there exist two sequences $\epsilon_j$ and $\tilde{\epsilon}_j$ such that for $j$ large, up to subsequence, there holds
\begin{equation*}
\phi_{\epsilon_j}\rightarrow \phi_0\quad\mbox{and}\quad\phi_{\tilde{\epsilon}_j}\rightarrow \tilde{\phi}_0\quad\mbox{in}\quad L^{\infty}(B(0,M\lambda_{\epsilon})).
\end{equation*}
Then, similarly to the first claim in the proof of Lemma \ref{upto}, we can verify that $\phi_0-\tilde{\phi}_0=0$ holds. The conclusion of Proposition follows easily.
\end{proof}


\vspace{1cm}

{\bf Acknowledgments}:
The authors would like to thank the anonymous referee for his/her useful comments and suggestions which help to improve the presentation of the paper greatly.
\vspace{0.5cm}

\appendix

\section{Technical lemmata}
\begin{lem}[\cite{HANZCHAO}]\label{regular}
Let $u$ solve
\begin{equation*}
\begin{cases}
-\Delta u=f\quad\mbox{in}\hspace{2mm}\Omega\subset\mathbb{R}^n,\\
u=0\quad\quad\hspace{2mm}\mbox{on}\hspace{2mm}\partial\Omega.
\end{cases}
\end{equation*}
Then
\begin{equation}\label{regu}
\|u\|_{W^{1,q}(\Omega)}+\|\nabla u\|_{C^{0,\alpha}(\omega^{\prime})}\leq C\big(\|f\|_{L^1(\Omega)}+\|f\|_{L^\infty(\omega)}\big)
\end{equation}
for $q<n/(n-1)$ and $\alpha\in(0,1)$. Here $\omega$ be a neighborhood of $\partial\Omega$ and $\omega^{\prime}\subset\omega$ is a strict subdomain of $\omega$.
\end{lem}
The following lemma in \cite{JW0} will play a crucial role in the proof of Proposition \ref{pro}.
\begin{lem}\label{regular000}
Let $u$ solve
\begin{equation*}
\begin{cases}
-\Delta u(y)=f(y)\quad\mbox{in}\hspace{2mm}\Omega_{\epsilon},\\
u=0\quad\quad\hspace{4mm}\quad \quad\mbox{on}\hspace{2mm}\partial\Omega_{\epsilon}.
\end{cases}
\end{equation*}
Then
\begin{equation}\label{regu11}
\|u\|_{W^{2,q^{\prime}}(\Omega_{\epsilon})}\leq C\big(\|f\|_{L^{r^{\prime}}(\Omega_{\epsilon})}+\|f\|_{L^{q^{\prime}}(\omega_{\epsilon})}\big)
\end{equation}
for $q^{\prime}>2$ and $1/{r^{\prime}}=1/{q^{\prime}}+2/n$.
\end{lem}

\begin{lem}[\cite{Rey-1990}]\label{Lem2.2}
	Assume that $a\in\Omega$ and $\lambda\in\mathbb{R}^{+}$, then we have the following property:
	\begin{equation*}	\psi[a,\lambda](x)=\frac{\tilde{c}_{n,\mu}}{\lambda^{\frac{n-2}{2}}}\tilde{H}(a,x)+f_{\lambda}\quad\mbox{with}\quad PW[a,\lambda](x)=W[a,\lambda](x)-\psi[a,\lambda](x),
\end{equation*}
where $\tilde{c}_{n,\mu}$ is defined in \eqref{fU} and $f_{\lambda}$ verifies the uniform estimates
$$f_{\lambda}=O\Big(\frac{1}{\lambda^{\frac{n+2}{2}}d^n}\Big),\quad \frac{\partial f_{\lambda}}{\partial x_i}=O\Big(\frac{1}{\lambda^{\frac{n+2}{2}}d^{n+1}}\Big).$$
\end{lem}

\end{document}